\documentclass[11pt]{article}
\usepackage[totalwidth=13.0cm,totalheight=20.0cm]{geometry}
\usepackage{latexsym,amsthm,amsmath,amssymb,url}
\usepackage[ruled, linesnumbered]{algorithm2e}

\usepackage{tikz,authblk}
\usepackage{comment}
\usetikzlibrary{decorations.pathreplacing}

\newcommand{\AY}[1]{{#1}}
\newcommand{\GG}[1]{{#1}}

\newcommand{\2}{\vspace{0.15cm}}

\newcommand{\smallQED}{\hfill {\tiny ($\Box$)}}

\newcommand{\YC}[1]{{#1}}

\newtheorem{theorem}{Theorem}[section]
\newtheorem{lemma}[theorem]{Lemma}
\newtheorem{corollary}[theorem]{Corollary}

\newtheorem{remark1}[theorem]{Remark}

\newtheorem{conjecture}[theorem]{Conjecture}
\newtheorem{claim}[theorem]{Claim}

\newtheorem{observation}[theorem]{Observation}

\begin{document}
	
	\title{Lower Bounds for Maximum Weight Bisections of Graphs with Bounded Degrees}
	\author{
	Stefanie Gerke\thanks{Department of Mathematics. Royal Holloway University of London.  {\tt stefanie.gerke@rhul.ac.uk}.} \hspace{2mm} Gregory Gutin\thanks{Department of Computer Science. Royal Holloway University of London. {\tt g.gutin@rhul.ac.uk}.} \hspace{2mm} Anders Yeo\thanks {Department of Mathematics and Computer Science, University of Southern Denmark. {\tt andersyeo@gmail.com}, and Department of Mathematics, University of Johannesburg, Auckland Park, 2006 South Africa.} \hspace{2mm} Yacong Zhou\thanks{Department of Computer Science. Royal Holloway University of London. {\tt Yacong.Zhou.2021@live.rhul.ac.uk}.} }
	\date{}
	\maketitle
	\begin{abstract}
A bisection in a graph is a cut in which the number of vertices in the two parts differ by at most 1.
In this paper, we give lower bounds for the maximum weight of bisections of edge-weighted graphs with bounded maximum degree. Our results improve a bound of Lee, Loh, and Sudakov (J. Comb. Th. Ser. B 103 (2013)) for (unweighted) maximum bisections in graphs whose maximum degree is either even or equals 3, and for almost all graphs. We show that a tight lower bound for maximum size of bisections in 3-regular graphs obtained by Bollob\'as and Scott (J. Graph Th. 46 (2004)) can be extended to weighted subcubic graphs. We also consider edge-weighted triangle-free subcubic graphs and show that a much better lower bound (than for edge-weighted subcubic graphs) holds for such graphs especially if we exclude $K_{1,3}$. We pose three conjectures. 
	\end{abstract}

	\section{Introduction}
In this paper, we consider {\em edge-weighted} (or, {\em weighted}) graphs $G=(V,E,w)$, where the weight $w(e)$ of each edge $e\in E$ is a non-negative real number. Given a graph $G=(V,E)$ a \emph{cut} of $G$ is the edge set between a non-empty set $A\not= V$ and $V\setminus A$. The sets $A$ and $V\setminus A$ are {\em partition classes} (or, {\em parts}). 
A {\em bisection} is a cut in which the sizes of its parts differ by at most 1.

Several researchers \cite{AGN1993,LZ1982,Locke1982} have shown that if the chromatic number of $G$ is $\chi$, then it admits a bipartition of size at least $\frac{\chi+1}{2\chi}|E(G)|$ when $\chi$ is odd and $\frac{\chi}{2(\chi-1)}|E(G)|$ when $\chi$ is even. Thus, it is implied by Brook's Theorem that if the maximum degree $\Delta(G)$ of a graph $G$ is at most $k$, then for odd $k$, every graph with maximum degree at most $k$ has a cut with at least $\frac{k+1}{2k}|E(G)|$ edges, and for even $k$, every graph with maximum degree at most $k$ has a cut with $\frac{k+2}{2(k+1)}|E(G)|$ edges. 
 
 These bounds are tight as the maximum cuts of complete graphs have this size. Note that the bound in the odd case implies the bound in the even case as every graph with maximum degree at most $k$ also has maximum degree at most $k+1$.

This result can be easily generalized to edge-weighted graphs $G$ 
by replacing the number of edges  $|E(G)|$ with the total weight  $w(G)$ of the edges. We conjecture that the same bound holds not only for a maximum weight cut but also for a maximum weight  bisection.
	\begin{conjecture}\label{cj:1}
		Let $k$ be a positive integer and let $G$ be a graph with maximum degree $\Delta(G)\leq k$. If $k$ is odd  then $G$ has a bisection of weight at least $\frac{k+1}{2k}w(G)$, and if $k$ is even then $G$ has a bisection of weight at least $\frac{k+2}{2(k+1)}w(G)$. 
	\end{conjecture}
	As far as we know there are no publications on  bounds for the maximum  bisections of weighted graphs. However, there are some results related to our conjecture for the unweighted case. In \cite{LLS}, Lee, Loh, and Sudakov showed the following result by using equitable colourings. 
	\begin{theorem} \cite{LLS}
		 If $G$ has maximum degree at most $k$, then there exists a bisection of size at least $\frac{k+1}{2k}|E(G)|-\frac{k(k+1)}{4}$
		if $k$ is odd and $\frac{k+2}{2(k+1)}|E(G)|-\frac{k(k+2)}{4}$ if  $k$ is even.
	\end{theorem}
	One can generalize this result easily to weighted graphs using an argument similar to that in \cite{LLS}. In Section \ref{sec:basic}, we will improve the lower bound of the above result by showing that Conjecture \ref{cj:1} holds when $k$ is even.   It turns out that the even case of Conjecture \ref{cj:1} is relatively easy, but the odd case is quite hard. In Section \ref{sec:basic}, we prove a lemma modifying a generic lower bound for maximum weighted cuts obtained in \cite{GY} to maximum weighted bisections. This lemma and Vizing's theorem (see \cite{SSTF,Vizing}) are used in proofs of various results of this paper. 

Bollob\'as and Scott \cite{BS2004} showed the following result, which means Conjecture \ref{cj:1} holds for unweighted regular graphs.
   \begin{theorem}\label{thm:BS} \cite{BS2004}
   	Let $k$ be an odd positive integer. If $G$ is $k$-regular, then there exists a bisection of size at least $\frac{k+1}{2k}|E(G)|$.
   \end{theorem}
 A 3-regular graph is called {\em cubic} and a graph of maximum degree at most 3 {\em subcubic}.
   In Section \ref{sec:subcubic}, we will show that Conjeture \ref{cj:1} holds for weighted subcubic graphs, which generalizes the cubic case of Theorem \ref{thm:BS}. We will also show that the conjecture holds for almost all graphs. 
	
Due to the above-mentioned lower bound by  Bollob\'as and Scott \cite{BS2002}, $\frac{2}{3}|E(G)|$ is a tight lower bound for the maximum size of a cut in a subcubic graph. However, Bondy and Locke \cite{BL1986} showed that this bound can be improved for triangle-free subcubic graphs as follows.
\begin{theorem}\label{thm:BL}\cite{BL1986}
		Every triangle-free subcubic graph $G=(V,E)$ has a cut with weight at least $\frac{4}{5}|E(G)|$. 
	\end{theorem}
	It was conjectured by Gutin and Yeo \cite{GY} that the same proportion ($\frac{4}{5}$) is also true for cuts in weighted graphs, and one of the main results in \cite{GY} is the following:
	
	\begin{theorem}\label{thm:GY}\cite{GY}
		Every triangle-free subcubic weighted graph $G=(V,E,w)$ has a cut with weight at least $\frac{8}{11}w(G)$. 
	\end{theorem}
	Note that the bound of Theorem \ref{thm:BL} cannot be extended to bisections in triangle-free subcubic graphs simply because the maximum bisection of $K_{1,3}$ has two edges. 
Moreover, $K_{1,3}$ shows that the best bound can be at most $(2/3) w(G)$.
However, it turns out that $K_{1,3}$ is the only triangle-free subcubic graph with a maximum bisection of size $(2/3) w(G)$.
	
	In Section \ref{sec:triangle-free}, using extensively probabilistic arguments we show a better lower bound for a maximum weight bisection of weighted triangle-free subcubic graphs, which are not a $K_{1,3}$. Indeed, we show that every triangle-free subcubic graph $G$, which is not a $K_{1,3}$, has a bisection with weight at least $({613}/{855})w(G)\approx 0.716959w(G)$, and we believe that this bound is not tight. Since the maximum bisection of \AY{the} Petersen graph has at most 11 edges, $(11/15)w(G)\approx0.733333w(G)$ may be the best possible lower bound. Thus, we conjecture the following:	
	\begin{conjecture}\label{conj:triangle-free cubic}
		Every weighted triangle-free subcubic graph $G=(V,E,w)$ has a bisection with weight at least $\frac{11}{15}w(G)$ unless $G\cong K_{1,3}$.
	\end{conjecture}
For disjoint $A, B\subseteq V(G)$, let $G[A]$ denote the subgraph induced by $A$ and $G[A, B]$ denote the subgraph induced by the edges between $A$ and $B$. 	
	
	\section{Conjecture \ref{cj:1} for even $k$ and almost all graphs}\label{sec:basic}

	Let $\mathcal{B}(G)$ denote the set of bipartite subgraphs \AY{of} $G$ such that every
	connected component of it is an induced subgraph of $G$. Gutin and Yeo \cite{GY} showed the following:
	
	\begin{lemma}\label{lem:1}\cite{GY}
		Let $G=(V,E,w)$ be a weighted graph and $B\in \mathcal{B}(G)$. Then, $G$ has a cut with weight at least $\frac{w(G)+w(B)}{2}$.
	\end{lemma}
	We say that $B\in \mathcal{B}_b(G)$, if $B$ is the union of vertex-disjoint bipartite subgraphs $B_i$'s (not necessarily connected) of $G$ with partite sets $(X_i, Y_i)$ where $G[X_i]$ and $G[Y_i]$ have no edges and $|X_i|=|Y_i|$. The following lemma can be proved using a similar argument to the one in the proof of Lemma \ref{lem:1}. For the sake of completeness, we provide a proof here.
	\begin{lemma}\label{lem:bbsg}
		Let $G=(V,E,w)$ be a weighted graph and $B\in \mathcal{B}_b(G)$. Then, $G$ has a bisection with weight at least $\frac{w(G)+w(B)}{2}$.
	\end{lemma}

	\begin{proof}
		Without loss of generality, we may assume that $|V(G)\setminus V(B)|\leq 1$. For, if there are at least two vertices $\{v,w\}\in V(G)\setminus V(B)$, then we can add the bipartite subgraph $B'$ with partite sets  $\{v\},\{w\}$ to $B$ and denote the new subgraph again by $B$. Clearly, $B$ is still in $\mathcal{B}_b(G)$ and the weight of $B$ does not decrease. Therefore, we only need to prove the bound for this new subgraph $B$.
		
		Suppose that $B=\cup_{i=1}^t B_i$, \AY{where $B_i$ is a bipartite subgraph with bipartition $(X_i,Y_i)$}.  For each $i\in [t]$, we color the vertices in $X_i$ red with probability $1/2$ and blue otherwise, independently of the colors of the other sets $X_j$. If $X_i$ is coloured red then $Y_i$ is coloured blue, and if $X_i$ is coloured blue then $Y_i$ is coloured red. If there exists a vertex in $V(G)\setminus V(B)$, then we colour \AY{it red}  with probability $1/2$ and blue otherwise. Let $X$ be the set of red vertices and let $Y$ be the set of blue vertices. Consider the bisection $(X, Y)$. In fact, $(X,Y)$ is always a bisection since $|X_i|=|Y_i|$ for all $i\in [t]$ and $|V(G)\setminus V(B)|\leq 1$. Observe that each edge in $B$ will be in this bisection with probability 1. All other edges will be in this bisection with probability at least 1/2. Thus, by the linearity of expectation, we have
		\[\mathbb{E}(w(X,Y))\geq\frac{w(G)-w(B)}{2}+w(B)=\frac{w(G)+w(B)}{2},\]
		which completes the proof.
	\end{proof}
	Since any matching $M$ of $G$ is clearly in $\mathcal{B}_b(G)$, we have the following corollary.
	
		\begin{corollary}\label{cor:mwmatching}
		Let $G=(V,E,w)$ be a weighted graph and $M$ its maximum weight matching. Then, $G$ has a bisection with weight at least $\frac{w(G)+w(M)}{2}$.
	\end{corollary}

We will need to use the following bound for chromatic index known as Vizing's Theorem.

\begin{theorem}\label{Vizing}\cite{Vizing}
For any simple graph $G$, $\Delta(G)\leq\chi'(G)\leq \Delta(G)+1$.
\end{theorem}

From Vizing's Theorem, every graph \AY{is classified into one of} two classes. A graph $G$ is in {\em Class 1} if $\chi'(G)=\Delta(G)$ and in {\em Class 2} if $\chi'(G)=\Delta(G)+1$. Now, we prove the following theorem which confirms Conjecture \ref{cj:1} for graphs in Class 1.
\begin{theorem}\label{thm:chi}
	Every weighted graph $G=(V,E,w)$ has a bisection with weight at least $\frac{\chi'(G)+1}{2\chi'(G)}w(G)$. In particular, Conjecture \ref{cj:1} holds for all graphs in Class 1.
\end{theorem} 
\begin{proof}
	Since the edge set of $G$ can be partitioned into $\chi'(G)$ matching,  there is a matching $M$ with weight at least $\frac{w(G)}{\chi'(G)}$. By Corollary \ref{cor:mwmatching}, $G$ has a bisection with weight at least $\frac{w(G)+w(M)}{2}\geq \frac{\chi'(G)+1}{2\chi'(G)}w(G)$. Furthermore, if $G$ is in Class 1 i.e. $\chi'(G)=\Delta(G)$, then $G$ has a bisection with weight at least $ \frac{\Delta(G)+1}{2\Delta(G)}w(G)$.
\end{proof}

Erd\H{o}s and Wilson \cite{EW} showed that almost all graphs are in Class 1. This and Theorem \ref{thm:chi} imply the following:
\begin{theorem}
For almost all graphs $G=(V,E)$ and every weight function $w$ on $E$, the bound of Conjecture \ref{cj:1} holds.
\end{theorem}
 
Combining Theorem \ref{Vizing} with Theorem \ref{thm:chi}, we have the following corollary which confirms the even case of Conjecture \ref{cj:1}. 
\begin{corollary}\label{thm:even}
Let $k$ be a positive integer. Any graph $G$ with  $\Delta(G)\leq k$ has a bisection with weight at least 
\[ \frac{k+2}{2(k+1)}w(G).\] In particular, Conjecture \ref{cj:1} holds when $k$ is even.
\end{corollary}

\begin{figure}\label{fig:1}
	\centering
	\begin{tikzpicture}[every node/.style={circle, draw, minimum size=0.4cm}]
		\foreach \i in {1,2,...,4} {
			\node (a\i) at (0,\i) {$a_{\i}$};
		}
		\node (a9) at (2,5) {$v$};
		
		\foreach \i in {1,2,...,4} {
			\node (b\i) at (4,\i) {$b_{\i}$};
		}
		
		\foreach \i in {1,2,...,4} {
			\foreach \j in {1,2,...,4} {
				\draw (a\i) -- (b\j);
			}
		}
		\foreach \i in {1,2,...,4} {
			\draw (a9) -- (b\i);
		}
		\draw (a9) -- (a4);
	\end{tikzpicture}
	\caption{Balanced bipartite graph with an additional vertex}
	\label{fig:bipartite}
\end{figure}
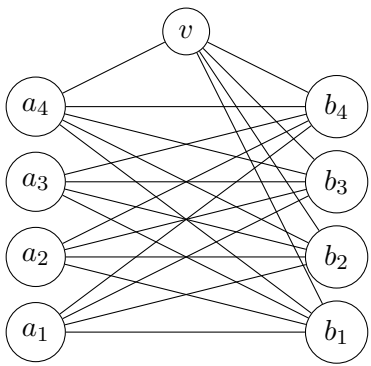
\begin{remark1} Consider the unweighted case of Conjecture \ref{cj:1}.
	By Theorem \ref{thm:chi}, we only have to consider the case of  $\chi'(G)=k+1$. One  example of such a graph is the Peterson graph $PG$. However, $PG$ has a perfect matching with $|E(PG)|/3$ edges. Thus, by applying Corollary \ref{cor:mwmatching}, we can also prove the conjecture for $PG$. Inspired by this fact, one may try to prove that if $k$ is odd and $G$ is a graph with  $\chi'(G)=k+1$, then there exists a matching with at least $|E(G)|/k$ edges. Unfortunately, this is not true. Assume $k=2t+1$ and consider the graph $G$ that is constructed from a complete bipartite graph with partite sets $(A, B)$ where $|A|=|B|=2t$ by adding a new vertex $v$, all edges between $v$ and $B$ and one edge between $A$ and $v$ (see Fig. \ref{fig:bipartite}). One can observe that $\Delta(G)=k$, $|E(G)|=k^2-k+1$ and that the maximum size of a matching in $G$ is $2t=k-1$. This implies $k+1\geq\chi'(G)\geq \frac{|E(G)|}{k-1}=k+\frac{1}{k-1}$ and therefore $\chi'(G)=k+1$. But as mentioned before every matching in $G$ has size at most $k-1=\frac{k-1}{k^2-k+1}|E(G)|<|E(G)|/k$.
\end{remark1}

\section{Conjecture \ref{cj:1} for weighted subcubic graphs} \label{sec:subcubic}
In this section, we prove the Conjecture \ref{cj:1} for weighted subcubic graphs i.e. when $k= 3$. 

For any integer $k>0$, a {\em $k$-bisection} of a graph $G$ is a bisection $(A,B)$ where every component in $G[A]\cup G[B]$ is a tree with at most $k$ vertices. Mattiolo and Mazzuoccolo \cite{MM} showed the following result.

\begin{lemma}\cite{MM}\label{lem:3-bisection}
	Every cubic multigraph has a 3-bisection.
\end{lemma}

\begin{lemma}\label{lem:forest}
		If $F$ is a forest with at most $|V(F)|/2$ edges, then there is a biparite subgraph $B\in \mathcal{B}_b(F)$ such that $E(B)=E(F)$.
\end{lemma}
\begin{proof}
	Let $c(F)$ be the number of connected components in $F$, $m_i$ the number of components with $i$ edges and $t$ the maximum size of a component in $F$. Since $F$ is a forest $|E(F)|=|V(F)|-c(F)\leq |V(F)|/2$,
which implies that $|V(F)|/2\leq c(F)=\sum_{i=0}^t m_i$. Thus,
\[\sum_{i=1}^t i \, m_i =|E(F)|\leq |V(F)|/2\leq \sum_{i=0}^t m_i,\]
which means that $\sum_{i=2}^tm_i(i-1)\leq m_0$. Note that for any $i\geq 2$, a tree with $i$ edges is a bipartite graph in which the sizes of two partite sets differ by at most $i-1$ vertices, and therefore, we need at most $\sum_{i=2}^tm_i(i-1)$ isolated vertices to balance the partite sets of every such tree. The inequality $\sum_{i=2}^tm_i(i-1)\leq m_0$ guarantees that we have enough isolated vertices to use. 
\end{proof}
\begin{lemma}\label{lem:cubic}
	Every weighted cubic multigraph $G$ has a bisection $(X,Y)$ such that the following holds.
		\begin{description}
		\item[(i)] $G[X]\cup G[Y]$ is a forest;
		\item[(ii)] $|(X,Y)|$ attains the maximum value among all bisection that satisfy (i);
		\item[(iii)] \label{lem:cubic3}  $\Delta(G[X])\leq 1$ and $|E(G[Y])|\leq |Y|/2$.
	\end{description}
\end{lemma}
\begin{proof}
By Lemma \ref{lem:3-bisection}, $G$ has a bisection  (not necessarily a 3-bisection) such  that both partition classes induce a forest.  Let $(X,Y)$ be a maximum bisection subject to the condition that  $G[X]\cup G[Y]$ is a forest. We will show that $(X,Y)$ also meets the last requirement. Assume without loss of the generality that $\Delta(G[X])\leq\Delta(G[Y])$.  We will first show that 
 \begin{equation} \label{cl:DGA=1} \Delta(G[X])\leq 1.\end{equation}  Suppose for a contradiction that $2\leq\Delta(G[X])\leq\Delta(G[Y])$. Then, there exist $v\in X$ and $w\in Y$ such that $d_{G[X]}(v)\geq 2$ and $d_{G[Y]}(w)\geq 2$. Since  both $v$ and $w$ have at most one neighbour in the other part, $G[X\cup\{w\}-\{v\}]$ and $G[Y\cup\{v\}-\{w\}]$ are still forests. However, $|(X\cup\{w\}-\{v\},Y\cup\{v\}-\{w\})|>|(X,Y)|$, a contradiction.
	  
	  By \eqref{cl:DGA=1}, every connected component of $G[X]$ is either an edge or a single vertex and therefore $|E(G[X])|\leq |X|/2$. Thus, since $G$ is 3-regular and $|X|=|Y|$ \AY{(as every graph has an even number of odd degree vertices)},  
	  
	  \begin{align*}
	  |E(G[Y])|& =\frac{3|Y|-|(X,Y)|}{2}=\frac{3|X|-|(X,Y)|}{2}\\
	  & =|E(G[X])|\leq |X|/2=|Y|/2,\end{align*}

	which completes the proof.
	  \end{proof}

	  \begin{theorem}
	  	Every weighted subcubic graph $G$ has a bisection with weight at least $\frac{2}{3}w(G)$.
	  \end{theorem}

	 \begin{proof}
	 	
We may assume that $G$ has at most one vertex with degree not equal to $3$ as we may add edges of weight $0$ between any two vertices with degree less than $3$.  We consider three cases. 	

{\bf Case 1:} $G$ is 3-regular.  Let $(X,Y)$ be a bisection as in Lemma \ref{lem:cubic}. If $w(X,Y)\geq \frac{2}{3}w(G)$ then the proof is complete. If not,  by Lemma \ref{lem:forest}, there is a bipartite subgraph $B_1\in\mathcal{B}_b(G[Y])$ that contains $E(G[Y])$. $B_1$ together with the maximum matching in $G[X]$ forms a bipartite subgraph $B_2\in\mathcal{B}_b(G)$ that contains all edges in $G[X]\cup G[Y]$. Now  \[w(B_2)=w(G[X])+w(G[Y])=w(G)-w(X,Y)>\frac{1}{3}w(G),\] 
	 	and the result follows by Lemma \ref{lem:bbsg}.
	 	
	 	{\bf Case 2:} $G$ has one vertex with degree one, say $d_{G}(z)=1$. Then, we add two vertices $x$ and $y$, two multiple edges $xy$ with weight 0 as well as edges $xz$ and $yz$ with weight 0. Denote the resulting graph by $H$. Clearly, $H$ is 3-regular. Let $(X,Y)$ be the bisection obtained by Lemma \ref{lem:cubic}. Note that $x$ and $y$ are in different parts since $H[X]\cup H[Y]$ is a forest (w.l.o.g. we may assume that $x\in X$ and $y\in Y$). If $w(X,Y)\geq\frac{2}{3}w(H)$, then  $(X\setminus\{x\}, Y\setminus\{y\})$ is a desired bisection in $G$. Thus, we may assume that $w(H[X\setminus\{x\}]\cup H[Y\setminus\{y\}])>\frac{1}{3}w(H)$ and consider the following two subcases.
	 	
	 	{\bf Subcase 2.1:} $z\in Y$. Then since $|E(H[Y])|\leq |Y|/2$ and $yz\in H[Y]$, $|E(H[Y\setminus\{y\}])|\leq |Y|/2-1$. Thus, by Lemma \ref{lem:forest} we can find a bipartite subgraph in $\mathcal{B}_b(H[Y\setminus\{y\}])$, which together with the maximum matching in $H[X\setminus\{x\}]$ forms a bipartite subgraph $B\in \mathcal{B}_b(H-\{x,y\})$ that contains all edges in $H[X\setminus\{x\}]\cup H[Y\setminus\{y\}]$. Therefore, by Lemma \ref{lem:bbsg} we have a bisection in $H-\{x,y\}$ with weight at least $\frac{2}{3}w(H)=\frac{2}{3}w(G)$. Clearly, this is also a desired bisection of $G$.
	 	
	 	{\bf Subcase 2.2:} $z\in X$. If $\Delta(H[Y])=1$ then there is a matching that contains all edges in $H[X\setminus\{x\}] \cup H[Y\setminus \{y\}]$ and  the result follows by Corollary \ref{cor:mwmatching}. It is impossible that  $\Delta(H[Y])=3$, since, if this were the case, the bisection obtained by  swapping the vertex with degree 3 in $H[Y]$ and $z\in X$  would yield a larger bisection the partition classes of which would form a forest. Thus, we may assume that there is a vertex $w$ with degree 2 in $H[Y]$. Let $w'$ be the neighbour of $w$ in $X$. Then, $d_{H[X]}(w')=0$, for otherwise swapping $w$ with $w'$ would give a larger  bisection  the partition classes of which would induce a forest. Thus, we can now swap $w$ and $z$, and the number of edges in the bisection will not decrease. In addition, it is easy to check that the new bisection $(X',Y')$ still satisfy (i)-(iii) in Lemma \ref{lem:cubic} and $z\in Y'$. Therefore, we can now proceed as in Subcase 2.1.
	 	
	 	{\bf Case 3:} If $G$ has a vertex with degree two, say $d_{G}(z)=2$. Then $|V(G)|$ is odd. We add a new vertex $z'$ and an edge $zz'$ with weight 0. Denote the resulting graph by $H'$. By Case 2, $H'$ has a bisection $(X,Y)$ with weight $\frac{2}{3}w(H')=\frac{2}{3}w(G)$. Therefore, $(X\setminus\{z'\}, Y\setminus\{z'\})$ is a desired bisection of $G$. 
	 \end{proof}
	 
\section{Triangle-free Subcubic Graphs}\label{sec:triangle-free}


In this section, we show a lower bound for a maximum weight bisection of a weighted triangle-free subcubic graph, which is not a $K_{1,3}$. Indeed, we show that every triangle-free subcubic graph $G$, which is not a $K_{1,3}$, has a bisection with weight at least $\frac{613}{855}w(G)\approx 0.716959w(G)$. Note that the maximum bisection of the Petersen graph has at most 11 edges. Therefore, $11/15w(G)\approx0.733333w(G)$ may be the best possible lower bound. 

\subsection{Bridgress triangle-free subcubic graphs}
In this subsection, we prove the above-mentioned lower bound for bridgeless triangle-free subcubic graphs. 

Theorem \ref{thm:petersen} is a slight generalization of Petersen's matching theorem \cite{Peterson}, which can be proved by applying Tutte's theorem \cite{Tutte}.
(This result  can also be found in Exercise 16.4.8 of Bondy and Murty's textbook on graph theory \cite{BM2008}).
\begin{theorem}\label{thm:petersen}
		Let $G$ be any bridgeless cubic graph, and let $e$ be any edge in $G$. Then, $G$ has a perfect matching containing $e$. 
\end{theorem}
We first show \AY{that} the following multigraph version of the above theorem holds.

\begin{lemma}\label{lem:mt-gvpt}
	Let $G$ be any bridgeless cubic multigraph, and let $e$ be any edge in $G$. Then, $G$ has a perfect matching containing $e$. 
\end{lemma}
\begin{proof}
We may assume that $G$ is connected. Then the result is trivial when there are three multiple edges between two vertices. Thus, we may assume that any pair of vertices in $V(G)$ has at most two multiple edges. For any two edges between vertices $x$ and $y$ in $G$, delete these two edges, add two new vertices $v_{xy}$ and $w_{xy}$, also edges $xv_{xy}$, $xw_{xy}$, $yv_{xy}$, $yw_{xy}$ and $v_{xy}w_{xy}$. Denote the resulting graph by $G'$. Observe that $G'$ is a bridgeless cubic graph. Note that if there are two edges between $x$ and $y$, then any perfect matching in $G'$ contains $\{xv_{xy},yw_{xy}\}$, $\{xw_{xy},yv_{xy}\}$ or $v_{xy}w_{xy}$. Thus, from any perfect matching $M'$ of $G'$, we can construct a perfect matching $M$ of $G$ by the following steps. 
\begin{itemize}
	\item If there are two edges between $x$ and $y$ in $G$, then we add one of them to $M$ if and only if $\{xv_{xy},yw_{xy}\}$ or $\{xw_{xy},yv_{xy}\}$ is contained in $M'$.
	 
	\item If there is only one edge between $x$ and $y$ in $G$, we add it to $M$ if and only if it is in $M'$. 
\end{itemize}

Observe that $M$ is a perfect matching of $G$. 
Thus, we are done by Theorem \ref{thm:petersen} \AY{(as we can force any edge in $G$ to belong to the matching by forcing the relevant
edge of $G'$ to be in $M'$).} 
\end{proof}

We will need to use the following lemma. The basic ideas of it come from Lemma 5.11 in \cite{GY}. For any matching $M$ of a graph $G$, let $G/M$ denote the graph obtained from $G$ by contracting each edge in $M$.
\begin{lemma}\label{lem:triangle-free cubic}
	Let $G=(V, E,w)$ be a triangle-free graph with a perfect matching $M$. If $\chi'(G/M)\leq k$, then there exists a bipartite subgraph $B_C\in \mathcal{B}_b(G)$ with weight at least $\frac{1}{k}w(G)+\frac{k-1}{k}w(M)$ and a bisection with weight at least $\frac{k+1}{2k}w(G)+\frac{k-1}{2k}w(M)$. 

In particular, if $\Delta(G)\leq 3$, then there exists a bipartite subgraph $B_C\in \mathcal{B}_b(G)$ with weight at least $\frac{1}{5}w(G)+\frac{4}{5}w(M)$ and a bisection with weight at least $\frac{3}{5}w(G)+\frac{2}{5}w(M)$.
\end{lemma}
\begin{proof}
	 For any edge $ab\in M$, we denote by $v_{ab}$ the vertex corresponding to it in $G/M$. Let the weight  $w(v_{ab}v_{cd})$ of any edge $v_{ab}v_{cd}$ be the sum of the weights of edges in the bipartite graph $G[\{a,b,c,d\}]$ (the graph is bipartite as it is triangle-free). 
	 
     Since $\chi'(G/M)\leq k$, there is a matching $M'$ in $G/M$ with weight at least $\frac{1}{k}w(G/M)=\frac{1}{k}(w(G)-w(M))$. Observe that each edge $v_{ab}v_{cd}\in M'$ corresponds to the bipartite subgraph $G[\{a,b,c,d\}]$ which is clearly in $\mathcal{B}_b(G)$ since $G[\{a,b,c,d\}]$ has a perfect matching, and $G$ is triangle-free. Let $B$ be the union of these bipartite subgraphs as well as edges in $M$ which are not contained in any of these bipartite subgraphs. Then,  
	 	\begin{eqnarray*}
	 	w(B)&\geq& w(M)+w(M')\\
	 	&\geq&w(M)+\frac{1}{k}(w(G)-w(M))\\
	 	&=&\frac{1}{k}w(G)+\frac{k-1}{k}w(M),
	 \end{eqnarray*}
	 and therefore by Lemma \ref{lem:bbsg} we have that $G$ admits a bisection with weight at least $\frac{w(G)+w(B)}{2}\geq\frac{1}{2}(w(G)+\frac{1}{k}w(G)+\frac{k-1}{k}w(M))=\frac{k+1}{2k}w(G)+\frac{k-1}{2k}w(M)$.
	
	If $\Delta(G)\leq 3$, then $\Delta(G/M)\leq 4$. Thus, by Theorem \ref{thm:chi}, we have $\chi'(G/M)\leq 5$, and therefore we are done.
\end{proof}

\begin{theorem}\label{thm:main}
		Every bridgeless triangle-free subcubic graph $G$ has a bisection with weight at least $\theta\cdot w(G)$, where $\theta = \frac{613}{855} \approx 0.716959$.
\end{theorem}
\begin{proof}
Without loss of generality, we may assume that $G$ is 2-edge-connected as proving this case immediately implies the general bridgeless case. 	   
 Now we assume that $|V(G)|$ is even. At the end of the proof of this theorem, we will show how to solve the odd case using similar arguments. 
   
Since $G$ is bridgeless, there is no vertex of degree one. For any $i\in \{2,3\}$, let $V_i(G)$ denote the set of vertices with degree $i$ in $G$. Note that since $|V(G)|$ and $|V_3(G)|$ are even, $|V_2(G)|$ is even. Since $\Delta(G[V_2(G)])\leq 2$ and $G$ is \AY{2-edge-connected}, $G[V_2(G)]$ is either an even cycle containing all vertices in $G$ or a collection of paths. If $G[V_2(G)]$ is an even cycle, then $G=G[V_2(G)]$. Thus, $G$ is a balanced bipartite graph by itself, and therefore partite sets of $G$ form a bisection with weight $w(G)$. In the following, we may assume that $G[V_2(G)]$ is a collection of paths, i.e., $G[V_2(G)]=\cup_{i=1}^tP_i$, \AY{where each $P_i$ is a path.} Then, for each $i\in [t]$, let $M_i$ be a maximum matching of $P_i$. We add a parallel edge with weight 0 to each edge in $M_i$. Note that vertices with degree two in the new multigraph now induce an independent set. Add an edge of weight 0 between two vertices of degree two recursively until adding any edge creates a triangle. Denote the resulting multigraph by $H$. Note that $H$ is also 2-edge-connected and triangle-free. Also, note that $w(G)=w(H)$ and every bisection of $H$ is a bisection of $G$ with the same weight. Thus, in the following, we only need to show that $H$ has a bisection with desired weights.
   
      \begin{claim}\label{cl:2}
       $|V_2(H)|=0$ or $2$.
       \end{claim} 
       
       {\bf The proof of Claim \ref{cl:2}. } Recall that every pair of vertices in $V_2(H)$ has at least one common neighbour and therefore in total there are at least $\binom{|V_2(H)|}{2}$ common neighbours (counting multiplicity). Let $N_i$ be the set of such common neighbours with $i\in\{2,3\}$ neighbours in $V_2(H)$. As $H$ is subcubic $N_2\cup N_3$ is the set of all such common neighbours. Note that every vertex in $N_3$ is a common neighbour of three pair of vertices in $V_2(H)$ and every vertex in $N_2$ is a common neighbour of one pair of vertices in $V_2(H)$. Therefore, we have $3|N_3|+|N_2|\geq \binom{|V_2(H)|}{2}$. Thus,
       \begin{eqnarray*}
       	2|V_2(H)|\geq|(N_2\cup N_3, V_2(H))|=3|N_3|+2|N_2|\geq 3|N_3|+|N_2|\geq \binom{|V_2(H)|}{2},
       \end{eqnarray*}
     which implies $|V_2(H)|\leq 5$. Thus, $|V_2(H)|\leq 4$ since $|V_2(H)|$ is even. If $|V_2(H)|< 4$ we are done. Thus, we may assume that $|V_2(H)|=4$. From the above inequalities, we have $3|N_3|+2|N_2|\leq 8$ and $3|N_3|+|N_2|\geq 6$ which implies $|N_2|\leq 2$. In addition, $|N_2|\neq 2$, since if so then $4\leq 3|N_3|\leq 4$ which contradicts the fact that $|N_3|$ is an integer. Thus, $|N_2|\leq 1$. If $|N_2|=0$, then $|N_3|=2$, and there are two vertices in $V_2(H)$ which don't have any common neighbour, a contradiction. If $|N_2|=1$, then $|N_3|=2$. Now the structure of $H[N_3\cup N_2\cup V_2(H)]$ is clear (see Fig. \ref{fig:N_2N_3}). Let $x$ be the vertex in $N_2$ and $y$ its neighbour outside $V_2(H)$.
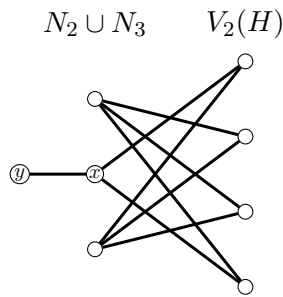
\begin{figure}
\centering
     \begin{tikzpicture}[scale=1]
     	     	\tikzstyle{vertexL}=[circle,draw, minimum size=8pt, scale=0.7, inner sep=0.5pt]
     	\tikzstyle{vertexS}=[circle,draw, minimum size=8pt, scale=0.5, inner sep=0.5pt]
     	\node (a) at (-1,2.5) [vertexL]{$y$};
     	\node (a0) at (0,3.5) [vertexL]{};
     	\node (a1) at (0,2.5) [vertexL]{$x$};
     	\node (a2) at (0,1.5) [vertexL]{};
     	\node (a3) at (2,1) [vertexL]{};
     	\node (a4) at (2,2) [vertexL]{};
     	\node (a5) at (2,3) [vertexL]{};
     	\node (a6) at (2,4) [vertexL]{};
     	 \node at (0,4.5) {$N_2\cup N_3$};
       \node at (2,4.5) {$V_2(H)$};
     	\draw[line width=0.04cm] (a) to (a1);  
     	\draw[line width=0.04cm] (a2) to (a6); 
     	\draw[line width=0.04cm] (a0) to (a3); 
     	\draw[line width=0.04cm] (a2) to (a5);  
     	\draw[line width=0.04cm] (a0) to (a5); 
     	\draw[line width=0.04cm] (a0) to (a4);  
     	\draw[line width=0.04cm] (a2) to (a4);   
     	\draw[line width=0.04cm] (a1) to (a3); 
     	\draw[line width=0.04cm] (a1) to (a6);  
     \end{tikzpicture}
     	\caption{ $H[N_3\cup N_2\cup V_2(H)]$}
     \label{fig:N_2N_3}
     \end{figure}
      One can observe that $xy$ is a bridge of $H$, a contradiction. \smallQED
       
       \begin{claim}\label{cl:pathp}
        $H$ has a perfect matching $M$ such that $H-M$ is either a collection of cycles (with no 3-cycles) or a collection of cycles (with no 3-cycles) with one additional path $P=p_1p_2\dots p_n$ where $p_2p_{n}\in M$ and $n\geq 5$.
       \end{claim}
       
       {\bf The proof of Claim \ref{cl:pathp}. } If $|V_2(H)|=0$, then by Lemma \ref{lem:mt-gvpt}, $H$ has a perfect matching. So, $H-M$ is 2-regular and thus a collection of cycles with no 3-cycles (since $H$ is triangle-free). If $|V_2(H)|=2$, then we assume that $V_2(H)=\{p_1,p_n\}$. Let $p_2$ be the common neighbour of $p_1$ and $p_n$. By Lemma \ref{lem:mt-gvpt}, $H+p_1p_n$ has a matching $M$ containing $p_2p_n$ and therefore does not contain $p_1p_n$. Thus, $M$ is also a perfect matching of $H$. Since $p_1$ and $p_n$ are the only two vertices with degree one in $H-M$, $H-M$ is a collection of cycles and a path $P=p_1p_2\dots p_n$. Since $p_2p_n\in M$, $n\geq 4$. And if $n=4$, then $p_2p_3p_4p_2$ is a triangle, a contradiction to $H$ is triangle-free. Thus, $n\geq 5$. \smallQED

       \vspace{2mm}
       
\GG{
\begin{claim}\label{cl:green}
 $H$ has a bisection $B^*$ such that 
 \begin{equation}\label{eq:1}
           w(B^*)\geq\frac{3}{5}w(H)+\frac{2}{5}w(M).
       \end{equation}      
 \end{claim}
  {\bf The proof of Claim \ref{cl:green}. } 
  Let $H_1$ be a new (simple) graph obtained from $H$ by deleting one edge with weight 0 from all multiple edges in $H$. Note that $M$ is still a perfect matching of $H_1$, $w(H_1)=w(H)$ and therefore by applying Lemma \ref{lem:triangle-free cubic}, $H_1$ has a bisection $B^*$ such that
      \[w(B^*)\geq\frac{3}{5}w(H_1)+\frac{2}{5}w(M)=\frac{3}{5}w(H)+\frac{2}{5}w(M).\]
Clearly, the bisection is also a bisection of $H$ with the same weight. \smallQED
}
       
       \vspace{2mm}
       
    In Claims \ref{cl:neq5711}-\ref{cl:pathcase}, we will define a probability space of some bipartite subgraphs in $\mathcal{B}_b(H)$ for each cycle and path in $H-M$. For simplicity, we call the space a random bipartite subgraph. Then, we will use its expectation to show another lower bound, which in combination with (\ref{eq:1}) implies our result. 

\GG{In Claims  \ref{cl:neq5711} -- \ref{cl:c11} we do not assume that $|V(G)|$ is even (in the other claims, we do assume it). This is because  Claims  \ref{cl:neq5711} -- \ref{cl:c11} are used in cases when $|V(G)|$ is even and odd.}
   
       \begin{claim}\label{cl:neq5711}
 Let $C$ be a cycle in $H-M$.   	If $|C|\neq 5,7$ or $11$, then there exists a random bipartite subgraph $B_C\in \mathcal{B}_b(H[V(C)])$ such that $\mathbb{E}(w(B_C))\geq \frac{5}{8}w(C)$ and every vertex is not in $V(B_C)$ with probability \AY{exactly} $1/8$.
      \end{claim}
      {\bf The proof of Claim \ref{cl:neq5711}. } Let $l=4k+r$ where $k \geq 1$ and $r \in \{0,1,2,3\}$ and let \AY{$C=c_0 c_1 c_2 \ldots c_{l-1}c_0$} be a cycle \AY{in $H-M$.} Choose $i$ from $[l]$ randomly and uniformly. 
      Let $B_C\in \mathcal{B}_b(H[V(C)])$ \AY{be defined as follows.
 \begin{itemize}
        \item $\bigcup_{t=0}^{k-1} H[\{c_{i+4t},c_{i+4t+1}, c_{i+4t+2},c_{i+4t+3}\}]$,
    \end{itemize}
}


 where indices are taken modulo $l$. Note that $B_C$ covers $3k$ edges of $C$, and $r$ vertices are not in $V(B_C)$. 
 
 We first consider the case when $7r \leq 4k$. In this case let $x=7(4k+r)/(32k)$ and note that $0 \leq x \leq 1$. Pick one of the randomly generated $B_C$ above with probability $x$ and let $B_C$ be a null graph (i.e., $B_C=(\emptyset,\emptyset)$) with probability $1-x$. The probability that a vertex is not included in $V(B_C)$ is as follows:
      \[
      x \frac{r}{4k+r} + (1-x) \cdot 1  =  1+ x \left( \frac{r}{4k+r} - 1 \right) 
      =  1+ \frac{7(4k+r)}{32k} \times \frac{-4k}{4k+r} =1 - \frac{7}{8}  =  \frac{1}{8}.
      \]
      
      The probability that any given edge of $C$ belongs to $B_C$ is as follows:
      
      \[
      x \frac{3k}{4k+r} + (1-x) \cdot 0  =  \frac{7(4k+r)}{32k} \times \frac{3k}{4k+r} 
      =  \frac{21}{32}  =  \frac{5.25}{8}  >  \frac{5}{8}.
      \]    
      This completes the case when $7r \leq 4k$. 
      
      Now assume that $7r > 4k$. Clearly, $r$ can not be 0 or 1 as $|C|\neq 5$. Thus, $r \geq 2$. Let $y = (7r-4k)/16$ and note that $0 < y < 1$ as $|C|\neq 7$. 
     Let $B_C\in \mathcal{B}_b(H[V(C)])$ formed by 
      
      \begin{itemize}
      	\item $(\cup_{t=0}^{\frac{l-6}{4}} H[\{c_{i+4t},c_{i+4t+1}, c_{i+4t+2},c_{i+4t+3}\}])\cup \{c_{i+l-2}c_{i+l-1}\}$ when $r=2$;
      	\item $(\cup_{t=0}^{\frac{l-7}{4}} H[\{c_{i+4t},c_{i+4t+1}, c_{i+4t+2},c_{i+4t+3}\}])\cup \{c_{i+l-3}c_{i+l-2}\}$ when $r=3$,
      \end{itemize}
      with probability $y$, where indices are taken modulo $l$. Note that in this case, $B_C$ covers $3k+1$ edges of $C$ and $r-2$ vertices are not in $V(B_C)$.
      
       Let $B_C\in \mathcal{B}_b(H[V(C)])$ formed by 
      
      \begin{itemize}
      	\item $(\cup_{t=0}^{\frac{l-6}{4}} H[\{c_{i+4t},c_{i+4t+1}, c_{i+4t+2},c_{i+4t+3}\}])$ when $r=2$;
      	\item $(\cup_{t=0}^{\frac{l-7}{4}} H[\{c_{i+4t},c_{i+4t+1}, c_{i+4t+2},c_{i+4t+3}\}])$ when $r=3$,
      \end{itemize}
      with probability $1-y$, where indices are taken modulo $l$. Note that in this case, $B_C$ covers $3k$ edges of $C$ and $r$ vertices are not in $V(B_C)$.
    
      The probability that a vertex is not in $V(B_C)$ in the resulting subgraph is as follows:
      
      \[
      \begin{array}{rcl}
      	y \frac{r-2}{4k+r} + (1-y) \frac{r}{4k+r} & = & \frac{r}{4k+r} + y \frac{(r-2)-r}{4k+r} = \frac{r}{4k+r} + \frac{-2(7r-4k)}{16(4k+r)}  \\
      	& = & \frac{16r-14r+8k}{16(4k+r)} = \frac{1}{8}. \\
      \end{array}
      \]
      
      The probability that any given edge of $C$ belongs to the resulting subgraph is as follows:
      
      \[
      \begin{array}{rcl}
      	y \frac{3k+1}{4k+r} + (1-y) \frac{3k}{4k+r} & = & \frac{3k}{4k+r} + y \frac{3k+1-3k}{4k+r} = \frac{3k}{4k+r} + \frac{7r-4k}{16} \times \frac{1}{4k+r} \\
      	& = & \frac{48k +7r-4k}{16(4k+r)} = \frac{44k +7r}{64k+16r}.     \\
      \end{array}
      \]
      
  If $4k \geq 3r$ then  $\frac{44k +7r}{64k+16r} \geq \frac{44k +7r -4k + 3r }{64k+16r} = 5/8$. So we get the desired result when $4k \geq 3r$ and may therefore assume that $4k < 3r$. When $r=2$ this implies that $k=1$ and $l=6$ and when $r=3$ then $k \in \{1,2\}$ and $l \in \{7,11\}$ but $|C|\neq 7$ or $11$. Thus, the only remaining case is when $l=6$ and $C$ is a $6$-cycle. 
In this case, \AY{$G[V(C)]$ is bipartite as $G$ is triangle-free, and} let $B_C=C$ with probability $7/8$ and let $B_C$ be a null graph with probability $1/8$. \smallQED

          \begin{claim}\label{cl:c5}
      	If $C$ is  a 5-cycle in $H-M$, then there exists a random bipartite subgraph $B_C\in \mathcal{B}_b(H[V(C)])$ such that $\mathbb{E}(w(B_C))\geq \frac{3}{5}w(C)$ and every vertex is not in the $V(B_C)$ with probability $1/5$.
      \end{claim}
      {\bf The proof of Claim \ref{cl:c5}. } Let $C=c_1c_2c_3c_4c_5c_1$. For each $i\in [5]$, let $B_C=C-c_i$ with probability $1/5$. Then, $B_C$ is a desired bipartite subgraph in $\mathcal{B}_b(H[V(C)])$. \smallQED

      \begin{claim}\label{cl:c7}
      	If $C$ is a 7-cycle of $H-M$, then there exists a random bipartite subgraph $B_C\in \mathcal{B}_b(H[V(C)])$ such that $\mathbb{E}(w(B_C))\geq \frac{5}{8}w(C)$ and every vertex is not in the $V(B_C)$ with probability at least $1/8$ and at most $1/4$.
      \end{claim}
      
      {\bf The proof of Claim \ref{cl:c7}. }   If $C$ has no chord in $G$, then by picking a vertex $c \in V(C)$ at random (uniformly with probability $1/7$) and considering the bipartite subgraph $B_C=C-c\in \mathcal{B}_b(H[V(C)])$ we note that every edge of $C$ will belong to $B_C$ with probability $5/7$.  Moreover, every vertex in $V(C)$ will not belong to $B_C$ with probability $1/7$.
      
We now consider the case when $C$ has exactly one chord. As $H$ is triangle-free, the cycle and chord must be as in Fig. \ref{figC7}. Let $B_C$ be each bipartite subgraph \AY{in} the table on the right of Fig. \ref{figC7} with probability $1/8$. Every edge in $C$ will appear in $B_C$ with probability at least $5/8$ as the numbers next to each edge in Fig. \ref{figC7} indicate how many subgraphs contain this edge. Also, we note that each vertex in $C$ is not in $V(B_C)$ in at least one, and at most two, of the subgraphs in the table and thus with probability at least $1/8$ and at most $1/4$. 
      
      \2
      
      	\begin{figure}
      	\begin{minipage}{3cm}
      		\tikzstyle{vertexL}=[circle,draw, minimum size=8pt, scale=0.6, inner sep=0.5pt]
      		\begin{tikzpicture}[scale=0.4]
      			\node at (0,10) {  };
      			\node (a1) at (3,9) [vertexL]{};
      			\node (a2) at (5,7) [vertexL]{};
      			\node (a3) at (5,4) [vertexL]{};
      			\node (a4) at (5,1) [vertexL]{};
      			\node (a5) at (1,1) [vertexL]{};
      			\node (a6) at (1,4) [vertexL]{};
      			\node (a7) at (1,7) [vertexL]{};
      			
      			\draw[line width=0.04cm] (a1) to (a2);   \node at (4.5,8.3) {{\small 5}};
      			\draw[line width=0.04cm] (a2) to (a3);   \node at (5.6,5.5) {{\small 5}};
      			\draw[line width=0.04cm] (a3) to (a4);   \node at (5.6,2.5) {{\small 5}};
      			\draw[line width=0.04cm] (a4) to (a5);   \node at (3,0.5) {{\small 5}};
      			\draw[line width=0.04cm] (a5) to (a6);   \node at (0.4,2.5) {{\small 6}};
      			\draw[line width=0.04cm] (a6) to (a7);   \node at (0.4,5.5) {{\small 5}};
      			\draw[line width=0.04cm] (a7) to (a1);   \node at (1.5,8.3) {{\small 5}};
      			
      			\draw[line width=0.04cm] (a1) to (a4);			
      		\end{tikzpicture}

      	\end{minipage} 
      \begin{tabular}{|c|c|c|c|} \hline
      	\tikzstyle{vertexL}=[circle,draw, minimum size=8pt, scale=0.6, inner sep=0.5pt]
      	\begin{tikzpicture}[scale=0.3]
      		\node at (3,10) {  };
      		\node (a1) at (3,9) [vertexL]{}; 
      		\node (a2) at (5,7) [vertexL]{};
      		\node (a3) at (5,4) [vertexL]{};
      		\node (a4) at (5,1) [vertexL]{};
      		\node (a5) at (1,1) [vertexL]{};
      		\node (a6) at (1,4) [vertexL]{};
      		\node (a7) at (1,7) [vertexL]{};
      		\draw[line width=0.04cm] (a1) to (a2);
      		\draw[line width=0.04cm] (a2) to (a3);
      		\draw[line width=0.04cm] (a3) to (a4);
      		\draw[line width=0.04cm] (a4) to (a5);
      		\draw[line width=0.04cm] (a5) to (a6);
      		\draw[line width=0.04cm] (a1) to (a4);
      	\end{tikzpicture} & 
      	\tikzstyle{vertexL}=[circle,draw, minimum size=8pt, scale=0.6, inner sep=0.5pt]
      	\begin{tikzpicture}[scale=0.3]
      		\node at (3,10) {  };
      		\node (a1) at (3,9) [vertexL]{};
      		\node (a2) at (5,7) [vertexL]{};
      		\node (a3) at (5,4) [vertexL]{};
      		\node (a4) at (5,1) [vertexL]{};
      		\node (a5) at (1,1) [vertexL]{};
      		\node (a6) at (1,4) [vertexL]{};
      		\node (a7) at (1,7) [vertexL]{};
      		\draw[line width=0.04cm] (a1) to (a2);
      		\draw[line width=0.04cm] (a2) to (a3);
      		\draw[line width=0.04cm] (a3) to (a4);
      		\draw[line width=0.04cm] (a4) to (a5);
      		\draw[line width=0.04cm] (a7) to (a1);
      		\draw[line width=0.04cm] (a1) to (a4);
      	\end{tikzpicture} & 
      	\tikzstyle{vertexL}=[circle,draw, minimum size=8pt, scale=0.6, inner sep=0.5pt]
      	\begin{tikzpicture}[scale=0.3]
      		\node at (3,10) {  };
      		\node (a1) at (3,9) [vertexL]{};
      		\node (a2) at (5,7) [vertexL]{};
      		\node (a3) at (5,4) [vertexL]{};
      		\node (a4) at (5,1) [vertexL]{};
      		\node (a5) at (1,1) [vertexL]{};
      		\node (a6) at (1,4) [vertexL]{};
      		\node (a7) at (1,7) [vertexL]{};
      		\draw[line width=0.04cm] (a1) to (a2);
      		\draw[line width=0.04cm] (a2) to (a3);
      		\draw[line width=0.04cm] (a3) to (a4);
      		\draw[line width=0.04cm] (a6) to (a7);
      		\draw[line width=0.04cm] (a7) to (a1);
      		\draw[line width=0.04cm] (a1) to (a4);
      	\end{tikzpicture} & 
      	\tikzstyle{vertexL}=[circle,draw, minimum size=8pt, scale=0.6, inner sep=0.5pt]
      	\begin{tikzpicture}[scale=0.3]
      		\node at (3,10) {  };
      		\node (a1) at (3,9) [vertexL]{};
      		\node (a2) at (5,7) [vertexL]{};
      		\node (a3) at (5,4) [vertexL]{};
      		\node (a4) at (5,1) [vertexL]{};
      		\node (a5) at (1,1) [vertexL]{};
      		\node (a6) at (1,4) [vertexL]{};
      		\node (a7) at (1,7) [vertexL]{};
      		\draw[line width=0.04cm] (a1) to (a2);
      		\draw[line width=0.04cm] (a2) to (a3);
      		\draw[line width=0.04cm] (a5) to (a6);
      		\draw[line width=0.04cm] (a6) to (a7);
      		\draw[line width=0.04cm] (a7) to (a1);
      	\end{tikzpicture} \\ \hline
      	\tikzstyle{vertexL}=[circle,draw, minimum size=8pt, scale=0.6, inner sep=0.5pt]
      	\begin{tikzpicture}[scale=0.3]
      		\node at (3,10) {  };
      		\node (a1) at (3,9) [vertexL]{}; 
      		\node (a2) at (5,7) [vertexL]{}; 
      		\node (a3) at (5,4) [vertexL]{}; 
      		\node (a4) at (5,1) [vertexL]{};
      		\node (a5) at (1,1) [vertexL]{};
      		\node (a6) at (1,4) [vertexL]{}; 
      		\node (a7) at (1,7) [vertexL]{}; 
      		\draw[line width=0.04cm] (a3) to (a4);
      		\draw[line width=0.04cm] (a4) to (a5);
      		\draw[line width=0.04cm] (a5) to (a6);
      		\draw[line width=0.04cm] (a7) to (a1);
      	\end{tikzpicture} & 
      	\tikzstyle{vertexL}=[circle,draw, minimum size=8pt, scale=0.6, inner sep=0.5pt]
      	\begin{tikzpicture}[scale=0.3]
      		\node at (3,10) {  };
      		\node (a1) at (3,9) [vertexL]{};
      		\node (a2) at (5,7) [vertexL]{};
      		\node (a3) at (5,4) [vertexL]{};
      		\node (a4) at (5,1) [vertexL]{};
      		\node (a5) at (1,1) [vertexL]{};
      		\node (a6) at (1,4) [vertexL]{};
      		\node (a7) at (1,7) [vertexL]{};
      		\draw[line width=0.04cm] (a2) to (a3);
      		\draw[line width=0.04cm] (a4) to (a5);
      		\draw[line width=0.04cm] (a5) to (a6);
      		\draw[line width=0.04cm] (a6) to (a7);
      	\end{tikzpicture} & 
      	\tikzstyle{vertexL}=[circle,draw, minimum size=8pt, scale=0.6, inner sep=0.5pt]
      	\begin{tikzpicture}[scale=0.3]
      		\node at (3,10) {  };
      		\node (a1) at (3,9) [vertexL]{};
      		\node (a2) at (5,7) [vertexL]{};
      		\node (a3) at (5,4) [vertexL]{};
      		\node (a4) at (5,1) [vertexL]{};
      		\node (a5) at (1,1) [vertexL]{};
      		\node (a6) at (1,4) [vertexL]{};
      		\node (a7) at (1,7) [vertexL]{};
      		\draw[line width=0.04cm] (a1) to (a2);
      		\draw[line width=0.04cm] (a4) to (a5);
      		\draw[line width=0.04cm] (a5) to (a6);
      		\draw[line width=0.04cm] (a6) to (a7);
      	\end{tikzpicture} & 
      	\tikzstyle{vertexL}=[circle,draw, minimum size=8pt, scale=0.6, inner sep=0.5pt]
      	\begin{tikzpicture}[scale=0.3]
      		\node at (3,10) {  };
      		\node (a1) at (3,9) [vertexL]{};
      		\node (a2) at (5,7) [vertexL]{};
      		\node (a3) at (5,4) [vertexL]{};
      		\node (a4) at (5,1) [vertexL]{};
      		\node (a5) at (1,1) [vertexL]{};
      		\node (a6) at (1,4) [vertexL]{};
      		\node (a7) at (1,7) [vertexL]{};
      		\draw[line width=0.04cm] (a3) to (a4);
      		\draw[line width=0.04cm] (a5) to (a6);
      		\draw[line width=0.04cm] (a6) to (a7);
      		\draw[line width=0.04cm] (a7) to (a1);
      	\end{tikzpicture} \\ \hline
      \end{tabular}
      	\caption{On the left $C$ with a chord, on the right graphs $B_C$}\label{figC7}
      \end{figure}
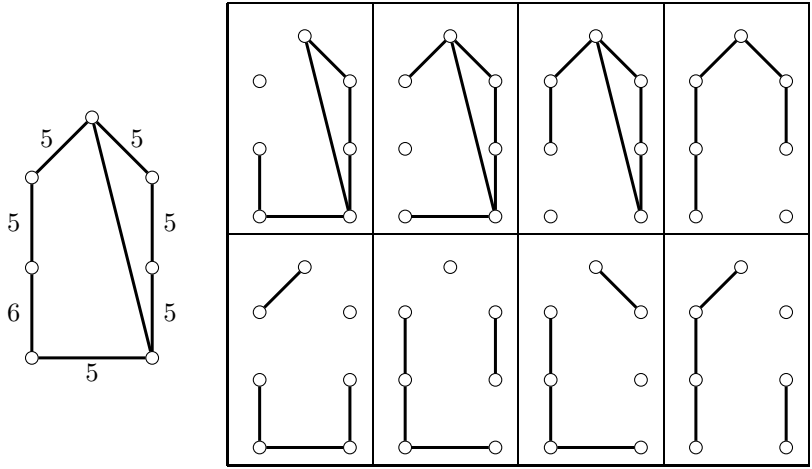
      
      \2
      
We now consider the case when the $7$-cycle has exactly two chords. There are only two non-isomorphic possibilities depicted on the left of Fig. \ref{figC7'}. In each case, let $B_C$ be each bipartite subgraph in the table on the right of Fig. \ref{figC7'} with probability $1/8$. One can observe that every edge in $C$ will appear in $B_C$ with probability at least $5/8$, and each vertex in $C$ is not in $V(B_C)$ with probability at least $1/8$ and at most $1/4$. 
      
      \2
      \begin{figure}

      \begin{minipage}{3cm}
      	Option 1
      	
      	\vspace{0.5cm} 
      	
      	\tikzstyle{vertexL}=[circle,draw, minimum size=8pt, scale=0.6, inner sep=0.5pt]
      	\begin{tikzpicture}[scale=0.4]
      		\node at (0,10) {  };
      		\node (a1) at (3,9) [vertexL]{};
      		\node (a2) at (5,7) [vertexL]{};
      		\node (a3) at (5,4) [vertexL]{};
      		\node (a4) at (5,1) [vertexL]{};
      		\node (a5) at (1,1) [vertexL]{};
      		\node (a6) at (1,4) [vertexL]{};
      		\node (a7) at (1,7) [vertexL]{};
      		
      		\draw[line width=0.04cm] (a1) to (a2);   \node at (4.5,8.3) {{\small 5}};
      		\draw[line width=0.04cm] (a2) to (a3);   \node at (5.6,5.5) {{\small 5}};
      		\draw[line width=0.04cm] (a3) to (a4);   \node at (5.6,2.5) {{\small 5}};
      		\draw[line width=0.04cm] (a4) to (a5);   \node at (3,0.5) {{\small 5}};
      		\draw[line width=0.04cm] (a5) to (a6);   \node at (0.4,2.5) {{\small 6}};
      		\draw[line width=0.04cm] (a6) to (a7);   \node at (0.4,5.5) {{\small 5}};
      		\draw[line width=0.04cm] (a7) to (a1);   \node at (1.5,8.3) {{\small 5}};
      		
      		\draw[line width=0.04cm] (a1) to (a4);
      		\draw[line width=0.04cm] (a7) to (a3);
      	\end{tikzpicture}
      \end{minipage} 
      \begin{tabular}{|c|c|c|c|} \hline
      \tikzstyle{vertexL}=[circle,draw, minimum size=8pt, scale=0.6, inner sep=0.5pt]
      \begin{tikzpicture}[scale=0.3]
      	\node at (3,10) {  };
      	\node (a1) at (3,9) [vertexL]{};
      	\node (a2) at (5,7) [vertexL]{};
      	\node (a3) at (5,4) [vertexL]{};
      	\node (a4) at (5,1) [vertexL]{};
      	\node (a5) at (1,1) [vertexL]{};
      	\node (a6) at (1,4) [vertexL]{};
      	\node (a7) at (1,7) [vertexL]{};
      	\draw[line width=0.04cm] (a1) to (a2);  
      	\draw[line width=0.04cm] (a2) to (a3);  
      	\draw[line width=0.04cm] (a3) to (a4);  
      	\draw[line width=0.04cm] (a4) to (a5);  
      	\draw[line width=0.04cm] (a5) to (a6);  
      	\draw[line width=0.04cm] (a1) to (a4);
      \end{tikzpicture} & 
      \tikzstyle{vertexL}=[circle,draw, minimum size=8pt, scale=0.6, inner sep=0.5pt]
      \begin{tikzpicture}[scale=0.3]
      	\node at (3,10) {  };
      	\node (a1) at (3,9) [vertexL]{};
      	\node (a2) at (5,7) [vertexL]{};
      	\node (a3) at (5,4) [vertexL]{};
      	\node (a4) at (5,1) [vertexL]{};
      	\node (a5) at (1,1) [vertexL]{};
      	\node (a6) at (1,4) [vertexL]{};
      	\node (a7) at (1,7) [vertexL]{};
      	\draw[line width=0.04cm] (a1) to (a2);   
      	\draw[line width=0.04cm] (a2) to (a3);   
      	\draw[line width=0.04cm] (a3) to (a4);   
      	\draw[line width=0.04cm] (a4) to (a5);   
      	\draw[line width=0.04cm] (a7) to (a1);   
      	\draw[line width=0.04cm] (a1) to (a4);
      	\draw[line width=0.04cm] (a7) to (a3);
      \end{tikzpicture} &
      \tikzstyle{vertexL}=[circle,draw, minimum size=8pt, scale=0.6, inner sep=0.5pt]
      \begin{tikzpicture}[scale=0.3]
      	\node at (3,10) {  };
      	\node (a1) at (3,9) [vertexL]{};
      	\node (a2) at (5,7) [vertexL]{};
      	\node (a3) at (5,4) [vertexL]{};
      	\node (a4) at (5,1) [vertexL]{};
      	\node (a5) at (1,1) [vertexL]{};
      	\node (a6) at (1,4) [vertexL]{};
      	\node (a7) at (1,7) [vertexL]{};
      	\draw[line width=0.04cm] (a1) to (a2);   
      	\draw[line width=0.04cm] (a2) to (a3);   
      	\draw[line width=0.04cm] (a3) to (a4);   
      	\draw[line width=0.04cm] (a6) to (a7);   
      	\draw[line width=0.04cm] (a7) to (a1);   
      	\draw[line width=0.04cm] (a1) to (a4);
      	\draw[line width=0.04cm] (a7) to (a3);
      \end{tikzpicture} &
      \tikzstyle{vertexL}=[circle,draw, minimum size=8pt, scale=0.6, inner sep=0.5pt]
      \begin{tikzpicture}[scale=0.3]
      	\node at (3,10) {  };
      	\node (a1) at (3,9) [vertexL]{};
      	\node (a2) at (5,7) [vertexL]{};
      	\node (a3) at (5,4) [vertexL]{};
      	\node (a4) at (5,1) [vertexL]{};
      	\node (a5) at (1,1) [vertexL]{};
      	\node (a6) at (1,4) [vertexL]{};
      	\node (a7) at (1,7) [vertexL]{};
      	\draw[line width=0.04cm] (a1) to (a2);   
      	\draw[line width=0.04cm] (a2) to (a3);   
      	\draw[line width=0.04cm] (a5) to (a6);   
      	\draw[line width=0.04cm] (a6) to (a7);   
      	\draw[line width=0.04cm] (a7) to (a1);   
      	\draw[line width=0.04cm] (a7) to (a3);
      \end{tikzpicture} \\ \hline
      \tikzstyle{vertexL}=[circle,draw, minimum size=8pt, scale=0.6, inner sep=0.5pt]
      \begin{tikzpicture}[scale=0.3]
      	\node at (3,10) {  };
      	\node (a1) at (3,9) [vertexL]{};
      	\node (a2) at (5,7) [vertexL]{};
      	\node (a3) at (5,4) [vertexL]{};
      	\node (a4) at (5,1) [vertexL]{};
      	\node (a5) at (1,1) [vertexL]{};
      	\node (a6) at (1,4) [vertexL]{};
      	\node (a7) at (1,7) [vertexL]{};
      	\draw[line width=0.04cm] (a3) to (a4);   
      	\draw[line width=0.04cm] (a4) to (a5);   
      	\draw[line width=0.04cm] (a5) to (a6);   
      	\draw[line width=0.04cm] (a7) to (a1);   
      \end{tikzpicture} &
      \tikzstyle{vertexL}=[circle,draw, minimum size=8pt, scale=0.6, inner sep=0.5pt]
      \begin{tikzpicture}[scale=0.3]
      	\node at (3,10) {  };
      	\node (a1) at (3,9) [vertexL]{};
      	\node (a2) at (5,7) [vertexL]{};
      	\node (a3) at (5,4) [vertexL]{};
      	\node (a4) at (5,1) [vertexL]{};
      	\node (a5) at (1,1) [vertexL]{};
      	\node (a6) at (1,4) [vertexL]{};
      	\node (a7) at (1,7) [vertexL]{};
      	\draw[line width=0.04cm] (a2) to (a3);   
      	\draw[line width=0.04cm] (a4) to (a5);   
      	\draw[line width=0.04cm] (a5) to (a6);   
      	\draw[line width=0.04cm] (a6) to (a7);   
      \end{tikzpicture} &
      \tikzstyle{vertexL}=[circle,draw, minimum size=8pt, scale=0.6, inner sep=0.5pt]
      \begin{tikzpicture}[scale=0.3]
      	\node at (3,10) {  };
      	\node (a1) at (3,9) [vertexL]{};
      	\node (a2) at (5,7) [vertexL]{};
      	\node (a3) at (5,4) [vertexL]{};
      	\node (a4) at (5,1) [vertexL]{};
      	\node (a5) at (1,1) [vertexL]{};
      	\node (a6) at (1,4) [vertexL]{};
      	\node (a7) at (1,7) [vertexL]{};
      	\draw[line width=0.04cm] (a1) to (a2);   
      	\draw[line width=0.04cm] (a4) to (a5);   
      	\draw[line width=0.04cm] (a5) to (a6);   
      	\draw[line width=0.04cm] (a6) to (a7);   
      \end{tikzpicture} &
      \tikzstyle{vertexL}=[circle,draw, minimum size=8pt, scale=0.6, inner sep=0.5pt]
      \begin{tikzpicture}[scale=0.3]
      	\node at (3,10) {  };
      	\node (a1) at (3,9) [vertexL]{};
      	\node (a2) at (5,7) [vertexL]{};
      	\node (a3) at (5,4) [vertexL]{};
      	\node (a4) at (5,1) [vertexL]{};
      	\node (a5) at (1,1) [vertexL]{};
      	\node (a6) at (1,4) [vertexL]{};
      	\node (a7) at (1,7) [vertexL]{};
      	\draw[line width=0.04cm] (a3) to (a4);   
      	\draw[line width=0.04cm] (a5) to (a6);   
      	\draw[line width=0.04cm] (a6) to (a7);   
      	\draw[line width=0.04cm] (a7) to (a1);   
      \end{tikzpicture} \\ \hline
      \end{tabular}
      
      \2
      
      \begin{minipage}{3cm}
      	Option 2 
      	
      	\vspace{0.5cm}
      	
      	\tikzstyle{vertexL}=[circle,draw, minimum size=8pt, scale=0.6, inner sep=0.5pt]
      	\begin{tikzpicture}[scale=0.4]
      		\node at (0,10) {  };
      		\node (a1) at (3,9) [vertexL]{};
      		\node (a2) at (5,7) [vertexL]{};
      		\node (a3) at (5,4) [vertexL]{};
      		\node (a4) at (5,1) [vertexL]{};
      		\node (a5) at (1,1) [vertexL]{};
      		\node (a6) at (1,4) [vertexL]{};
      		\node (a7) at (1,7) [vertexL]{};
      		
      		\draw[line width=0.04cm] (a1) to (a2);   \node at (4.5,8.3) {{\small 5}};
      		\draw[line width=0.04cm] (a2) to (a3);   \node at (5.6,5.5) {{\small 5}};
      		\draw[line width=0.04cm] (a3) to (a4);   \node at (5.6,2.5) {{\small 5}}; 
      		\draw[line width=0.04cm] (a4) to (a5);   \node at (3,0.5) {{\small 5}};
      		\draw[line width=0.04cm] (a5) to (a6);   \node at (0.4,2.5) {{\small 5}};
      		\draw[line width=0.04cm] (a6) to (a7);   \node at (0.4,5.5) {{\small 5}};
      		\draw[line width=0.04cm] (a7) to (a1);   \node at (1.5,8.3) {{\small 5}};
      		
      		\draw[line width=0.04cm] (a1) to (a4);
      		\draw[line width=0.04cm] (a2) to (a6);
      	\end{tikzpicture}
      \end{minipage} 
      \begin{tabular}{|c|c|c|c|} \hline
      \tikzstyle{vertexL}=[circle,draw, minimum size=8pt, scale=0.6, inner sep=0.5pt]
      \begin{tikzpicture}[scale=0.3]
      	\node at (3,10) {  };
      	\node (a1) at (3,9) [vertexL]{};
      	\node (a2) at (5,7) [vertexL]{};
      	\node (a3) at (5,4) [vertexL]{};
      	\node (a4) at (5,1) [vertexL]{};
      	\node (a5) at (1,1) [vertexL]{};
      	\node (a6) at (1,4) [vertexL]{};
      	\node (a7) at (1,7) [vertexL]{};
      	
      	\draw[line width=0.04cm] (a1) to (a2);   
      	\draw[line width=0.04cm] (a2) to (a3);  
      	\draw[line width=0.04cm] (a3) to (a4);  
      	\draw[line width=0.04cm] (a4) to (a5);  
      	\draw[line width=0.04cm] (a7) to (a1);  
      	
      	\draw[line width=0.04cm] (a1) to (a4);
      \end{tikzpicture} &
      \tikzstyle{vertexL}=[circle,draw, minimum size=8pt, scale=0.6, inner sep=0.5pt]
      \begin{tikzpicture}[scale=0.3]
      	\node at (3,10) {  };
      	\node (a1) at (3,9) [vertexL]{};
      	\node (a2) at (5,7) [vertexL]{};
      	\node (a3) at (5,4) [vertexL]{};
      	\node (a4) at (5,1) [vertexL]{};
      	\node (a5) at (1,1) [vertexL]{};
      	\node (a6) at (1,4) [vertexL]{};
      	\node (a7) at (1,7) [vertexL]{};
      	
      	\draw[line width=0.04cm] (a1) to (a2);
      	\draw[line width=0.04cm] (a2) to (a3);
      	\draw[line width=0.04cm] (a3) to (a4);
      	\draw[line width=0.04cm] (a6) to (a7);
      	\draw[line width=0.04cm] (a7) to (a1);
      	
      	\draw[line width=0.04cm] (a1) to (a4);
      	\draw[line width=0.04cm] (a2) to (a6);
      \end{tikzpicture} & 
      \tikzstyle{vertexL}=[circle,draw, minimum size=8pt, scale=0.6, inner sep=0.5pt]
      \begin{tikzpicture}[scale=0.3]
      	\node at (3,10) {  };
      	\node (a1) at (3,9) [vertexL]{};
      	\node (a2) at (5,7) [vertexL]{};
      	\node (a3) at (5,4) [vertexL]{};
      	\node (a4) at (5,1) [vertexL]{};
      	\node (a5) at (1,1) [vertexL]{};
      	\node (a6) at (1,4) [vertexL]{};
      	\node (a7) at (1,7) [vertexL]{};
      	
      	\draw[line width=0.04cm] (a1) to (a2);
      	\draw[line width=0.04cm] (a2) to (a3);
      	\draw[line width=0.04cm] (a5) to (a6);
      	\draw[line width=0.04cm] (a6) to (a7);
      	\draw[line width=0.04cm] (a7) to (a1);
      	
      	\draw[line width=0.04cm] (a2) to (a6);
      \end{tikzpicture} & 
      \tikzstyle{vertexL}=[circle,draw, minimum size=8pt, scale=0.6, inner sep=0.5pt]
      \begin{tikzpicture}[scale=0.3]
      	\node at (3,10) {  };
      	\node (a1) at (3,9) [vertexL]{};
      	\node (a2) at (5,7) [vertexL]{};
      	\node (a3) at (5,4) [vertexL]{};
      	\node (a4) at (5,1) [vertexL]{};
      	\node (a5) at (1,1) [vertexL]{};
      	\node (a6) at (1,4) [vertexL]{};
      	\node (a7) at (1,7) [vertexL]{};
      	
      	\draw[line width=0.04cm] (a1) to (a2);
      	\draw[line width=0.04cm] (a4) to (a5);
      	\draw[line width=0.04cm] (a5) to (a6);
      	\draw[line width=0.04cm] (a6) to (a7);
      	
      \end{tikzpicture} \\ \hline
      \tikzstyle{vertexL}=[circle,draw, minimum size=8pt, scale=0.6, inner sep=0.5pt]
      \begin{tikzpicture}[scale=0.3]
      	\node at (3,10) {  };
      	\node (a1) at (3,9) [vertexL]{};
      	\node (a2) at (5,7) [vertexL]{};
      	\node (a3) at (5,4) [vertexL]{};
      	\node (a4) at (5,1) [vertexL]{};
      	\node (a5) at (1,1) [vertexL]{};
      	\node (a6) at (1,4) [vertexL]{};
      	\node (a7) at (1,7) [vertexL]{};
      	
      	\draw[line width=0.04cm] (a1) to (a2);
      	\draw[line width=0.04cm] (a3) to (a4);
      	\draw[line width=0.04cm] (a4) to (a5);
      	\draw[line width=0.04cm] (a5) to (a6);
      	
      \end{tikzpicture} & 
      \tikzstyle{vertexL}=[circle,draw, minimum size=8pt, scale=0.6, inner sep=0.5pt]
      \begin{tikzpicture}[scale=0.3]
      	\node at (3,10) {  };
      	\node (a1) at (3,9) [vertexL]{};
      	\node (a2) at (5,7) [vertexL]{};
      	\node (a3) at (5,4) [vertexL]{};
      	\node (a4) at (5,1) [vertexL]{};
      	\node (a5) at (1,1) [vertexL]{};
      	\node (a6) at (1,4) [vertexL]{};
      	\node (a7) at (1,7) [vertexL]{};
      	
      	\draw[line width=0.04cm] (a2) to (a3);
      	\draw[line width=0.04cm] (a3) to (a4);
      	\draw[line width=0.04cm] (a4) to (a5);
      	\draw[line width=0.04cm] (a7) to (a1);
      	
      \end{tikzpicture} & 
      \tikzstyle{vertexL}=[circle,draw, minimum size=8pt, scale=0.6, inner sep=0.5pt]
      \begin{tikzpicture}[scale=0.3]
      	\node at (3,10) {  };
      	\node (a1) at (3,9) [vertexL]{};
      	\node (a2) at (5,7) [vertexL]{};
      	\node (a3) at (5,4) [vertexL]{};
      	\node (a4) at (5,1) [vertexL]{};
      	\node (a5) at (1,1) [vertexL]{};
      	\node (a6) at (1,4) [vertexL]{};
      	\node (a7) at (1,7) [vertexL]{};
      	
      	\draw[line width=0.04cm] (a3) to (a4);
      	\draw[line width=0.04cm] (a5) to (a6);
      	\draw[line width=0.04cm] (a6) to (a7);
      	\draw[line width=0.04cm] (a7) to (a1);
      	
      \end{tikzpicture} & 
      \tikzstyle{vertexL}=[circle,draw, minimum size=8pt, scale=0.6, inner sep=0.5pt]
      \begin{tikzpicture}[scale=0.3]
      	\node at (3,10) {  };
      	\node (a1) at (3,9) [vertexL]{};
      	\node (a2) at (5,7) [vertexL]{};
      	\node (a3) at (5,4) [vertexL]{};
      	\node (a4) at (5,1) [vertexL]{};
      	\node (a5) at (1,1) [vertexL]{};
      	\node (a6) at (1,4) [vertexL]{};
      	\node (a7) at (1,7) [vertexL]{};
      	
      	\draw[line width=0.04cm] (a2) to (a3);
      	\draw[line width=0.04cm] (a4) to (a5);
      	\draw[line width=0.04cm] (a5) to (a6);
      	\draw[line width=0.04cm] (a6) to (a7);
      	
      \end{tikzpicture} \\ \hline
      \end{tabular}
      \caption{On the left $C$ with two chords, on the right graphs $B_C$}\label{figC7'}
      \end{figure}
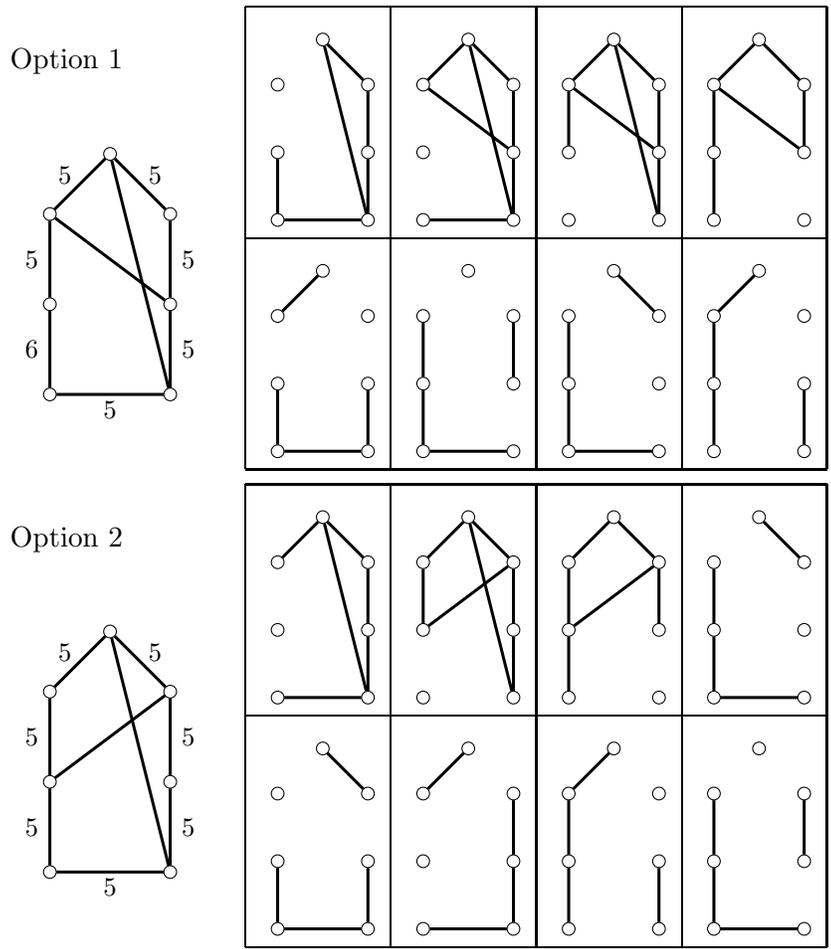
      
      \2
      
     Note that if \YC{$|V(G)|$ is even, then}
     a $7$-cycle cannot have three chords as the matching edge touching the vertex not adjacent to a chord would be a bridge. Thus, the above exhausts all possibilities \YC{when 
 $|V(G)$ is even}.

\YC{If  $|V(G)|$ is odd, then a 7-cycle can have three chords. Let $7$-cycle $C$ have three chords.
Then there is only one non-isomorphic possibility depicted on the left of Fig. \ref{figT}. Let $B_C$ be each bipartite subgraph in the table on the right of Fig. \ref{figT} with probability $1/8$. One can observe that every edge in $C$ will appear in $B_C$ with probability at least $5/8$, and each vertex in $C$ is not in $V(B_C)$ with probability at least $1/8$ and at most $1/4$. \smallQED
     }
     \YC{
     	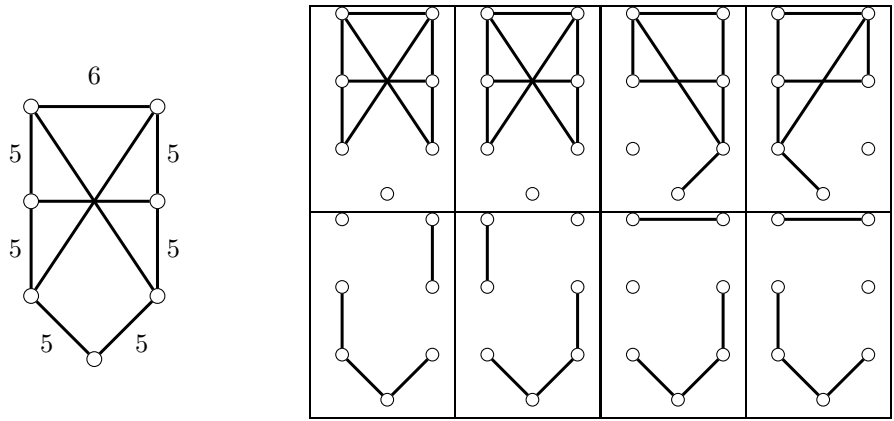
\begin{figure}
     	\begin{minipage}{4cm}
     		\tikzstyle{vertexL}=[circle,draw, minimum size=8pt, scale=0.7, inner sep=0.5pt]
     		\tikzstyle{vertexS}=[circle,draw, minimum size=8pt, scale=0.5, inner sep=0.5pt]
     		\begin{tikzpicture}[scale=0.42]
     			\node (a7) at (1,1) [vertexL]{};
     			\node (a6) at (1,4) [vertexL]{};
     			\node (a5) at (1,7) [vertexL]{};
     			\node (a4) at (5,7) [vertexL]{};
     			\node (a3) at (5,4) [vertexL]{};
     			\node (a2) at (5,1) [vertexL]{};
     			\node (a1) at (3,-1) [vertexL]{};
     			
     			\draw[line width=0.04cm] (a2) to (a3);  
     			\draw[line width=0.04cm] (a3) to (a4);  
     			\draw[line width=0.04cm] (a4) to (a5);  
     			\draw[line width=0.04cm] (a6) to (a7);   
     			\draw[line width=0.04cm] (a7) to (a1);

     			\draw[line width=0.04cm] (a1) to (a2);  
     			\draw[line width=0.04cm] (a5) to (a6);  
     			\draw[line width=0.04cm] (a2) to (a5);  
     			\draw[line width=0.04cm] (a4) to (a7);   
     			\draw[line width=0.04cm] (a3) to (a6);   
     			\node at (4.5,-0.5) {{\small 5}};
     			\node at (1.5,-0.5) {{\small 5}};
     			\node at (5.5,2.5) {{\small 5}};
     			\node at (0.5,2.5) {{\small 5}};
     			\node at (5.5,5.5) {{\small 5}};
     			\node at (0.5,5.5) {{\small 5}};
     			\node at (3,8) {{\small 6}};
     		\end{tikzpicture}
     	\end{minipage} 
     \begin{tabular}{|c|c|c|c|} \hline

     	\tikzstyle{vertexL}=[circle,draw, minimum size=8pt, scale=0.6, inner sep=0.5pt]
     	\begin{tikzpicture}[scale=0.3]
     		\node (a7) at (1,1) [vertexL]{};
     		\node (a6) at (1,4) [vertexL]{};
     		\node (a5) at (1,7) [vertexL]{};
     		\node (a4) at (5,7) [vertexL]{};
     		\node (a3) at (5,4) [vertexL]{};
     		\node (a2) at (5,1) [vertexL]{};
     		\node (a1) at (3,-1) [vertexL]{};

     		\draw[line width=0.04cm] (a2) to (a3);  
     		\draw[line width=0.04cm] (a3) to (a4);  
     		\draw[line width=0.04cm] (a4) to (a5);  
     		\draw[line width=0.04cm] (a6) to (a7);

     		\draw[line width=0.04cm] (a5) to (a6);  
     		\draw[line width=0.04cm] (a2) to (a5);  
     		\draw[line width=0.04cm] (a4) to (a7);   
     		\draw[line width=0.04cm] (a3) to (a6);   
     	\end{tikzpicture} & 
     	\tikzstyle{vertexL}=[circle,draw, minimum size=8pt, scale=0.6, inner sep=0.5pt]
     	\begin{tikzpicture}[scale=0.3]
     		\node (a7) at (1,1) [vertexL]{};
     		\node (a6) at (1,4) [vertexL]{};
     		\node (a5) at (1,7) [vertexL]{};
     		\node (a4) at (5,7) [vertexL]{};
     		\node (a3) at (5,4) [vertexL]{};
     		\node (a2) at (5,1) [vertexL]{};
     		\node (a1) at (3,-1) [vertexL]{};
     		
     		\draw[line width=0.04cm] (a2) to (a3);  
     		\draw[line width=0.04cm] (a3) to (a4);  
     		\draw[line width=0.04cm] (a4) to (a5);  
     		\draw[line width=0.04cm] (a6) to (a7);  
     		
     		\draw[line width=0.04cm] (a5) to (a6);  
     		\draw[line width=0.04cm] (a2) to (a5);  
     		\draw[line width=0.04cm] (a4) to (a7);   
     		\draw[line width=0.04cm] (a3) to (a6);   
     	\end{tikzpicture} &
     	\tikzstyle{vertexL}=[circle,draw, minimum size=8pt, scale=0.6, inner sep=0.5pt]
     	\begin{tikzpicture}[scale=0.3]
     		
     		\node (a7) at (1,1) [vertexL]{};
     		\node (a6) at (1,4) [vertexL]{};
     		\node (a5) at (1,7) [vertexL]{};
     		\node (a4) at (5,7) [vertexL]{};
     		\node (a3) at (5,4) [vertexL]{};
     		\node (a2) at (5,1) [vertexL]{};
     		\node (a1) at (3,-1) [vertexL]{};
     		
     		\draw[line width=0.04cm] (a2) to (a3);  
     		\draw[line width=0.04cm] (a3) to (a4);  
     		\draw[line width=0.04cm] (a4) to (a5);

     		\draw[line width=0.04cm] (a1) to (a2);  
     		\draw[line width=0.04cm] (a5) to (a6);  
     		\draw[line width=0.04cm] (a2) to (a5);  
     		\draw[line width=0.04cm] (a3) to (a6);   
     		
     	\end{tikzpicture} & 
     	\tikzstyle{vertexL}=[circle,draw, minimum size=8pt, scale=0.6, inner sep=0.5pt]
     	\begin{tikzpicture}[scale=0.3]
     		\node (a7) at (1,1) [vertexL]{};
     		\node (a6) at (1,4) [vertexL]{};
     		\node (a5) at (1,7) [vertexL]{};
     		\node (a4) at (5,7) [vertexL]{};
     		\node (a3) at (5,4) [vertexL]{};
     		\node (a2) at (5,1) [vertexL]{};
     		\node (a1) at (3,-1) [vertexL]{};
     		
     		\draw[line width=0.04cm] (a3) to (a4);  
     		\draw[line width=0.04cm] (a4) to (a5);  
     		\draw[line width=0.04cm] (a6) to (a7);   
     		\draw[line width=0.04cm] (a7) to (a1);   
     		
     		\draw[line width=0.04cm] (a5) to (a6);  
     		\draw[line width=0.04cm] (a4) to (a7);   
     		\draw[line width=0.04cm] (a3) to (a6);   
     	\end{tikzpicture} \\ \hline
     	\tikzstyle{vertexL}=[circle,draw, minimum size=8pt, scale=0.6, inner sep=0.5pt]
     	\begin{tikzpicture}[scale=0.3]
     		\node (a7) at (1,1) [vertexL]{};
     		\node (a6) at (1,4) [vertexL]{};
     		\node (a5) at (1,7) [vertexL]{};
     		\node (a4) at (5,7) [vertexL]{};
     		\node (a3) at (5,4) [vertexL]{};
     		\node (a2) at (5,1) [vertexL]{};
     		\node (a1) at (3,-1) [vertexL]{};
     		
     		\draw[line width=0.04cm] (a3) to (a4);  
     		\draw[line width=0.04cm] (a6) to (a7);   
     		\draw[line width=0.04cm] (a7) to (a1);

     		\draw[line width=0.04cm] (a1) to (a2);  
     		
     	\end{tikzpicture} & 
     	\tikzstyle{vertexL}=[circle,draw, minimum size=8pt, scale=0.6, inner sep=0.5pt]
     	\begin{tikzpicture}[scale=0.3]
     		
     		\node (a7) at (1,1) [vertexL]{};
     		\node (a6) at (1,4) [vertexL]{};
     		\node (a5) at (1,7) [vertexL]{};
     		\node (a4) at (5,7) [vertexL]{};
     		\node (a3) at (5,4) [vertexL]{};
     		\node (a2) at (5,1) [vertexL]{};
     		\node (a1) at (3,-1) [vertexL]{};
     		
     		\draw[line width=0.04cm] (a2) to (a3);  
     		\draw[line width=0.04cm] (a7) to (a1);

     		\draw[line width=0.04cm] (a1) to (a2);  
     		\draw[line width=0.04cm] (a5) to (a6);  
     	\end{tikzpicture}  &
     	\tikzstyle{vertexL}=[circle,draw, minimum size=8pt, scale=0.6, inner sep=0.5pt]
     	\begin{tikzpicture}[scale=0.3]
     		\node (a7) at (1,1) [vertexL]{};
     		\node (a6) at (1,4) [vertexL]{};
     		\node (a5) at (1,7) [vertexL]{};
     		\node (a4) at (5,7) [vertexL]{};
     		\node (a3) at (5,4) [vertexL]{};
     		\node (a2) at (5,1) [vertexL]{};
     		\node (a1) at (3,-1) [vertexL]{};
     		
     		\draw[line width=0.04cm] (a2) to (a3);  
     		\draw[line width=0.04cm] (a4) to (a5);  
     		\draw[line width=0.04cm] (a7) to (a1);

     		\draw[line width=0.04cm] (a1) to (a2);  
     		
     	\end{tikzpicture} & 
     	\tikzstyle{vertexL}=[circle,draw, minimum size=8pt, scale=0.6, inner sep=0.5pt]
     	\begin{tikzpicture}[scale=0.3]
     		\node (a7) at (1,1) [vertexL]{};
     		\node (a6) at (1,4) [vertexL]{};
     		\node (a5) at (1,7) [vertexL]{};
     		\node (a4) at (5,7) [vertexL]{};
     		\node (a3) at (5,4) [vertexL]{};
     		\node (a2) at (5,1) [vertexL]{};
     		\node (a1) at (3,-1) [vertexL]{};
     		
     		\draw[line width=0.04cm] (a6) to (a7);  
     		\draw[line width=0.04cm] (a4) to (a5);  
     		\draw[line width=0.04cm] (a7) to (a1);

     		\draw[line width=0.04cm] (a1) to (a2);  
     	\end{tikzpicture} 
     	
     	\\ \hline
     	
     \end{tabular}
     \caption{On the left $C$ with three chords, on the right graphs $B_C$}\label{figT}
     \end{figure}}
       
       \begin{claim}\label{cl:c11}
  	If $|C|= 11$, then there exists a random bipartite subgraph $B_C\in \mathcal{B}_b(H[V(C)])$ such that $\mathbb{E}(w(B_C))\geq \frac{5}{8}w(C)$ and every vertex, which is not incident to a chord of $C$, is not in the $V(B_C)$ with probability at least $1/8$ and at most $1/4$.
       \end{claim}
       {\bf The proof of Claim \ref{cl:c11}. } Let $C=c_0c_1c_2 \ldots c_{10} c_0$. First, assume that some chord of $C$ in $H$ creates a $5$-cycle when added to $C$. Assume without loss of generality that $c_0 c_4$ is a chord of $C$. We may also assume without loss of generality that $c_5c_9$ is not a chord, since if it is we can consider $c_5c_9$ instead of $c_0c_4$, and if $c_{10} c_3$ is now a chord then $c_4 c_8$ is not a chord, as $c_0 c_4$ was a chord (and we could have considered $c_{10} c_3$). So, in this case, neither $c_5c_9$ nor $c_4 c_8$ is a chord (as $c_0 c_4$ is a chord of $C$). We depict $C$ on the left of Fig. \ref{figC11} (dotted lines mean that they are not chords). 
       
Let $B_C$ be each bipartite subgraph in the table on the right of Fig. \ref{figC11} with probability $1/8$. Every edge in $C$ will appear in $B_C$ with probability at least $5/8$ as the numbers next to each edge on the left of Fig. \ref{figC11} indicate how many subgraphs contain this edge. Also, we note that each vertex in $V(C)\setminus \{c_0,c_4\}$ is not in $V(B_C)$ in at least one, and at most two, of the subgraphs in the table and thus with probability at least $1/8$ and at most $1/4$. 
       
       \2
       
       	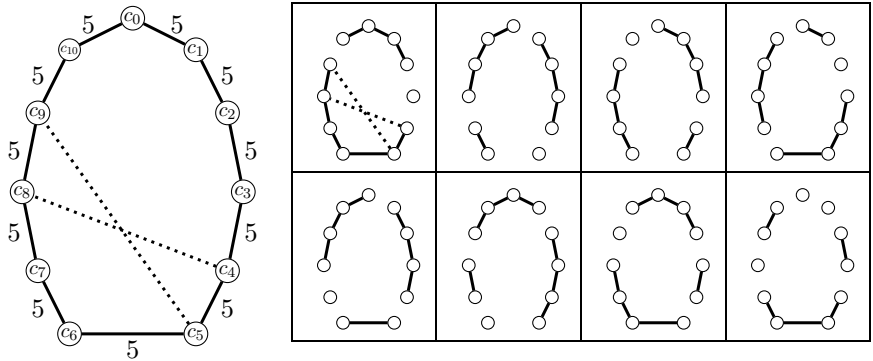
\begin{figure}
       	\begin{minipage}{4cm}
       		\tikzstyle{vertexL}=[circle,draw, minimum size=8pt, scale=0.7, inner sep=0.5pt]
       		\tikzstyle{vertexS}=[circle,draw, minimum size=8pt, scale=0.5, inner sep=0.5pt]
       		\begin{tikzpicture}[scale=0.42]
       			\node at (-0.5,12) {  };
       			\node (a0) at (4,11) [vertexL]{$c_0$};
       			\node (a1) at (6,10) [vertexL]{$c_1$};
       			\node (a2) at (7,8) [vertexL]{$c_2$};
       			\node (a3) at (7.5,5.5) [vertexL]{$c_3$};
       			\node (a4) at (7,3) [vertexL]{$c_4$};
       			\node (a5) at (6,1) [vertexL]{$c_5$};
       			\node (a6) at (2,1) [vertexL]{$c_6$};
       			\node (a7) at (1,3) [vertexL]{$c_7$};
       			\node (a8) at (0.5,5.5) [vertexL]{$c_8$};
       			\node (a9) at (1,8) [vertexL]{$c_9$};
       			\node (a10) at (2,10) [vertexS]{$c_{10}$};
       			\draw[line width=0.04cm] (a0) to (a1);   \node at (5.4,10.8) {{\small 5}};
       			\draw[line width=0.04cm] (a1) to (a2);   \node at (7,9.2) {{\small 5}};
       			\draw[line width=0.04cm] (a2) to (a3);   \node at (7.75,6.8) {{\small 5}};
       			\draw[line width=0.04cm] (a3) to (a4);   \node at (7.75,4.2) {{\small 5}};
       			\draw[line width=0.04cm] (a4) to (a5);   \node at (7,1.8) {{\small 5}};
       			\draw[line width=0.04cm] (a5) to (a6);   \node at (4,0.5) {{\small 5}};
       			\draw[line width=0.04cm] (a6) to (a7);   \node at (1,1.8) {{\small 5}};   
       			\draw[line width=0.04cm] (a7) to (a8);   \node at (0.25,4.2) {{\small 5}};
       			\draw[line width=0.04cm] (a8) to (a9);   \node at (0.25,6.8) {{\small 5}};
       			\draw[line width=0.04cm] (a9) to (a10);  \node at (1,9.2) {{\small 5}}; 
       			\draw[line width=0.04cm] (a10) to (a0);  \node at (2.6,10.8) {{\small 5}};
       			
       			\draw[dotted, line width=0.04cm] (a5) to (a9);
       			\draw[dotted, line width=0.04cm] (a4) to (a8);
       		\end{tikzpicture}
       	\end{minipage} 
       \begin{tabular}{|c|c|c|c|} \hline
       	\tikzstyle{vertexL}=[circle,draw, minimum size=8pt, scale=0.6, inner sep=0.5pt]
       	\begin{tikzpicture}[scale=0.17]
       		\node at (4,12) {  }; 
       		\node (a0) at (4,11) [vertexL]{};
       		\node (a1) at (6,10) [vertexL]{}; 
       		\node (a2) at (7,8) [vertexL]{};
       		\node (a3) at (7.5,5.5) [vertexL]{};
       		\node (a4) at (7,3) [vertexL]{};
       		\node (a5) at (6,1) [vertexL]{};
       		\node (a6) at (2,1) [vertexL]{};
       		\node (a7) at (1,3) [vertexL]{};
       		\node (a8) at (0.5,5.5) [vertexL]{};
       		\node (a9) at (1,8) [vertexL]{};
       		\node (a10) at (2,10) [vertexL]{};
       		\draw[line width=0.04cm] (a0) to (a1);  
       		\draw[line width=0.04cm] (a1) to (a2); 
       		\draw[line width=0.04cm] (a4) to (a5); 
       		\draw[line width=0.04cm] (a5) to (a6); 
       		\draw[line width=0.04cm] (a6) to (a7); 
       		\draw[line width=0.04cm] (a7) to (a8); 
       		\draw[line width=0.04cm] (a8) to (a9); 
       		\draw[line width=0.04cm] (a10) to (a0); 
       		\draw[dotted, line width=0.04cm] (a5) to (a9);
       		\draw[dotted, line width=0.04cm] (a4) to (a8);
       	\end{tikzpicture} & 
       	\tikzstyle{vertexL}=[circle,draw, minimum size=8pt, scale=0.6, inner sep=0.5pt]
       	\begin{tikzpicture}[scale=0.17]
       		\node at (4,12) {  };
       		\node (a0) at (4,11) [vertexL]{};
       		\node (a1) at (6,10) [vertexL]{};
       		\node (a2) at (7,8) [vertexL]{};
       		\node (a3) at (7.5,5.5) [vertexL]{};
       		\node (a4) at (7,3) [vertexL]{};
       		\node (a5) at (6,1) [vertexL]{};
       		\node (a6) at (2,1) [vertexL]{};
       		\node (a7) at (1,3) [vertexL]{};
       		\node (a8) at (0.5,5.5) [vertexL]{};
       		\node (a9) at (1,8) [vertexL]{};
       		\node (a10) at (2,10) [vertexL]{};
       		\draw[line width=0.04cm] (a1) to (a2);
       		\draw[line width=0.04cm] (a2) to (a3);
       		\draw[line width=0.04cm] (a3) to (a4);
       		\draw[line width=0.04cm] (a6) to (a7);
       		\draw[line width=0.04cm] (a8) to (a9);
       		\draw[line width=0.04cm] (a9) to (a10);
       		\draw[line width=0.04cm] (a10) to (a0);
       	\end{tikzpicture} & 
       	\tikzstyle{vertexL}=[circle,draw, minimum size=8pt, scale=0.6, inner sep=0.5pt]
       	\begin{tikzpicture}[scale=0.17]
       		\node at (4,12) {  };
       		\node (a0) at (4,11) [vertexL]{};
       		\node (a1) at (6,10) [vertexL]{};
       		\node (a2) at (7,8) [vertexL]{};
       		\node (a3) at (7.5,5.5) [vertexL]{};
       		\node (a4) at (7,3) [vertexL]{};
       		\node (a5) at (6,1) [vertexL]{};
       		\node (a6) at (2,1) [vertexL]{};
       		\node (a7) at (1,3) [vertexL]{};
       		\node (a8) at (0.5,5.5) [vertexL]{};
       		\node (a9) at (1,8) [vertexL]{};
       		\node (a10) at (2,10) [vertexL]{};
       		\draw[line width=0.04cm] (a0) to (a1); 
       		\draw[line width=0.04cm] (a1) to (a2);
       		\draw[line width=0.04cm] (a2) to (a3);
       		\draw[line width=0.04cm] (a4) to (a5);
       		\draw[line width=0.04cm] (a6) to (a7);
       		\draw[line width=0.04cm] (a7) to (a8);
       		\draw[line width=0.04cm] (a8) to (a9);
       	\end{tikzpicture} & 
       	\tikzstyle{vertexL}=[circle,draw, minimum size=8pt, scale=0.6, inner sep=0.5pt]
       	\begin{tikzpicture}[scale=0.17]
       		\node at (4,12) {  };
       		\node (a0) at (4,11) [vertexL]{};
       		\node (a1) at (6,10) [vertexL]{};
       		\node (a2) at (7,8) [vertexL]{};
       		\node (a3) at (7.5,5.5) [vertexL]{};
       		\node (a4) at (7,3) [vertexL]{};
       		\node (a5) at (6,1) [vertexL]{};
       		\node (a6) at (2,1) [vertexL]{};
       		\node (a7) at (1,3) [vertexL]{};
       		\node (a8) at (0.5,5.5) [vertexL]{};
       		\node (a9) at (1,8) [vertexL]{};
       		\node (a10) at (2,10) [vertexL]{};
       		\draw[line width=0.04cm] (a0) to (a1); 
       		\draw[line width=0.04cm] (a3) to (a4);
       		\draw[line width=0.04cm] (a4) to (a5);
       		\draw[line width=0.04cm] (a5) to (a6);
       		\draw[line width=0.04cm] (a7) to (a8);
       		\draw[line width=0.04cm] (a8) to (a9);
       		\draw[line width=0.04cm] (a9) to (a10);
       	\end{tikzpicture} \\ \hline
       	\tikzstyle{vertexL}=[circle,draw, minimum size=8pt, scale=0.6, inner sep=0.5pt]
       	\begin{tikzpicture}[scale=0.17]
       		\node at (4,12) {  };
       		\node (a0) at (4,11) [vertexL]{};
       		\node (a1) at (6,10) [vertexL]{};
       		\node (a2) at (7,8) [vertexL]{};
       		\node (a3) at (7.5,5.5) [vertexL]{};
       		\node (a4) at (7,3) [vertexL]{};
       		\node (a5) at (6,1) [vertexL]{};
       		\node (a6) at (2,1) [vertexL]{};
       		\node (a7) at (1,3) [vertexL]{};
       		\node (a8) at (0.5,5.5) [vertexL]{};
       		\node (a9) at (1,8) [vertexL]{};
       		\node (a10) at (2,10) [vertexL]{};
       		\draw[line width=0.04cm] (a1) to (a2);
       		\draw[line width=0.04cm] (a2) to (a3);
       		\draw[line width=0.04cm] (a3) to (a4);
       		\draw[line width=0.04cm] (a5) to (a6);
       		\draw[line width=0.04cm] (a8) to (a9);
       		\draw[line width=0.04cm] (a9) to (a10);
       		\draw[line width=0.04cm] (a10) to (a0);
       	\end{tikzpicture} & 
       	\tikzstyle{vertexL}=[circle,draw, minimum size=8pt, scale=0.6, inner sep=0.5pt]
       	\begin{tikzpicture}[scale=0.17]
       		\node at (4,12) {  };
       		\node (a0) at (4,11) [vertexL]{};
       		\node (a1) at (6,10) [vertexL]{};
       		\node (a2) at (7,8) [vertexL]{};
       		\node (a3) at (7.5,5.5) [vertexL]{};
       		\node (a4) at (7,3) [vertexL]{};
       		\node (a5) at (6,1) [vertexL]{};
       		\node (a6) at (2,1) [vertexL]{};
       		\node (a7) at (1,3) [vertexL]{};
       		\node (a8) at (0.5,5.5) [vertexL]{};
       		\node (a9) at (1,8) [vertexL]{};
       		\node (a10) at (2,10) [vertexL]{};
       		\draw[line width=0.04cm] (a0) to (a1);
       		\draw[line width=0.04cm] (a2) to (a3);
       		\draw[line width=0.04cm] (a3) to (a4);
       		\draw[line width=0.04cm] (a4) to (a5);
       		\draw[line width=0.04cm] (a7) to (a8);
       		\draw[line width=0.04cm] (a9) to (a10);
       		\draw[line width=0.04cm] (a10) to (a0);
       	\end{tikzpicture} & 
       	\tikzstyle{vertexL}=[circle,draw, minimum size=8pt, scale=0.6, inner sep=0.5pt]
       	\begin{tikzpicture}[scale=0.17]
       		\node at (4,12) {  };
       		\node (a0) at (4,11) [vertexL]{};
       		\node (a1) at (6,10) [vertexL]{};
       		\node (a2) at (7,8) [vertexL]{};
       		\node (a3) at (7.5,5.5) [vertexL]{};
       		\node (a4) at (7,3) [vertexL]{};
       		\node (a5) at (6,1) [vertexL]{};
       		\node (a6) at (2,1) [vertexL]{};
       		\node (a7) at (1,3) [vertexL]{};
       		\node (a8) at (0.5,5.5) [vertexL]{};
       		\node (a9) at (1,8) [vertexL]{};
       		\node (a10) at (2,10) [vertexL]{};
       		\draw[line width=0.04cm] (a0) to (a1);
       		\draw[line width=0.04cm] (a1) to (a2);
       		\draw[line width=0.04cm] (a3) to (a4);
       		\draw[line width=0.04cm] (a5) to (a6);
       		\draw[line width=0.04cm] (a6) to (a7);
       		\draw[line width=0.04cm] (a7) to (a8);
       		\draw[line width=0.04cm] (a10) to (a0);
       	\end{tikzpicture} & 
       	\tikzstyle{vertexL}=[circle,draw, minimum size=8pt, scale=0.6, inner sep=0.5pt]
       	\begin{tikzpicture}[scale=0.17]
       		\node at (4,12) {  };
       		\node (a0) at (4,11) [vertexL]{};
       		\node (a1) at (6,10) [vertexL]{};
       		\node (a2) at (7,8) [vertexL]{};
       		\node (a3) at (7.5,5.5) [vertexL]{};
       		\node (a4) at (7,3) [vertexL]{};
       		\node (a5) at (6,1) [vertexL]{};
       		\node (a6) at (2,1) [vertexL]{};
       		\node (a7) at (1,3) [vertexL]{};
       		\node (a8) at (0.5,5.5) [vertexL]{};
       		\node (a9) at (1,8) [vertexL]{};
       		\node (a10) at (2,10) [vertexL]{};
       		\draw[line width=0.04cm] (a2) to (a3);
       		\draw[line width=0.04cm] (a4) to (a5);
       		\draw[line width=0.04cm] (a5) to (a6);
       		\draw[line width=0.04cm] (a6) to (a7);
       		\draw[line width=0.04cm] (a9) to (a10);
       	\end{tikzpicture} \\ \hline
       \end{tabular}
       \caption{On the left $C$ with two non-chords, on the right graphs $B_C$}\label{figC11}
       \end{figure}
       
If no chord of $C$ in $H$ creates a $5$-cycle when added to $C$, then every six consecutive vertices will induce a bipartite subgraph in $\mathcal{B}_b(H[V(C)])$. Therefore, we let $B_C$ be each bipartite subgraph in the table on the right of Fig. \ref{figC11'} with probability $1/8$. One can observe that every edge in $C$ will appear in $B_C$ with probability at least $5/8$, and each vertex in $C$ is not in $V(B_C)$ with probability at least $1/8$ and at most $1/4$. 
               \smallQED
       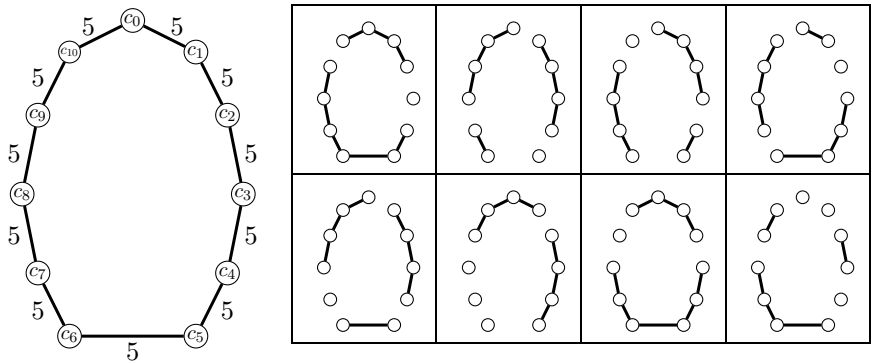
\begin{figure}
       	\begin{minipage}{4cm}
       	\tikzstyle{vertexL}=[circle,draw, minimum size=8pt, scale=0.7, inner sep=0.5pt]
       	\tikzstyle{vertexS}=[circle,draw, minimum size=8pt, scale=0.5, inner sep=0.5pt]
       	\begin{tikzpicture}[scale=0.42]
       		\node at (-0.5,12) {  };
       		\node (a0) at (4,11) [vertexL]{$c_0$};
       		\node (a1) at (6,10) [vertexL]{$c_1$};
       		\node (a2) at (7,8) [vertexL]{$c_2$};
       		\node (a3) at (7.5,5.5) [vertexL]{$c_3$};
       		\node (a4) at (7,3) [vertexL]{$c_4$};
       		\node (a5) at (6,1) [vertexL]{$c_5$};
       		\node (a6) at (2,1) [vertexL]{$c_6$};
       		\node (a7) at (1,3) [vertexL]{$c_7$};
       		\node (a8) at (0.5,5.5) [vertexL]{$c_8$};
       		\node (a9) at (1,8) [vertexL]{$c_9$};
       		\node (a10) at (2,10) [vertexS]{$c_{10}$};
       		\draw[line width=0.04cm] (a0) to (a1);   \node at (5.4,10.8) {{\small 5}};
       		\draw[line width=0.04cm] (a1) to (a2);   \node at (7,9.2) {{\small 5}};
       		\draw[line width=0.04cm] (a2) to (a3);   \node at (7.75,6.8) {{\small 5}};
       		\draw[line width=0.04cm] (a3) to (a4);   \node at (7.75,4.2) {{\small 5}};
       		\draw[line width=0.04cm] (a4) to (a5);   \node at (7,1.8) {{\small 5}};
       		\draw[line width=0.04cm] (a5) to (a6);   \node at (4,0.5) {{\small 5}};
       		\draw[line width=0.04cm] (a6) to (a7);   \node at (1,1.8) {{\small 5}};   
       		\draw[line width=0.04cm] (a7) to (a8);   \node at (0.25,4.2) {{\small 5}};
       		\draw[line width=0.04cm] (a8) to (a9);   \node at (0.25,6.8) {{\small 5}};
       		\draw[line width=0.04cm] (a9) to (a10);  \node at (1,9.2) {{\small 5}}; 
       		\draw[line width=0.04cm] (a10) to (a0);  \node at (2.6,10.8) {{\small 5}};
       		
       	\end{tikzpicture}
       \end{minipage} 
       \begin{tabular}{|c|c|c|c|} \hline
       \tikzstyle{vertexL}=[circle,draw, minimum size=8pt, scale=0.6, inner sep=0.5pt]
       \begin{tikzpicture}[scale=0.17]
       	\node at (4,12) {  }; 
       	\node (a0) at (4,11) [vertexL]{};
       	\node (a1) at (6,10) [vertexL]{}; 
       	\node (a2) at (7,8) [vertexL]{};
       	\node (a3) at (7.5,5.5) [vertexL]{};
       	\node (a4) at (7,3) [vertexL]{};
       	\node (a5) at (6,1) [vertexL]{};
       	\node (a6) at (2,1) [vertexL]{};
       	\node (a7) at (1,3) [vertexL]{};
       	\node (a8) at (0.5,5.5) [vertexL]{};
       	\node (a9) at (1,8) [vertexL]{};
       	\node (a10) at (2,10) [vertexL]{};
       	\draw[line width=0.04cm] (a0) to (a1);  
       	\draw[line width=0.04cm] (a1) to (a2); 
       	\draw[line width=0.04cm] (a4) to (a5); 
       	\draw[line width=0.04cm] (a5) to (a6); 
       	\draw[line width=0.04cm] (a6) to (a7); 
       	\draw[line width=0.04cm] (a7) to (a8); 
       	\draw[line width=0.04cm] (a8) to (a9); 
       	\draw[line width=0.04cm] (a10) to (a0); 
       \end{tikzpicture} & 
       \tikzstyle{vertexL}=[circle,draw, minimum size=8pt, scale=0.6, inner sep=0.5pt]
       \begin{tikzpicture}[scale=0.17]
       	\node at (4,12) {  };
       	\node (a0) at (4,11) [vertexL]{};
       	\node (a1) at (6,10) [vertexL]{};
       	\node (a2) at (7,8) [vertexL]{};
       	\node (a3) at (7.5,5.5) [vertexL]{};
       	\node (a4) at (7,3) [vertexL]{};
       	\node (a5) at (6,1) [vertexL]{};
       	\node (a6) at (2,1) [vertexL]{};
       	\node (a7) at (1,3) [vertexL]{};
       	\node (a8) at (0.5,5.5) [vertexL]{};
       	\node (a9) at (1,8) [vertexL]{};
       	\node (a10) at (2,10) [vertexL]{};
       	\draw[line width=0.04cm] (a1) to (a2);
       	\draw[line width=0.04cm] (a2) to (a3);
       	\draw[line width=0.04cm] (a3) to (a4);
       	\draw[line width=0.04cm] (a6) to (a7);
       	\draw[line width=0.04cm] (a8) to (a9);
       	\draw[line width=0.04cm] (a9) to (a10);
       	\draw[line width=0.04cm] (a10) to (a0);
       \end{tikzpicture} & 
       \tikzstyle{vertexL}=[circle,draw, minimum size=8pt, scale=0.6, inner sep=0.5pt]
       \begin{tikzpicture}[scale=0.17]
       	\node at (4,12) {  };
       	\node (a0) at (4,11) [vertexL]{};
       	\node (a1) at (6,10) [vertexL]{};
       	\node (a2) at (7,8) [vertexL]{};
       	\node (a3) at (7.5,5.5) [vertexL]{};
       	\node (a4) at (7,3) [vertexL]{};
       	\node (a5) at (6,1) [vertexL]{};
       	\node (a6) at (2,1) [vertexL]{};
       	\node (a7) at (1,3) [vertexL]{};
       	\node (a8) at (0.5,5.5) [vertexL]{};
       	\node (a9) at (1,8) [vertexL]{};
       	\node (a10) at (2,10) [vertexL]{};
       	\draw[line width=0.04cm] (a0) to (a1); 
       	\draw[line width=0.04cm] (a1) to (a2);
       	\draw[line width=0.04cm] (a2) to (a3);
       	\draw[line width=0.04cm] (a4) to (a5);
       	\draw[line width=0.04cm] (a6) to (a7);
       	\draw[line width=0.04cm] (a7) to (a8);
       	\draw[line width=0.04cm] (a8) to (a9);
       \end{tikzpicture} & 
       \tikzstyle{vertexL}=[circle,draw, minimum size=8pt, scale=0.6, inner sep=0.5pt]
       \begin{tikzpicture}[scale=0.17]
       	\node at (4,12) {  };
       	\node (a0) at (4,11) [vertexL]{};
       	\node (a1) at (6,10) [vertexL]{};
       	\node (a2) at (7,8) [vertexL]{};
       	\node (a3) at (7.5,5.5) [vertexL]{};
       	\node (a4) at (7,3) [vertexL]{};
       	\node (a5) at (6,1) [vertexL]{};
       	\node (a6) at (2,1) [vertexL]{};
       	\node (a7) at (1,3) [vertexL]{};
       	\node (a8) at (0.5,5.5) [vertexL]{};
       	\node (a9) at (1,8) [vertexL]{};
       	\node (a10) at (2,10) [vertexL]{};
       	\draw[line width=0.04cm] (a0) to (a1); 
       	\draw[line width=0.04cm] (a3) to (a4);
       	\draw[line width=0.04cm] (a4) to (a5);
       	\draw[line width=0.04cm] (a5) to (a6);
       	\draw[line width=0.04cm] (a7) to (a8);
       	\draw[line width=0.04cm] (a8) to (a9);
       	\draw[line width=0.04cm] (a9) to (a10);
       \end{tikzpicture} \\ \hline
       \tikzstyle{vertexL}=[circle,draw, minimum size=8pt, scale=0.6, inner sep=0.5pt]
       \begin{tikzpicture}[scale=0.17]
       	\node at (4,12) {  };
       	\node (a0) at (4,11) [vertexL]{};
       	\node (a1) at (6,10) [vertexL]{};
       	\node (a2) at (7,8) [vertexL]{};
       	\node (a3) at (7.5,5.5) [vertexL]{};
       	\node (a4) at (7,3) [vertexL]{};
       	\node (a5) at (6,1) [vertexL]{};
       	\node (a6) at (2,1) [vertexL]{};
       	\node (a7) at (1,3) [vertexL]{};
       	\node (a8) at (0.5,5.5) [vertexL]{};
       	\node (a9) at (1,8) [vertexL]{};
       	\node (a10) at (2,10) [vertexL]{};
       	\draw[line width=0.04cm] (a1) to (a2);
       	\draw[line width=0.04cm] (a2) to (a3);
       	\draw[line width=0.04cm] (a3) to (a4);
       	\draw[line width=0.04cm] (a5) to (a6);
       	\draw[line width=0.04cm] (a8) to (a9);
       	\draw[line width=0.04cm] (a9) to (a10);
       	\draw[line width=0.04cm] (a10) to (a0);
       \end{tikzpicture} & 
       \tikzstyle{vertexL}=[circle,draw, minimum size=8pt, scale=0.6, inner sep=0.5pt]
       \begin{tikzpicture}[scale=0.17]
       	\node at (4,12) {  };
       	\node (a0) at (4,11) [vertexL]{};
       	\node (a1) at (6,10) [vertexL]{};
       	\node (a2) at (7,8) [vertexL]{};
       	\node (a3) at (7.5,5.5) [vertexL]{};
       	\node (a4) at (7,3) [vertexL]{};
       	\node (a5) at (6,1) [vertexL]{};
       	\node (a6) at (2,1) [vertexL]{};
       	\node (a7) at (1,3) [vertexL]{};
       	\node (a8) at (0.5,5.5) [vertexL]{};
       	\node (a9) at (1,8) [vertexL]{};
       	\node (a10) at (2,10) [vertexL]{};
       	\draw[line width=0.04cm] (a0) to (a1);
       	\draw[line width=0.04cm] (a2) to (a3);
       	\draw[line width=0.04cm] (a3) to (a4);
       	\draw[line width=0.04cm] (a4) to (a5);
       	\draw[line width=0.04cm] (a9) to (a10);
       	\draw[line width=0.04cm] (a10) to (a0);
       \end{tikzpicture} & 
       \tikzstyle{vertexL}=[circle,draw, minimum size=8pt, scale=0.6, inner sep=0.5pt]
       \begin{tikzpicture}[scale=0.17]
       	\node at (4,12) {  };
       	\node (a0) at (4,11) [vertexL]{};
       	\node (a1) at (6,10) [vertexL]{};
       	\node (a2) at (7,8) [vertexL]{};
       	\node (a3) at (7.5,5.5) [vertexL]{};
       	\node (a4) at (7,3) [vertexL]{};
       	\node (a5) at (6,1) [vertexL]{};
       	\node (a6) at (2,1) [vertexL]{};
       	\node (a7) at (1,3) [vertexL]{};
       	\node (a8) at (0.5,5.5) [vertexL]{};
       	\node (a9) at (1,8) [vertexL]{};
       	\node (a10) at (2,10) [vertexL]{};
       	\draw[line width=0.04cm] (a0) to (a1);
       	\draw[line width=0.04cm] (a1) to (a2);
       	\draw[line width=0.04cm] (a3) to (a4);
       	  \draw[line width=0.04cm] (a4) to (a5);
       	\draw[line width=0.04cm] (a5) to (a6);
       	\draw[line width=0.04cm] (a6) to (a7);
       	\draw[line width=0.04cm] (a7) to (a8);
       	\draw[line width=0.04cm] (a10) to (a0);
       \end{tikzpicture} & 
       \tikzstyle{vertexL}=[circle,draw, minimum size=8pt, scale=0.6, inner sep=0.5pt]
       \begin{tikzpicture}[scale=0.17]
       	\node at (4,12) {  };
       	\node (a0) at (4,11) [vertexL]{};
       	\node (a1) at (6,10) [vertexL]{};
       	\node (a2) at (7,8) [vertexL]{};
       	\node (a3) at (7.5,5.5) [vertexL]{};
       	\node (a4) at (7,3) [vertexL]{};
       	\node (a5) at (6,1) [vertexL]{};
       	\node (a6) at (2,1) [vertexL]{};
       	\node (a7) at (1,3) [vertexL]{};
       	\node (a8) at (0.5,5.5) [vertexL]{};
       	\node (a9) at (1,8) [vertexL]{};
       	\node (a10) at (2,10) [vertexL]{};
       	\draw[line width=0.04cm] (a2) to (a3);
       	\draw[line width=0.04cm] (a5) to (a6);
       	\draw[line width=0.04cm] (a6) to (a7);
       	 \draw[line width=0.04cm] (a7) to (a8);
       	\draw[line width=0.04cm] (a9) to (a10);
       \end{tikzpicture} \\ \hline
       \end{tabular}
             	\caption{On the left $C$ where no chord forms a 5-cycle, on the right graphs $B_C$}\label{figC11'}
     \end{figure}

        \begin{claim}\label{cl:pathcase}
       	If $P$ is the path in $H-M$, then there exists a random bipartite subgraph $B_P\in \mathcal{B}_b(H[V(P)])$ such that $\mathbb{E}(w(B_P))\geq \frac{5}{8}w(P)$ and every vertex, which is not incident to a chord of $P$ is not in the $V(B_P)$ with probability at least $1/8$.
       \end{claim}
       {\bf The proof of Claim \ref{cl:pathcase}.} Let $P=p_1p_2\dots p_n$ be the path. Note that by Claim \ref{cl:pathp}, $p_2p_n\in E(H)$ and $n\geq 5$.  We consider the following cases.
       
       {\bf Case 1:} $n=5$. 
       $P=p_1 p_2 p_3 p_4 p_5$ and $p_2 p_5 \in E(H)$. Let $B_P$ be each bipartite subgraphs in the following table with the given probabilities.
       
       \begin{center}
       	\begin{tabular}{|c|c|} \hline
       		$Prob$ & Subgraph \\ \hline \hline
       		$3/8$ & \tikzstyle{vertexL}=[circle,draw, minimum size=8pt, scale=0.6, inner sep=0.5pt]
       		\begin{tikzpicture}[scale=0.2]
       			\node at (1,2) {  }; 
       			\node (p1) at (1,1) [vertexL]{$p_1$}; 
       			\node (p2) at (4,1) [vertexL]{$p_2$}; 
       			\node (p3) at (7,1) [vertexL]{$p_3$};
       			\node (p4) at (10,1) [vertexL]{$p_4$}; 
       			\node (p5) at (13,1) [vertexL]{$p_5$}; 
       			\draw[line width=0.04cm] (p2) -- (p3);
       			\draw[line width=0.04cm] (p3) -- (p4);
       			\draw[line width=0.04cm] (p4) -- (p5);
       		\end{tikzpicture} \\ \hline
       		$2/8$ & \tikzstyle{vertexL}=[circle,draw, minimum size=8pt, scale=0.6, inner sep=0.5pt]
       		\begin{tikzpicture}[scale=0.2]
       			\node at (1,2) {  }; 
       			\node (p1) at (1,1) [vertexL]{$p_1$};
       			\node (p2) at (4,1) [vertexL]{$p_2$};
       			\node (p3) at (7,1) [vertexL]{$p_3$};
       			\node (p4) at (10,1) [vertexL]{$p_4$};
       			\node (p5) at (13,1) [vertexL]{$p_5$};
       			\draw[line width=0.04cm] (p1) -- (p2);
       			\draw[line width=0.04cm] (p2) -- (p3);
       			\draw[line width=0.04cm] (p3) -- (p4);
       		\end{tikzpicture} \\ \hline
       		$2/8$ & \tikzstyle{vertexL}=[circle,draw, minimum size=8pt, scale=0.6, inner sep=0.5pt]
       		\begin{tikzpicture}[scale=0.2]
       			\node at (1,2) {  }; 
       			\node (p1) at (1,1) [vertexL]{$p_1$};
       			\node (p2) at (4,1) [vertexL]{$p_2$};
       			\node (p3) at (7,1) [vertexL]{$p_3$};
       			\node (p4) at (10,1) [vertexL]{$p_4$};
       			\node (p5) at (13,1) [vertexL]{$p_5$};
       			\draw[line width=0.04cm] (p1) -- (p2);
       			\draw[line width=0.04cm] (p4) -- (p5);
       		\end{tikzpicture} \\ \hline
       		$1/8$ & \tikzstyle{vertexL}=[circle,draw, minimum size=8pt, scale=0.6, inner sep=0.5pt]
       		\begin{tikzpicture}[scale=0.2]
       			\node at (1,2) {  }; 
       			\node (p1) at (1,1) [vertexL]{$p_1$};
       			\node (p2) at (4,1) [vertexL]{$p_2$};
       			\node (p3) at (7,1) [vertexL]{$p_3$};
       			\node (p4) at (10,1) [vertexL]{$p_4$};
       			\node (p5) at (13,1) [vertexL]{$p_5$};
       			\draw[line width=0.04cm] (p1) -- (p2);
       		\end{tikzpicture} \\ \hline
       	\end{tabular}
       \end{center}
       
       Note that every edge of $P$ is in $B_C$ with probability at least $5/8$ and every vertex except $p_2$ is not in $V(B_C)$ vertex with probability at least $1/8$ (and $p_2$ is incident with the edge $p_2p_5$).
       
In the following cases, we consider when $p_1 p_{n-1}$ is an edge in $E(H)$ and when it is not, separately. Note that if $p_1 p_{n-1} \not\in E(H)$, then the subgraph $H[\{p_{n-2},p_{n-1},p_n,p_1\}]$ is a bipartite subgraph in $\mathcal{B}_b(H)$ with partite sets $\{p_{n-2},p_n\}$ and $\{p_1,p_{n-1}\}$. 
        
       {\bf Case 2:} $n=4k$, where $k \geq 2$.

       First consider the case when $p_1 p_{n-1} \not\in E(H)$. Let $B_P$ be the bipartite subgraphs induced by the following sets (with the given probabilities).
       
       \begin{itemize}
       	\item Use $\{p_1,p_2,p_3,p_4\}, \{p_5,p_6,p_7,p_8\}, \ldots , \{p_{n-3},p_{n-2},p_{n-1},p_n\}$ with probability $2/8$.
       	\item Use $ \{p_1,p_2\}, \{p_3,p_4,p_5,p_6\}, \{p_7,p_8,p_9,p_{10}\}, \ldots ,$ $\{p_{n-5},p_{n-4},p_{n-3},p_{n-2}\}, \{p_{n-1},p_{n}\}$ with probability $2/8$.
       	\item Use $ \{p_2,p_3,p_4,p_5\}, \{p_6,p_7,p_8,p_{9}\}, \ldots , \{p_{n-6},p_{n-5},p_{n-4},p_{n-3}\}, \{p_{n-2},p_{n-1}\}$ with probability $1/8$.
       	\item Use $ \{p_2,p_3\}, \{p_4,p_5,p_6,p_7\}, \{p_8,p_9,p_{10},p_{11}\}, \ldots , \{p_{n-4},p_{n-3},p_{n-2},p_{n-1}\}$ with probability $1/8$.
       	\item Use $ \{p_1,p_2\}$ with probability $1/8$.
       	\item Use $ \{p_2,p_3\}, \{p_4,p_5\}, \ldots , \{p_{n-4},p_{n-3}\}, \{p_{n-2},p_{n-1},p_n,p_1\}$ with probability $1/8$.
       \end{itemize}
       
Now, every edge of $P$ is in $B_C$ with probability at least $5/8$, and every vertex except $p_2$ is not in $V(B_C)$ with probability at least $1/8$ (and $p_2$ is incident with the edge $p_2 p_n$). We illustrate the above probability space for $n=8$ below (in the last figure the A's indicate one partite set and the B's the other partite set of the subgraph induced by $\{p_1,p_6,p_7,p_8\}$). 

       \begin{center}
       	\begin{tabular}{|c|c|} \hline
       		$Prob$ & Subgraph \\ \hline \hline
       		$2/8$ & \tikzstyle{vertexL}=[circle,draw, minimum size=8pt, scale=0.6, inner sep=0.5pt]
       		\begin{tikzpicture}[scale=0.2]
       			\node at (1,2) {  }; 
       			\node (p1) at (1,1) [vertexL]{$p_1$};
       			\node (p2) at (4,1) [vertexL]{$p_2$};
       			\node (p3) at (7,1) [vertexL]{$p_3$};
       			\node (p4) at (10,1) [vertexL]{$p_4$};
       			\node (p5) at (13,1) [vertexL]{$p_5$};
       			\node (p6) at (16,1) [vertexL]{$p_6$};
       			\node (p7) at (19,1) [vertexL]{$p_7$};
       			\node (p8) at (22,1) [vertexL]{$p_8$};
       			\draw[line width=0.04cm] (p1) -- (p2);
       			\draw[line width=0.04cm] (p2) -- (p3);
       			\draw[line width=0.04cm] (p3) -- (p4);
       			\draw[line width=0.04cm] (p5) -- (p6);
       			\draw[line width=0.04cm] (p6) -- (p7);
       			\draw[line width=0.04cm] (p7) -- (p8);
       		\end{tikzpicture} \\ \hline
       		$2/8$ & \tikzstyle{vertexL}=[circle,draw, minimum size=8pt, scale=0.6, inner sep=0.5pt]
       		\begin{tikzpicture}[scale=0.2]
       			\node at (1,2) {  };
       			\node (p1) at (1,1) [vertexL]{$p_1$};
       			\node (p2) at (4,1) [vertexL]{$p_2$};
       			\node (p3) at (7,1) [vertexL]{$p_3$};
       			\node (p4) at (10,1) [vertexL]{$p_4$};
       			\node (p5) at (13,1) [vertexL]{$p_5$};
       			\node (p6) at (16,1) [vertexL]{$p_6$};
       			\node (p7) at (19,1) [vertexL]{$p_7$};
       			\node (p8) at (22,1) [vertexL]{$p_8$};
       			\draw[line width=0.04cm] (p1) -- (p2);
       			\draw[line width=0.04cm] (p3) -- (p4);
       			\draw[line width=0.04cm] (p4) -- (p5);
       			\draw[line width=0.04cm] (p5) -- (p6);
       			\draw[line width=0.04cm] (p7) -- (p8);
       		\end{tikzpicture} \\ \hline
       		$1/8$ & \tikzstyle{vertexL}=[circle,draw, minimum size=8pt, scale=0.6, inner sep=0.5pt]
       		\begin{tikzpicture}[scale=0.2]
       			\node at (1,2) {  };
       			\node (p1) at (1,1) [vertexL]{$p_1$};
       			\node (p2) at (4,1) [vertexL]{$p_2$};
       			\node (p3) at (7,1) [vertexL]{$p_3$};
       			\node (p4) at (10,1) [vertexL]{$p_4$};
       			\node (p5) at (13,1) [vertexL]{$p_5$};
       			\node (p6) at (16,1) [vertexL]{$p_6$};
       			\node (p7) at (19,1) [vertexL]{$p_7$};
       			\node (p8) at (22,1) [vertexL]{$p_8$};
       			\draw[line width=0.04cm] (p2) -- (p3);
       			\draw[line width=0.04cm] (p3) -- (p4);
       			\draw[line width=0.04cm] (p4) -- (p5);
       			\draw[line width=0.04cm] (p6) -- (p7);
       		\end{tikzpicture} \\ \hline
       		$1/8$ & \tikzstyle{vertexL}=[circle,draw, minimum size=8pt, scale=0.6, inner sep=0.5pt]
       		\begin{tikzpicture}[scale=0.2]
       			\node at (1,2) {  };
       			\node (p1) at (1,1) [vertexL]{$p_1$};
       			\node (p2) at (4,1) [vertexL]{$p_2$};
       			\node (p3) at (7,1) [vertexL]{$p_3$};
       			\node (p4) at (10,1) [vertexL]{$p_4$};
       			\node (p5) at (13,1) [vertexL]{$p_5$};
       			\node (p6) at (16,1) [vertexL]{$p_6$};
       			\node (p7) at (19,1) [vertexL]{$p_7$};
       			\node (p8) at (22,1) [vertexL]{$p_8$};
       			\draw[line width=0.04cm] (p2) -- (p3);
       			\draw[line width=0.04cm] (p4) -- (p5);
       			\draw[line width=0.04cm] (p5) -- (p6);
       			\draw[line width=0.04cm] (p6) -- (p7);
       		\end{tikzpicture} \\ \hline
       		$1/8$ & \tikzstyle{vertexL}=[circle,draw, minimum size=8pt, scale=0.6, inner sep=0.5pt]
       		\begin{tikzpicture}[scale=0.2]
       			\node at (1,2) {  };
       			\node (p1) at (1,1) [vertexL]{$p_1$};
       			\node (p2) at (4,1) [vertexL]{$p_2$};
       			\node (p3) at (7,1) [vertexL]{$p_3$};
       			\node (p4) at (10,1) [vertexL]{$p_4$};
       			\node (p5) at (13,1) [vertexL]{$p_5$};
       			\node (p6) at (16,1) [vertexL]{$p_6$};
       			\node (p7) at (19,1) [vertexL]{$p_7$};
       			\node (p8) at (22,1) [vertexL]{$p_8$};
       			\draw[line width=0.04cm] (p1) -- (p2);
       		\end{tikzpicture} \\ \hline
       		$1/8$ & \tikzstyle{vertexL}=[circle,draw, minimum size=8pt, scale=0.6, inner sep=0.5pt]
       		\begin{tikzpicture}[scale=0.2]
       			\node at (1,2) {  };
       			\node (p1) at (1,1) [vertexL]{$p_1$};  \node at (1,-0.4) {{\tiny A}};
       			\node (p2) at (4,1) [vertexL]{$p_2$};
       			\node (p3) at (7,1) [vertexL]{$p_3$};
       			\node (p4) at (10,1) [vertexL]{$p_4$};
       			\node (p5) at (13,1) [vertexL]{$p_5$};
       			\node (p6) at (16,1) [vertexL]{$p_6$}; \node at (16,-0.4) {{\tiny B}};
       			\node (p7) at (19,1) [vertexL]{$p_7$}; \node at (19,-0.4) {{\tiny A}};
       			\node (p8) at (22,1) [vertexL]{$p_8$}; \node at (22,-0.4) {{\tiny B}};
       			\draw[line width=0.04cm] (p2) -- (p3);
       			\draw[line width=0.04cm] (p4) -- (p5);
       			\draw[line width=0.04cm] (p6) -- (p7);
       			\draw[line width=0.04cm] (p7) -- (p8);
       		\end{tikzpicture} \\ \hline
       	\end{tabular}
       \end{center}
       
       We now consider the case when $p_1 p_{n-1} \in E(H)$.  In this case, let $B_P$ be bipartite subgraphs induced by the following sets (with the given probabilities).
       
       \begin{itemize}
       	\item Use $\{p_1,p_2,p_3,p_4\}, \{p_5,p_6,p_7,p_8\}, \ldots , \{p_{n-3},p_{n-2},p_{n-1},p_n\}$ with probability $2/8$.
       	\item Use $ \{p_1,p_2\}, \{p_3,p_4,p_5,p_6\}, \{p_7,p_8,p_9,p_{10}\}, \ldots ,$ $\{p_{n-5},p_{n-4},p_{n-3},p_{n-2}\}, \{p_{n-1},p_{n}\}$ with probability $2/8$.
       	\item Use $ \{p_2,p_3,p_4,p_5\}, \{p_6,p_7,p_8,p_{9}\}, \ldots , \{p_{n-6},p_{n-5},p_{n-4},p_{n-3}\}, \{p_{n-2},p_{n-1}\}$ with probability $1/8$.
       	\item Use $ \{p_2,p_3\}, \{p_4,p_5,p_6,p_7\}, \{p_8,p_9,p_{10},p_{11}\}, \ldots , \{p_{n-4},p_{n-3},p_{n-2},p_{n-1}\}$ with probability $1/8$.
       	\item Use $ \{p_1,p_2\}, \{p_{n-1},p_{n}\}$ with probability $1/8$.
       	\item Use $ \{p_2,p_3\}, \{p_4,p_5\}, \ldots , \{p_{n-2},p_{n-1}\}$ with probability $1/8$.
       \end{itemize}
       
       Now, any edge of $P$ is in $B_C$ with probability at least $5/8$ and every vertex except $p_2$ and $p_{n-1}$ is not in $V(B_C)$ with probability at least $1/8$ (and $p_2$ is incident with the edge $p_2 p_n$ and $p_{n-1}$ is incident with the edge $p_1 p_{n-1}$). We illustrate the above probability space for $n=8$ below.
       
       \begin{center}
       	\begin{tabular}{|c|c|} \hline
       		$Prob$ & Subgraph \\ \hline \hline 
       		$2/8$ & \tikzstyle{vertexL}=[circle,draw, minimum size=8pt, scale=0.6, inner sep=0.5pt]
       		\begin{tikzpicture}[scale=0.2]
       			\node at (1,2) {  };
       			\node (p1) at (1,1) [vertexL]{$p_1$};  
       			\node (p2) at (4,1) [vertexL]{$p_2$};
       			\node (p3) at (7,1) [vertexL]{$p_3$};
       			\node (p4) at (10,1) [vertexL]{$p_4$};
       			\node (p5) at (13,1) [vertexL]{$p_5$};
       			\node (p6) at (16,1) [vertexL]{$p_6$}; 
       			\node (p7) at (19,1) [vertexL]{$p_7$}; 
       			\node (p8) at (22,1) [vertexL]{$p_8$}; 
       			\draw[line width=0.04cm] (p1) -- (p2);
       			\draw[line width=0.04cm] (p2) -- (p3);
       			\draw[line width=0.04cm] (p3) -- (p4);
       			\draw[line width=0.04cm] (p5) -- (p6);
       			\draw[line width=0.04cm] (p6) -- (p7);
       			\draw[line width=0.04cm] (p7) -- (p8);
       		\end{tikzpicture} \\ \hline
       		$2/8$ & \tikzstyle{vertexL}=[circle,draw, minimum size=8pt, scale=0.6, inner sep=0.5pt]
       		\begin{tikzpicture}[scale=0.2]
       			\node at (1,2) {  }; 
       			\node (p1) at (1,1) [vertexL]{$p_1$};  
       			\node (p2) at (4,1) [vertexL]{$p_2$};
       			\node (p3) at (7,1) [vertexL]{$p_3$};
       			\node (p4) at (10,1) [vertexL]{$p_4$};
       			\node (p5) at (13,1) [vertexL]{$p_5$};
       			\node (p6) at (16,1) [vertexL]{$p_6$}; 
       			\node (p7) at (19,1) [vertexL]{$p_7$}; 
       			\node (p8) at (22,1) [vertexL]{$p_8$}; 
       			\draw[line width=0.04cm] (p1) -- (p2);
       			\draw[line width=0.04cm] (p3) -- (p4);
       			\draw[line width=0.04cm] (p4) -- (p5);
       			\draw[line width=0.04cm] (p5) -- (p6);
       			\draw[line width=0.04cm] (p7) -- (p8);
       		\end{tikzpicture} \\ \hline
       		$1/8$ & \tikzstyle{vertexL}=[circle,draw, minimum size=8pt, scale=0.6, inner sep=0.5pt]
       		\begin{tikzpicture}[scale=0.2]
       			\node at (1,2) {  }; 
       			\node (p1) at (1,1) [vertexL]{$p_1$};  
       			\node (p2) at (4,1) [vertexL]{$p_2$};
       			\node (p3) at (7,1) [vertexL]{$p_3$};
       			\node (p4) at (10,1) [vertexL]{$p_4$};
       			\node (p5) at (13,1) [vertexL]{$p_5$};
       			\node (p6) at (16,1) [vertexL]{$p_6$}; 
       			\node (p7) at (19,1) [vertexL]{$p_7$}; 
       			\node (p8) at (22,1) [vertexL]{$p_8$}; 
       			\draw[line width=0.04cm] (p2) -- (p3);
       			\draw[line width=0.04cm] (p3) -- (p4);
       			\draw[line width=0.04cm] (p4) -- (p5);
       			\draw[line width=0.04cm] (p6) -- (p7);
       		\end{tikzpicture} \\ \hline
       		$1/8$ & \tikzstyle{vertexL}=[circle,draw, minimum size=8pt, scale=0.6, inner sep=0.5pt]
       		\begin{tikzpicture}[scale=0.2]
       			\node at (1,2) {  }; 
       			\node (p1) at (1,1) [vertexL]{$p_1$};  
       			\node (p2) at (4,1) [vertexL]{$p_2$};
       			\node (p3) at (7,1) [vertexL]{$p_3$};
       			\node (p4) at (10,1) [vertexL]{$p_4$};
       			\node (p5) at (13,1) [vertexL]{$p_5$};
       			\node (p6) at (16,1) [vertexL]{$p_6$}; 
       			\node (p7) at (19,1) [vertexL]{$p_7$}; 
       			\node (p8) at (22,1) [vertexL]{$p_8$}; 
       			\draw[line width=0.04cm] (p2) -- (p3);
       			\draw[line width=0.04cm] (p4) -- (p5);
       			\draw[line width=0.04cm] (p5) -- (p6);
       			\draw[line width=0.04cm] (p6) -- (p7);
       		\end{tikzpicture} \\ \hline
       		$1/8$ & \tikzstyle{vertexL}=[circle,draw, minimum size=8pt, scale=0.6, inner sep=0.5pt]
       		\begin{tikzpicture}[scale=0.2]
       			\node at (1,2) {  }; 
       			\node (p1) at (1,1) [vertexL]{$p_1$};  
       			\node (p2) at (4,1) [vertexL]{$p_2$};
       			\node (p3) at (7,1) [vertexL]{$p_3$};
       			\node (p4) at (10,1) [vertexL]{$p_4$};
       			\node (p5) at (13,1) [vertexL]{$p_5$};
       			\node (p6) at (16,1) [vertexL]{$p_6$}; 
       			\node (p7) at (19,1) [vertexL]{$p_7$}; 
       			\node (p8) at (22,1) [vertexL]{$p_8$}; 
       			\draw[line width=0.04cm] (p1) -- (p2);
       			\draw[line width=0.04cm] (p7) -- (p8);
       		\end{tikzpicture} \\ \hline
       		$1/8$ & \tikzstyle{vertexL}=[circle,draw, minimum size=8pt, scale=0.6, inner sep=0.5pt]
       		\begin{tikzpicture}[scale=0.2]
       			\node at (1,2) {  }; 
       			\node (p1) at (1,1) [vertexL]{$p_1$};  
       			\node (p2) at (4,1) [vertexL]{$p_2$};
       			\node (p3) at (7,1) [vertexL]{$p_3$};
       			\node (p4) at (10,1) [vertexL]{$p_4$};
       			\node (p5) at (13,1) [vertexL]{$p_5$};
       			\node (p6) at (16,1) [vertexL]{$p_6$}; 
       			\node (p7) at (19,1) [vertexL]{$p_7$}; 
       			\node (p8) at (22,1) [vertexL]{$p_8$}; 
       			\draw[line width=0.04cm] (p2) -- (p3);
       			\draw[line width=0.04cm] (p4) -- (p5);
       			\draw[line width=0.04cm] (p6) -- (p7);
       		\end{tikzpicture} \\ \hline
       	\end{tabular}
       \end{center}

       {\bf Case 3:} $n=4k+1$, where $k \geq 2$.
       
       First consider the case when $p_1 p_{n-1} \not\in E(H)$.  In this case, let $B_P$ be the bipartite subgraphs induced by the following sets (with the given probabilities).
       
       \begin{itemize} 
       	\item Use $\{p_1,p_2,p_3,p_4\}, \{p_5,p_6,p_7,p_8\}, \ldots , \{p_{n-4},p_{n-3},p_{n-2},p_{n-1}\}$ with probability $2/8$.
       	\item Use $\{p_2,p_3,p_4,p_5\}, \{p_6,p_7,p_8,p_{9}\}, \ldots , \{p_{n-3},p_{n-2},p_{n-1},p_{n}\}$ with probability $2/8$.
       	\item Use $ \{p_1,p_2\}, \{p_3,p_4,p_5,p_6\}, \{p_7,p_8,p_9,p_{10}\}, \ldots , \{p_{n-6},p_{n-5},p_{n-4},p_{n-3}\}, \{p_{n-1},p_{n}\}$ with probability $1/8$.
       	\item Use $ \{p_1,p_2\}, \{p_4,p_5,p_6,p_7\}, \{p_8,p_9,p_{10},p_{11}\}, \ldots , \{p_{n-5},p_{n-4},p_{n-3},p_{n-2}\}, \{p_{n-1},p_{n}\}$ with probability $1/8$.
       	\item Use $ \{p_1,p_2\}, \{p_4,p_5,p_6,p_7\}, \{p_8,p_9,p_{10},p_{11}\}, \ldots , \{p_{n-5},p_{n-4},p_{n-3},p_{n-2}\}$ with probability $1/8$.
       	\item Use $ \{p_2,p_3\}, \{p_{n-2},p_{n-1},p_n,p_1\}$ with probability $1/8$.
       \end{itemize}

Now, every edge of $P$ is in $B_P$ with probability at least $5/8$, and every vertex except $p_2$ is not in $V(B_P)$ with probability at least $1/8$. We illustrate the above probability space for $n=9$ below (in the last figure the A's indicate one partite set and the B's the other partite set of the subgraph induced by $\{p_1,p_7,p_8,p_9\}$).
       
       \begin{center}
       	\begin{tabular}{|c|c|} \hline
       		$Prob$ & Subgraph \\ \hline \hline
       		$2/8$ & \tikzstyle{vertexL}=[circle,draw, minimum size=8pt, scale=0.6, inner sep=0.5pt]
       		\begin{tikzpicture}[scale=0.2]
       			\node at (1,2) {  }; 
       			\node (p1) at (1,1) [vertexL]{$p_1$};  
       			\node (p2) at (4,1) [vertexL]{$p_2$};
       			\node (p3) at (7,1) [vertexL]{$p_3$};
       			\node (p4) at (10,1) [vertexL]{$p_4$};
       			\node (p5) at (13,1) [vertexL]{$p_5$};
       			\node (p6) at (16,1) [vertexL]{$p_6$};
       			\node (p7) at (19,1) [vertexL]{$p_7$}; 
       			\node (p8) at (22,1) [vertexL]{$p_8$}; 
       			\node (p9) at (25,1) [vertexL]{$p_9$};
       			\draw[line width=0.04cm] (p1) -- (p2); 
       			\draw[line width=0.04cm] (p2) -- (p3);
       			\draw[line width=0.04cm] (p3) -- (p4);
       			\draw[line width=0.04cm] (p5) -- (p6);
       			\draw[line width=0.04cm] (p6) -- (p7);
       			\draw[line width=0.04cm] (p7) -- (p8);
       		\end{tikzpicture} \\ \hline
       		$2/8$ & \tikzstyle{vertexL}=[circle,draw, minimum size=8pt, scale=0.6, inner sep=0.5pt]
       		\begin{tikzpicture}[scale=0.2]
       			\node at (1,2) {  };
       			\node (p1) at (1,1) [vertexL]{$p_1$};
       			\node (p2) at (4,1) [vertexL]{$p_2$};
       			\node (p3) at (7,1) [vertexL]{$p_3$};
       			\node (p4) at (10,1) [vertexL]{$p_4$};
       			\node (p5) at (13,1) [vertexL]{$p_5$};
       			\node (p6) at (16,1) [vertexL]{$p_6$};
       			\node (p7) at (19,1) [vertexL]{$p_7$};
       			\node (p8) at (22,1) [vertexL]{$p_8$};
       			\node (p9) at (25,1) [vertexL]{$p_9$};
       			\draw[line width=0.04cm] (p2) -- (p3);
       			\draw[line width=0.04cm] (p3) -- (p4);
       			\draw[line width=0.04cm] (p4) -- (p5);
       			\draw[line width=0.04cm] (p6) -- (p7);
       			\draw[line width=0.04cm] (p7) -- (p8);
       			\draw[line width=0.04cm] (p8) -- (p9);
       		\end{tikzpicture} \\ \hline
       		$1/8$ & \tikzstyle{vertexL}=[circle,draw, minimum size=8pt, scale=0.6, inner sep=0.5pt]
       		\begin{tikzpicture}[scale=0.2]
       			\node at (1,2) {  };
       			\node (p1) at (1,1) [vertexL]{$p_1$};
       			\node (p2) at (4,1) [vertexL]{$p_2$};
       			\node (p3) at (7,1) [vertexL]{$p_3$};
       			\node (p4) at (10,1) [vertexL]{$p_4$};
       			\node (p5) at (13,1) [vertexL]{$p_5$};
       			\node (p6) at (16,1) [vertexL]{$p_6$};
       			\node (p7) at (19,1) [vertexL]{$p_7$};
       			\node (p8) at (22,1) [vertexL]{$p_8$};
       			\node (p9) at (25,1) [vertexL]{$p_9$};
       			\draw[line width=0.04cm] (p1) -- (p2); 
       			\draw[line width=0.04cm] (p3) -- (p4);
       			\draw[line width=0.04cm] (p4) -- (p5);
       			\draw[line width=0.04cm] (p5) -- (p6);
       			\draw[line width=0.04cm] (p8) -- (p9);
       		\end{tikzpicture} \\ \hline
       		$1/8$ & \tikzstyle{vertexL}=[circle,draw, minimum size=8pt, scale=0.6, inner sep=0.5pt]
       		\begin{tikzpicture}[scale=0.2]
       			\node at (1,2) {  };
       			\node (p1) at (1,1) [vertexL]{$p_1$};
       			\node (p2) at (4,1) [vertexL]{$p_2$};
       			\node (p3) at (7,1) [vertexL]{$p_3$};
       			\node (p4) at (10,1) [vertexL]{$p_4$};
       			\node (p5) at (13,1) [vertexL]{$p_5$};
       			\node (p6) at (16,1) [vertexL]{$p_6$};
       			\node (p7) at (19,1) [vertexL]{$p_7$};
       			\node (p8) at (22,1) [vertexL]{$p_8$};
       			\node (p9) at (25,1) [vertexL]{$p_9$};
       			\draw[line width=0.04cm] (p1) -- (p2); 
       			\draw[line width=0.04cm] (p4) -- (p5);
       			\draw[line width=0.04cm] (p5) -- (p6);
       			\draw[line width=0.04cm] (p6) -- (p7);
       			\draw[line width=0.04cm] (p8) -- (p9);
       		\end{tikzpicture} \\ \hline
       		$1/8$ & \tikzstyle{vertexL}=[circle,draw, minimum size=8pt, scale=0.6, inner sep=0.5pt]
       		\begin{tikzpicture}[scale=0.2]
       			\node at (1,2) {  };
       			\node (p1) at (1,1) [vertexL]{$p_1$};
       			\node (p2) at (4,1) [vertexL]{$p_2$};
       			\node (p3) at (7,1) [vertexL]{$p_3$};
       			\node (p4) at (10,1) [vertexL]{$p_4$};
       			\node (p5) at (13,1) [vertexL]{$p_5$};
       			\node (p6) at (16,1) [vertexL]{$p_6$};
       			\node (p7) at (19,1) [vertexL]{$p_7$};
       			\node (p8) at (22,1) [vertexL]{$p_8$};
       			\node (p9) at (25,1) [vertexL]{$p_9$};
       			\draw[line width=0.04cm] (p1) -- (p2); 
       			\draw[line width=0.04cm] (p4) -- (p5);
       			\draw[line width=0.04cm] (p5) -- (p6);
       			\draw[line width=0.04cm] (p6) -- (p7);
       		\end{tikzpicture} \\ \hline
       		$1/8$ & \tikzstyle{vertexL}=[circle,draw, minimum size=8pt, scale=0.6, inner sep=0.5pt]
       		\begin{tikzpicture}[scale=0.2]
       			\node at (1,2) {  };
       			\node (p1) at (1,1) [vertexL]{$p_1$};  \node at (1,-0.4) {{\tiny A}};
       			\node (p2) at (4,1) [vertexL]{$p_2$};
       			\node (p3) at (7,1) [vertexL]{$p_3$};
       			\node (p4) at (10,1) [vertexL]{$p_4$};
       			\node (p5) at (13,1) [vertexL]{$p_5$};
       			\node (p6) at (16,1) [vertexL]{$p_6$};
       			\node (p7) at (19,1) [vertexL]{$p_7$}; \node at (19,-0.4) {{\tiny B}};
       			\node (p8) at (22,1) [vertexL]{$p_8$}; \node at (22,-0.4) {{\tiny A}};
       			\node (p9) at (25,1) [vertexL]{$p_9$}; \node at (25,-0.4) {{\tiny B}};
       			\draw[line width=0.04cm] (p2) -- (p3);
       			\draw[line width=0.04cm] (p7) -- (p8);
       			\draw[line width=0.04cm] (p8) -- (p9);
       		\end{tikzpicture} \\ \hline
       	\end{tabular}
       \end{center}

       We now consider the case when $p_1 p_{n-1} \in E(H)$.  In this case let $B_P$ be the bipartite subgraphs induced by the following sets (with the given probabilities).

       \begin{itemize} 
       	\item Use $\{p_1,p_2,p_3,p_4\}, \{p_5,p_6,p_7,p_8\}, \ldots , \{p_{n-4},p_{n-3},p_{n-2},p_{n-1}\}$ with probability $2/8$.
       	\item Use $\{p_2,p_3,p_4,p_5\}, \{p_6,p_7,p_8,p_{9}\}, \ldots , \{p_{n-3},p_{n-2},p_{n-1},p_{n}\}$ with probability $2/8$.
       	\item Use $ \{p_1,p_2\}, \{p_3,p_4,p_5,p_6\}, \{p_7,p_8,p_9,p_{10}\}, \ldots , \{p_{n-6},p_{n-5},p_{n-4},p_{n-3}\}, \{p_{n-1},p_{n}\}$ with probability $2/8$.
       	\item Use $ \{p_1,p_2\}, \{p_4,p_5,p_6,p_7\}, \{p_8,p_9,p_{10},p_{11}\}, \ldots , \{p_{n-5},p_{n-4},p_{n-3},p_{n-2}\}, \{p_{n-1},p_{n}\}$ with probability $1/8$.
       	\item Use $ \{p_2,p_3\}, \{p_{n-2},p_{n-1}\}$ with probability $1/8$.
       \end{itemize}

       Now, every edge of $P$ is in $B_P$ with probability at least $5/8$, and every vertex except $p_2$ and $p_{n-1}$ is not in $V(B_P)$ with probability at least $1/8$. We illustrate the above probability space for $n=9$ below.
       
       \begin{center}
       	\begin{tabular}{|c|c|} \hline
       		$Prob$ & Subgraph \\ \hline \hline 
       		$2/8$ & \tikzstyle{vertexL}=[circle,draw, minimum size=8pt, scale=0.6, inner sep=0.5pt]
       		\begin{tikzpicture}[scale=0.2]
       			\node at (1,2) {  };
       			\node (p1) at (1,1) [vertexL]{$p_1$};  
       			\node (p2) at (4,1) [vertexL]{$p_2$};
       			\node (p3) at (7,1) [vertexL]{$p_3$};
       			\node (p4) at (10,1) [vertexL]{$p_4$};
       			\node (p5) at (13,1) [vertexL]{$p_5$};
       			\node (p6) at (16,1) [vertexL]{$p_6$};
       			\node (p7) at (19,1) [vertexL]{$p_7$}; 
       			\node (p8) at (22,1) [vertexL]{$p_8$}; 
       			\node (p9) at (25,1) [vertexL]{$p_9$}; 
       			\draw[line width=0.04cm] (p1) -- (p2);
       			\draw[line width=0.04cm] (p2) -- (p3);
       			\draw[line width=0.04cm] (p3) -- (p4);
       			\draw[line width=0.04cm] (p5) -- (p6);
       			\draw[line width=0.04cm] (p6) -- (p7);
       			\draw[line width=0.04cm] (p7) -- (p8);
       		\end{tikzpicture} \\ \hline
       		$2/8$ & \tikzstyle{vertexL}=[circle,draw, minimum size=8pt, scale=0.6, inner sep=0.5pt]
       		\begin{tikzpicture}[scale=0.2]
       			\node at (1,2) {  };
       			\node (p1) at (1,1) [vertexL]{$p_1$};
       			\node (p2) at (4,1) [vertexL]{$p_2$};
       			\node (p3) at (7,1) [vertexL]{$p_3$};
       			\node (p4) at (10,1) [vertexL]{$p_4$};
       			\node (p5) at (13,1) [vertexL]{$p_5$};
       			\node (p6) at (16,1) [vertexL]{$p_6$};
       			\node (p7) at (19,1) [vertexL]{$p_7$};
       			\node (p8) at (22,1) [vertexL]{$p_8$};
       			\node (p9) at (25,1) [vertexL]{$p_9$};
       			\draw[line width=0.04cm] (p2) -- (p3);
       			\draw[line width=0.04cm] (p3) -- (p4);
       			\draw[line width=0.04cm] (p4) -- (p5);
       			\draw[line width=0.04cm] (p6) -- (p7);
       			\draw[line width=0.04cm] (p7) -- (p8);
       			\draw[line width=0.04cm] (p8) -- (p9);
       		\end{tikzpicture} \\ \hline
       		$2/8$ & \tikzstyle{vertexL}=[circle,draw, minimum size=8pt, scale=0.6, inner sep=0.5pt]
       		\begin{tikzpicture}[scale=0.2]
       			\node at (1,2) {  };
       			\node (p1) at (1,1) [vertexL]{$p_1$};
       			\node (p2) at (4,1) [vertexL]{$p_2$};
       			\node (p3) at (7,1) [vertexL]{$p_3$};
       			\node (p4) at (10,1) [vertexL]{$p_4$};
       			\node (p5) at (13,1) [vertexL]{$p_5$};
       			\node (p6) at (16,1) [vertexL]{$p_6$};
       			\node (p7) at (19,1) [vertexL]{$p_7$};
       			\node (p8) at (22,1) [vertexL]{$p_8$};
       			\node (p9) at (25,1) [vertexL]{$p_9$};
       			\draw[line width=0.04cm] (p1) -- (p2);
       			\draw[line width=0.04cm] (p3) -- (p4);
       			\draw[line width=0.04cm] (p4) -- (p5);
       			\draw[line width=0.04cm] (p5) -- (p6);
       			\draw[line width=0.04cm] (p8) -- (p9);
       		\end{tikzpicture} \\ \hline
       		$1/8$ & \tikzstyle{vertexL}=[circle,draw, minimum size=8pt, scale=0.6, inner sep=0.5pt]
       		\begin{tikzpicture}[scale=0.2]
       			\node at (1,2) {  };
       			\node (p1) at (1,1) [vertexL]{$p_1$};
       			\node (p2) at (4,1) [vertexL]{$p_2$};
       			\node (p3) at (7,1) [vertexL]{$p_3$};
       			\node (p4) at (10,1) [vertexL]{$p_4$};
       			\node (p5) at (13,1) [vertexL]{$p_5$};
       			\node (p6) at (16,1) [vertexL]{$p_6$};
       			\node (p7) at (19,1) [vertexL]{$p_7$};
       			\node (p8) at (22,1) [vertexL]{$p_8$};
       			\node (p9) at (25,1) [vertexL]{$p_9$};
       			\draw[line width=0.04cm] (p1) -- (p2);
       			\draw[line width=0.04cm] (p4) -- (p5);
       			\draw[line width=0.04cm] (p5) -- (p6);
       			\draw[line width=0.04cm] (p6) -- (p7);
       			\draw[line width=0.04cm] (p8) -- (p9);
       		\end{tikzpicture} \\ \hline
       		$1/8$ & \tikzstyle{vertexL}=[circle,draw, minimum size=8pt, scale=0.6, inner sep=0.5pt]
       		\begin{tikzpicture}[scale=0.2]
       			\node at (1,2) {  };
       			\node (p1) at (1,1) [vertexL]{$p_1$};
       			\node (p2) at (4,1) [vertexL]{$p_2$};
       			\node (p3) at (7,1) [vertexL]{$p_3$};
       			\node (p4) at (10,1) [vertexL]{$p_4$};
       			\node (p5) at (13,1) [vertexL]{$p_5$};
       			\node (p6) at (16,1) [vertexL]{$p_6$};
       			\node (p7) at (19,1) [vertexL]{$p_7$};
       			\node (p8) at (22,1) [vertexL]{$p_8$};
       			\node (p9) at (25,1) [vertexL]{$p_9$};
       			\draw[line width=0.04cm] (p2) -- (p3);
       			\draw[line width=0.04cm] (p7) -- (p8);
       		\end{tikzpicture} \\ \hline
       	\end{tabular}
       \end{center}

       {\bf Case 4:} $n=4k+2$, where $k \geq 1$.
       
       First consider the case when $p_1 p_{n-1} \not\in E(H)$.  In this case, let $B_P$ be bipartite subgraphs induced by the following sets (with the given probabilities).
       
       \begin{itemize}
       	\item Use $\{p_1,p_2,p_3,p_4\}, \{p_5,p_6,p_7,p_8\}, \ldots , \{p_{n-5},p_{n-4},p_{n-3},p_{n-2}\}, \{p_{n-1},p_{n}\}$ with probability $2/8$.
       	\item Use $ \{p_1,p_2\}, \{p_3,p_4,p_5,p_6\}, \{p_7,p_8,p_9,p_{10}\}, \ldots , \{p_{n-3},p_{n-2},p_{n-1},p_{n}\}$ with probability $2/8$.
       	\item Use $\{p_2,p_3,p_4,p_5\}, \{p_6,p_7,p_8,p_{9}\}, \ldots , \{p_{n-4},p_{n-3},p_{n-2},p_{n-1}\}$ with probability $1/8$.
       	\item Use $ \{p_2,p_3\}, \{p_4,p_5,p_6,p_7\}, \{p_8,p_9,p_{10},p_{11}\}, \ldots , \{p_{n-6},p_{n-5},p_{n-4},p_{n-3}\}, \{p_{n-2},p_{n-1}\}$ with probability $1/8$.
       	\item Use $ \{p_1,p_2\}$ with probability $1/8$.
       	\item Use $ \{p_2,p_3\}, \{p_4,p_5\}, \ldots , \{p_{n-4},p_{n-3}\},\{p_{n-2},p_{n-1},p_n,p_1\}$ with probability $1/8$.
       \end{itemize}

      Now, every edge of $P$ is in $B_P$ with probability at least $5/8$ and every vertex except $p_2$ is not in $V(B_P)$ with probability at least $1/8$. We illustrate the above probability space for $n=6$ and $n=10$ below (in the last picture the A's indicate one partite set and the B's the other partite set of the subgraph induced by $\{p_1,p_{n-2},p_{n-1},p_n\}$).
       
       \begin{center}
       	\begin{tabular}{|c|c|c|} \hline
       		$Prob$ & Subgraph $n=6$ & Subgraph $n=10$ \\ \hline \hline
       		$2/8$ & \tikzstyle{vertexL}=[circle,draw, minimum size=8pt, scale=0.6, inner sep=0.5pt]
       		\begin{tikzpicture}[scale=0.2]
       			\node at (1,2) {  };
       			\node (p1) at (1,1) [vertexL]{$p_1$};
       			\node (p2) at (4,1) [vertexL]{$p_2$};
       			\node (p3) at (7,1) [vertexL]{$p_3$};
       			\node (p4) at (10,1) [vertexL]{$p_4$};
       			\node (p5) at (13,1) [vertexL]{$p_5$};
       			\node (p6) at (16,1) [vertexL]{$p_6$};
       			\draw[line width=0.04cm] (p1) -- (p2);
       			\draw[line width=0.04cm] (p2) -- (p3);
       			\draw[line width=0.04cm] (p3) -- (p4);
       			\draw[line width=0.04cm] (p5) -- (p6);
       		\end{tikzpicture} &  \tikzstyle{vertexL}=[circle,draw, minimum size=8pt, scale=0.6, inner sep=0.5pt]
       		\tikzstyle{vertexS}=[circle,draw, minimum size=8pt, scale=0.45, inner sep=0.5pt]
       		\begin{tikzpicture}[scale=0.2]
       			\node at (1,2) {  };
       			\node (p1) at (1,1) [vertexL]{$p_1$};
       			\node (p2) at (4,1) [vertexL]{$p_2$};
       			\node (p3) at (7,1) [vertexL]{$p_3$};
       			\node (p4) at (10,1) [vertexL]{$p_4$}; 
       			\node (p5) at (13,1) [vertexL]{$p_5$}; 
       			\node (p6) at (16,1) [vertexL]{$p_6$};
       			\node (p7) at (19,1) [vertexL]{$p_7$};
       			\node (p8) at (22,1) [vertexL]{$p_8$};
       			\node (p9) at (25,1) [vertexL]{$p_9$};
       			\node (p10) at (28,1) [vertexS]{$p_{10}$};
       			\draw[line width=0.04cm] (p1) -- (p2);
       			\draw[line width=0.04cm] (p2) -- (p3);
       			\draw[line width=0.04cm] (p3) -- (p4);
       			\draw[line width=0.04cm] (p5) -- (p6);
       			\draw[line width=0.04cm] (p6) -- (p7);
       			\draw[line width=0.04cm] (p7) -- (p8);
       			\draw[line width=0.04cm] (p9) -- (p10);
       		\end{tikzpicture} \\ \hline
       		$2/8$ & \tikzstyle{vertexL}=[circle,draw, minimum size=8pt, scale=0.6, inner sep=0.5pt]
       		\begin{tikzpicture}[scale=0.2]
       			\node at (1,2) {  };
       			\node (p1) at (1,1) [vertexL]{$p_1$};
       			\node (p2) at (4,1) [vertexL]{$p_2$};
       			\node (p3) at (7,1) [vertexL]{$p_3$};
       			\node (p4) at (10,1) [vertexL]{$p_4$};
       			\node (p5) at (13,1) [vertexL]{$p_5$};
       			\node (p6) at (16,1) [vertexL]{$p_6$};
       			\draw[line width=0.04cm] (p1) -- (p2);
       			\draw[line width=0.04cm] (p3) -- (p4);
       			\draw[line width=0.04cm] (p4) -- (p5);
       			\draw[line width=0.04cm] (p5) -- (p6);
       		\end{tikzpicture} &  \tikzstyle{vertexL}=[circle,draw, minimum size=8pt, scale=0.6, inner sep=0.5pt]
       		\tikzstyle{vertexS}=[circle,draw, minimum size=8pt, scale=0.45, inner sep=0.5pt]
       		\begin{tikzpicture}[scale=0.2]
       			\node at (1,2) {  };
       			\node (p1) at (1,1) [vertexL]{$p_1$};
       			\node (p2) at (4,1) [vertexL]{$p_2$};
       			\node (p3) at (7,1) [vertexL]{$p_3$};
       			\node (p4) at (10,1) [vertexL]{$p_4$}; 
       			\node (p5) at (13,1) [vertexL]{$p_5$}; 
       			\node (p6) at (16,1) [vertexL]{$p_6$};
       			\node (p7) at (19,1) [vertexL]{$p_7$};
       			\node (p8) at (22,1) [vertexL]{$p_8$};
       			\node (p9) at (25,1) [vertexL]{$p_9$};
       			\node (p10) at (28,1) [vertexS]{$p_{10}$};
       			\draw[line width=0.04cm] (p1) -- (p2);
       			\draw[line width=0.04cm] (p3) -- (p4);
       			\draw[line width=0.04cm] (p4) -- (p5);
       			\draw[line width=0.04cm] (p5) -- (p6);
       			\draw[line width=0.04cm] (p7) -- (p8);
       			\draw[line width=0.04cm] (p8) -- (p9);
       			\draw[line width=0.04cm] (p9) -- (p10);
       		\end{tikzpicture}  \\ \hline
       		$1/8$ & \tikzstyle{vertexL}=[circle,draw, minimum size=8pt, scale=0.6, inner sep=0.5pt]
       		\begin{tikzpicture}[scale=0.2]
       			\node at (1,2) {  };
       			\node (p1) at (1,1) [vertexL]{$p_1$};
       			\node (p2) at (4,1) [vertexL]{$p_2$};
       			\node (p3) at (7,1) [vertexL]{$p_3$};
       			\node (p4) at (10,1) [vertexL]{$p_4$};
       			\node (p5) at (13,1) [vertexL]{$p_5$};
       			\node (p6) at (16,1) [vertexL]{$p_6$};
       			\draw[line width=0.04cm] (p2) -- (p3);
       			\draw[line width=0.04cm] (p3) -- (p4);
       			\draw[line width=0.04cm] (p4) -- (p5);
       		\end{tikzpicture} &  \tikzstyle{vertexL}=[circle,draw, minimum size=8pt, scale=0.6, inner sep=0.5pt]
       		\tikzstyle{vertexS}=[circle,draw, minimum size=8pt, scale=0.45, inner sep=0.5pt]
       		\begin{tikzpicture}[scale=0.2]
       			\node at (1,2) {  };
       			\node (p1) at (1,1) [vertexL]{$p_1$};
       			\node (p2) at (4,1) [vertexL]{$p_2$};
       			\node (p3) at (7,1) [vertexL]{$p_3$};
       			\node (p4) at (10,1) [vertexL]{$p_4$}; 
       			\node (p5) at (13,1) [vertexL]{$p_5$}; 
       			\node (p6) at (16,1) [vertexL]{$p_6$};
       			\node (p7) at (19,1) [vertexL]{$p_7$};
       			\node (p8) at (22,1) [vertexL]{$p_8$};
       			\node (p9) at (25,1) [vertexL]{$p_9$};
       			\node (p10) at (28,1) [vertexS]{$p_{10}$};
       			\draw[line width=0.04cm] (p2) -- (p3);
       			\draw[line width=0.04cm] (p3) -- (p4);
       			\draw[line width=0.04cm] (p4) -- (p5);
       			\draw[line width=0.04cm] (p6) -- (p7);
       			\draw[line width=0.04cm] (p7) -- (p8);
       			\draw[line width=0.04cm] (p8) -- (p9);
       		\end{tikzpicture} \\ \hline
       		$1/8$ & \tikzstyle{vertexL}=[circle,draw, minimum size=8pt, scale=0.6, inner sep=0.5pt]
       		\begin{tikzpicture}[scale=0.2]
       			\node at (1,2) {  };
       			\node (p1) at (1,1) [vertexL]{$p_1$};
       			\node (p2) at (4,1) [vertexL]{$p_2$};
       			\node (p3) at (7,1) [vertexL]{$p_3$};
       			\node (p4) at (10,1) [vertexL]{$p_4$};
       			\node (p5) at (13,1) [vertexL]{$p_5$};
       			\node (p6) at (16,1) [vertexL]{$p_6$};
       			\draw[line width=0.04cm] (p2) -- (p3);
       			\draw[line width=0.04cm] (p4) -- (p5);
       		\end{tikzpicture} &  \tikzstyle{vertexL}=[circle,draw, minimum size=8pt, scale=0.6, inner sep=0.5pt]
       		\tikzstyle{vertexS}=[circle,draw, minimum size=8pt, scale=0.45, inner sep=0.5pt]
       		\begin{tikzpicture}[scale=0.2]
       			\node at (1,2) {  };
       			\node (p1) at (1,1) [vertexL]{$p_1$};
       			\node (p2) at (4,1) [vertexL]{$p_2$};
       			\node (p3) at (7,1) [vertexL]{$p_3$};
       			\node (p4) at (10,1) [vertexL]{$p_4$}; 
       			\node (p5) at (13,1) [vertexL]{$p_5$}; 
       			\node (p6) at (16,1) [vertexL]{$p_6$};
       			\node (p7) at (19,1) [vertexL]{$p_7$};
       			\node (p8) at (22,1) [vertexL]{$p_8$};
       			\node (p9) at (25,1) [vertexL]{$p_9$};
       			\node (p10) at (28,1) [vertexS]{$p_{10}$};
       			\draw[line width=0.04cm] (p2) -- (p3);
       			\draw[line width=0.04cm] (p4) -- (p5);
       			\draw[line width=0.04cm] (p5) -- (p6);
       			\draw[line width=0.04cm] (p6) -- (p7);
       			\draw[line width=0.04cm] (p8) -- (p9);
       		\end{tikzpicture} \\ \hline
       		$1/8$ & \tikzstyle{vertexL}=[circle,draw, minimum size=8pt, scale=0.6, inner sep=0.5pt]
       		\begin{tikzpicture}[scale=0.2]
       			\node at (1,2) {  };
       			\node (p1) at (1,1) [vertexL]{$p_1$};
       			\node (p2) at (4,1) [vertexL]{$p_2$};
       			\node (p3) at (7,1) [vertexL]{$p_3$};
       			\node (p4) at (10,1) [vertexL]{$p_4$};
       			\node (p5) at (13,1) [vertexL]{$p_5$};
       			\node (p6) at (16,1) [vertexL]{$p_6$};
       			\draw[line width=0.04cm] (p1) -- (p2);
       		\end{tikzpicture} &  \tikzstyle{vertexL}=[circle,draw, minimum size=8pt, scale=0.6, inner sep=0.5pt]
       		\tikzstyle{vertexS}=[circle,draw, minimum size=8pt, scale=0.45, inner sep=0.5pt]
       		\begin{tikzpicture}[scale=0.2]
       			\node at (1,2) {  };
       			\node (p1) at (1,1) [vertexL]{$p_1$};
       			\node (p2) at (4,1) [vertexL]{$p_2$};
       			\node (p3) at (7,1) [vertexL]{$p_3$};
       			\node (p4) at (10,1) [vertexL]{$p_4$}; 
       			\node (p5) at (13,1) [vertexL]{$p_5$}; 
       			\node (p6) at (16,1) [vertexL]{$p_6$};
       			\node (p7) at (19,1) [vertexL]{$p_7$};
       			\node (p8) at (22,1) [vertexL]{$p_8$};
       			\node (p9) at (25,1) [vertexL]{$p_9$};
       			\node (p10) at (28,1) [vertexS]{$p_{10}$};
       			\draw[line width=0.04cm] (p1) -- (p2);
       		\end{tikzpicture} \\ \hline
       		$1/8$ & \tikzstyle{vertexL}=[circle,draw, minimum size=8pt, scale=0.6, inner sep=0.5pt]
       		\begin{tikzpicture}[scale=0.2]
       			\node at (1,2) {  };
       			\node (p1) at (1,1) [vertexL]{$p_1$}; \node at (1,-0.4) {{\tiny A}};
       			\node (p2) at (4,1) [vertexL]{$p_2$};
       			\node (p3) at (7,1) [vertexL]{$p_3$};
       			\node (p4) at (10,1) [vertexL]{$p_4$}; \node at (10,-0.4) {{\tiny B}};
       			\node (p5) at (13,1) [vertexL]{$p_5$}; \node at (13,-0.4) {{\tiny A}};
       			\node (p6) at (16,1) [vertexL]{$p_6$}; \node at (16,-0.4) {{\tiny B}};
       			\draw[line width=0.04cm] (p2) -- (p3);
       			\draw[line width=0.04cm] (p4) -- (p5);
       			\draw[line width=0.04cm] (p5) -- (p6);
       		\end{tikzpicture} &  \tikzstyle{vertexL}=[circle,draw, minimum size=8pt, scale=0.6, inner sep=0.5pt]
       		\tikzstyle{vertexS}=[circle,draw, minimum size=8pt, scale=0.45, inner sep=0.5pt]
       		\begin{tikzpicture}[scale=0.2]
       			\node at (1,2) {  };
       			\node (p1) at (1,1) [vertexL]{$p_1$};  \node at (1,-0.4) {{\tiny A}};
       			\node (p2) at (4,1) [vertexL]{$p_2$};
       			\node (p3) at (7,1) [vertexL]{$p_3$};
       			\node (p4) at (10,1) [vertexL]{$p_4$}; 
       			\node (p5) at (13,1) [vertexL]{$p_5$}; 
       			\node (p6) at (16,1) [vertexL]{$p_6$};
       			\node (p7) at (19,1) [vertexL]{$p_7$};
       			\node (p8) at (22,1) [vertexL]{$p_8$}; \node at (22,-0.4) {{\tiny B}};
       			\node (p9) at (25,1) [vertexL]{$p_9$}; \node at (25,-0.4) {{\tiny A}};
       			\node (p10) at (28,1) [vertexS]{$p_{10}$}; \node at (28,-0.4) {{\tiny B}};
       			\draw[line width=0.04cm] (p2) -- (p3);
       			\draw[line width=0.04cm] (p4) -- (p5);
       			\draw[line width=0.04cm] (p6) -- (p7);
       			\draw[line width=0.04cm] (p8) -- (p9);
       			\draw[line width=0.04cm] (p9) -- (p10);
       		\end{tikzpicture} \\ \hline
       	\end{tabular}
       \end{center}
       
       We now consider the case when $p_1 p_{n-1} \in E(H)$.  In this case we pick the bipartite subgraphs induced by the following sets (with the given probabilities).
       
       \begin{itemize}
       	\item Use $\{p_1,p_2,p_3,p_4\}, \{p_5,p_6,p_7,p_8\}, \ldots , \{p_{n-5},p_{n-4},p_{n-3},p_{n-2}\}, \{p_{n-1},p_{n}\}$ with probability $2/8$.
       	\item Use $ \{p_1,p_2\}, \{p_3,p_4,p_5,p_6\}, \{p_7,p_8,p_9,p_{10}\}, \ldots , \{p_{n-3},p_{n-2},p_{n-1},p_{n}\}$ with probability $2/8$.
       	\item Use $\{p_2,p_3,p_4,p_5\}, \{p_6,p_7,p_8,p_{9}\}, \ldots , \{p_{n-4},p_{n-3},p_{n-2},p_{n-1}\}$ with probability $1/8$.
       	\item Use $ \{p_2,p_3\}, \{p_4,p_5,p_6,p_7\}, \{p_8,p_9,p_{10},p_{11}\}, \ldots , \{p_{n-6},p_{n-5},p_{n-4},p_{n-3}\}, \{p_{n-2},p_{n-1}\}$ with probability $1/8$.
       	\item Use $ \{p_1,p_2\},\{p_{n-1},p_{n}\}$ with probability $1/8$.
       	\item Use $ \{p_2,p_3\}, \{p_4,p_5\}, \ldots , \{p_{n-2},p_{n-1}\}$ with probability $1/8$.
       \end{itemize}
       
       Now, every edge of $P$ is in $B_C$ with probability at least $5/8$ and every vertex except $p_2$ and $p_{n-1}$ will be an isolated vertex with probability at least $1/8$
       (and $p_2$ is incident with the edge $p_2 p_n$ and $p_{n-1}$ is incident with the edge $p_1 p_{n-1}$).
       We illustrate the above probability space for $n=6$  and $n=10$ below.

       \begin{center}
       	\begin{tabular}{|c|c|c|} \hline
       		$Prob$ & Subgraph $n=6$ & Subgraph $n=10$ \\ \hline \hline
       		$2/8$ & \tikzstyle{vertexL}=[circle,draw, minimum size=8pt, scale=0.6, inner sep=0.5pt]
       		\begin{tikzpicture}[scale=0.2]
       			\node at (1,2) {  };
       			\node (p1) at (1,1) [vertexL]{$p_1$};
       			\node (p2) at (4,1) [vertexL]{$p_2$};
       			\node (p3) at (7,1) [vertexL]{$p_3$};
       			\node (p4) at (10,1) [vertexL]{$p_4$};
       			\node (p5) at (13,1) [vertexL]{$p_5$};
       			\node (p6) at (16,1) [vertexL]{$p_6$};
       			\draw[line width=0.04cm] (p1) -- (p2);
       			\draw[line width=0.04cm] (p2) -- (p3);
       			\draw[line width=0.04cm] (p3) -- (p4);
       			\draw[line width=0.04cm] (p5) -- (p6);
       		\end{tikzpicture} &  \tikzstyle{vertexL}=[circle,draw, minimum size=8pt, scale=0.6, inner sep=0.5pt]
       		\tikzstyle{vertexS}=[circle,draw, minimum size=8pt, scale=0.45, inner sep=0.5pt]
       		\begin{tikzpicture}[scale=0.2]
       			\node at (1,2) {  }; 
       			\node (p1) at (1,1) [vertexL]{$p_1$};
       			\node (p2) at (4,1) [vertexL]{$p_2$};
       			\node (p3) at (7,1) [vertexL]{$p_3$};
       			\node (p4) at (10,1) [vertexL]{$p_4$};
       			\node (p5) at (13,1) [vertexL]{$p_5$};
       			\node (p6) at (16,1) [vertexL]{$p_6$};
       			\node (p7) at (19,1) [vertexL]{$p_7$};
       			\node (p8) at (22,1) [vertexL]{$p_8$};
       			\node (p9) at (25,1) [vertexL]{$p_9$};
       			\node (p10) at (28,1) [vertexS]{$p_{10}$};
       			\draw[line width=0.04cm] (p1) -- (p2);
       			\draw[line width=0.04cm] (p2) -- (p3);
       			\draw[line width=0.04cm] (p3) -- (p4);
       			\draw[line width=0.04cm] (p5) -- (p6);
       			\draw[line width=0.04cm] (p6) -- (p7);
       			\draw[line width=0.04cm] (p7) -- (p8);
       			\draw[line width=0.04cm] (p9) -- (p10);
       		\end{tikzpicture} \\ \hline
       		$2/8$ & \tikzstyle{vertexL}=[circle,draw, minimum size=8pt, scale=0.6, inner sep=0.5pt]
       		\begin{tikzpicture}[scale=0.2]
       			\node at (1,2) {  };
       			\node (p1) at (1,1) [vertexL]{$p_1$};
       			\node (p2) at (4,1) [vertexL]{$p_2$};
       			\node (p3) at (7,1) [vertexL]{$p_3$};
       			\node (p4) at (10,1) [vertexL]{$p_4$};
       			\node (p5) at (13,1) [vertexL]{$p_5$};
       			\node (p6) at (16,1) [vertexL]{$p_6$};
       			\draw[line width=0.04cm] (p1) -- (p2);
       			\draw[line width=0.04cm] (p3) -- (p4);
       			\draw[line width=0.04cm] (p4) -- (p5);
       			\draw[line width=0.04cm] (p5) -- (p6);
       		\end{tikzpicture} &  \tikzstyle{vertexL}=[circle,draw, minimum size=8pt, scale=0.6, inner sep=0.5pt]
       		\tikzstyle{vertexS}=[circle,draw, minimum size=8pt, scale=0.45, inner sep=0.5pt]
       		\begin{tikzpicture}[scale=0.2]
       			\node at (1,2) {  };
       			\node (p1) at (1,1) [vertexL]{$p_1$};
       			\node (p2) at (4,1) [vertexL]{$p_2$};
       			\node (p3) at (7,1) [vertexL]{$p_3$};
       			\node (p4) at (10,1) [vertexL]{$p_4$}; 
       			\node (p5) at (13,1) [vertexL]{$p_5$}; 
       			\node (p6) at (16,1) [vertexL]{$p_6$};
       			\node (p7) at (19,1) [vertexL]{$p_7$};
       			\node (p8) at (22,1) [vertexL]{$p_8$};
       			\node (p9) at (25,1) [vertexL]{$p_9$};
       			\node (p10) at (28,1) [vertexS]{$p_{10}$};
       			\draw[line width=0.04cm] (p1) -- (p2);
       			\draw[line width=0.04cm] (p3) -- (p4);
       			\draw[line width=0.04cm] (p4) -- (p5);
       			\draw[line width=0.04cm] (p5) -- (p6);
       			\draw[line width=0.04cm] (p7) -- (p8);
       			\draw[line width=0.04cm] (p8) -- (p9);
       			\draw[line width=0.04cm] (p9) -- (p10);
       		\end{tikzpicture} \\ \hline
       		$1/8$ & \tikzstyle{vertexL}=[circle,draw, minimum size=8pt, scale=0.6, inner sep=0.5pt]
       		\begin{tikzpicture}[scale=0.2]
       			\node at (1,2) {  };
       			\node (p1) at (1,1) [vertexL]{$p_1$};
       			\node (p2) at (4,1) [vertexL]{$p_2$};
       			\node (p3) at (7,1) [vertexL]{$p_3$};
       			\node (p4) at (10,1) [vertexL]{$p_4$};
       			\node (p5) at (13,1) [vertexL]{$p_5$};
       			\node (p6) at (16,1) [vertexL]{$p_6$};
       			\draw[line width=0.04cm] (p2) -- (p3);
       			\draw[line width=0.04cm] (p3) -- (p4);
       			\draw[line width=0.04cm] (p4) -- (p5);
       		\end{tikzpicture} &  \tikzstyle{vertexL}=[circle,draw, minimum size=8pt, scale=0.6, inner sep=0.5pt]
       		\tikzstyle{vertexS}=[circle,draw, minimum size=8pt, scale=0.45, inner sep=0.5pt]
       		\begin{tikzpicture}[scale=0.2]
       			\node at (1,2) {  };
       			\node (p1) at (1,1) [vertexL]{$p_1$};
       			\node (p2) at (4,1) [vertexL]{$p_2$};
       			\node (p3) at (7,1) [vertexL]{$p_3$};
       			\node (p4) at (10,1) [vertexL]{$p_4$}; 
       			\node (p5) at (13,1) [vertexL]{$p_5$}; 
       			\node (p6) at (16,1) [vertexL]{$p_6$};
       			\node (p7) at (19,1) [vertexL]{$p_7$};
       			\node (p8) at (22,1) [vertexL]{$p_8$};
       			\node (p9) at (25,1) [vertexL]{$p_9$};
       			\node (p10) at (28,1) [vertexS]{$p_{10}$};
       			\draw[line width=0.04cm] (p2) -- (p3);
       			\draw[line width=0.04cm] (p3) -- (p4);
       			\draw[line width=0.04cm] (p4) -- (p5);
       			\draw[line width=0.04cm] (p6) -- (p7);
       			\draw[line width=0.04cm] (p7) -- (p8);
       			\draw[line width=0.04cm] (p8) -- (p9);
       		\end{tikzpicture} \\ \hline
       		$1/8$ & \tikzstyle{vertexL}=[circle,draw, minimum size=8pt, scale=0.6, inner sep=0.5pt]
       		\begin{tikzpicture}[scale=0.2]
       			\node at (1,2) {  };
       			\node (p1) at (1,1) [vertexL]{$p_1$};
       			\node (p2) at (4,1) [vertexL]{$p_2$};
       			\node (p3) at (7,1) [vertexL]{$p_3$};
       			\node (p4) at (10,1) [vertexL]{$p_4$};
       			\node (p5) at (13,1) [vertexL]{$p_5$};
       			\node (p6) at (16,1) [vertexL]{$p_6$};
       			\draw[line width=0.04cm] (p2) -- (p3);
       			\draw[line width=0.04cm] (p4) -- (p5);
       		\end{tikzpicture} &  \tikzstyle{vertexL}=[circle,draw, minimum size=8pt, scale=0.6, inner sep=0.5pt]
       		\tikzstyle{vertexS}=[circle,draw, minimum size=8pt, scale=0.45, inner sep=0.5pt]
       		\begin{tikzpicture}[scale=0.2]
       			\node at (1,2) {  };
       			\node (p1) at (1,1) [vertexL]{$p_1$};
       			\node (p2) at (4,1) [vertexL]{$p_2$};
       			\node (p3) at (7,1) [vertexL]{$p_3$};
       			\node (p4) at (10,1) [vertexL]{$p_4$}; 
       			\node (p5) at (13,1) [vertexL]{$p_5$}; 
       			\node (p6) at (16,1) [vertexL]{$p_6$};
       			\node (p7) at (19,1) [vertexL]{$p_7$};
       			\node (p8) at (22,1) [vertexL]{$p_8$};
       			\node (p9) at (25,1) [vertexL]{$p_9$};
       			\node (p10) at (28,1) [vertexS]{$p_{10}$};
       			\draw[line width=0.04cm] (p2) -- (p3);
       			\draw[line width=0.04cm] (p4) -- (p5);
       			\draw[line width=0.04cm] (p5) -- (p6);
       			\draw[line width=0.04cm] (p6) -- (p7);
       			\draw[line width=0.04cm] (p8) -- (p9);
       		\end{tikzpicture} \\ \hline
       		$1/8$ & \tikzstyle{vertexL}=[circle,draw, minimum size=8pt, scale=0.6, inner sep=0.5pt]
       		\begin{tikzpicture}[scale=0.2]
       			\node at (1,2) {  };
       			\node (p1) at (1,1) [vertexL]{$p_1$};
       			\node (p2) at (4,1) [vertexL]{$p_2$};
       			\node (p3) at (7,1) [vertexL]{$p_3$};
       			\node (p4) at (10,1) [vertexL]{$p_4$};
       			\node (p5) at (13,1) [vertexL]{$p_5$};
       			\node (p6) at (16,1) [vertexL]{$p_6$};
       			\draw[line width=0.04cm] (p1) -- (p2);
       			\draw[line width=0.04cm] (p5) -- (p6);
       		\end{tikzpicture} &  \tikzstyle{vertexL}=[circle,draw, minimum size=8pt, scale=0.6, inner sep=0.5pt]
       		\tikzstyle{vertexS}=[circle,draw, minimum size=8pt, scale=0.45, inner sep=0.5pt]
       		\begin{tikzpicture}[scale=0.2]
       			\node at (1,2) {  };
       			\node (p1) at (1,1) [vertexL]{$p_1$};
       			\node (p2) at (4,1) [vertexL]{$p_2$};
       			\node (p3) at (7,1) [vertexL]{$p_3$};
       			\node (p4) at (10,1) [vertexL]{$p_4$}; 
       			\node (p5) at (13,1) [vertexL]{$p_5$}; 
       			\node (p6) at (16,1) [vertexL]{$p_6$};
       			\node (p7) at (19,1) [vertexL]{$p_7$};
       			\node (p8) at (22,1) [vertexL]{$p_8$};
       			\node (p9) at (25,1) [vertexL]{$p_9$};
       			\node (p10) at (28,1) [vertexS]{$p_{10}$};
       			\draw[line width=0.04cm] (p1) -- (p2);
       			\draw[line width=0.04cm] (p9) -- (p10);
       		\end{tikzpicture} \\ \hline
       		$1/8$ & \tikzstyle{vertexL}=[circle,draw, minimum size=8pt, scale=0.6, inner sep=0.5pt]
       		\begin{tikzpicture}[scale=0.2]
       			\node at (1,2) {  };
       			\node (p1) at (1,1) [vertexL]{$p_1$}; 
       			\node (p2) at (4,1) [vertexL]{$p_2$};
       			\node (p3) at (7,1) [vertexL]{$p_3$};
       			\node (p4) at (10,1) [vertexL]{$p_4$}; 
       			\node (p5) at (13,1) [vertexL]{$p_5$}; 
       			\node (p6) at (16,1) [vertexL]{$p_6$}; 
       			\draw[line width=0.04cm] (p2) -- (p3);
       			\draw[line width=0.04cm] (p4) -- (p5);
       		\end{tikzpicture} &  \tikzstyle{vertexL}=[circle,draw, minimum size=8pt, scale=0.6, inner sep=0.5pt]
       		\tikzstyle{vertexS}=[circle,draw, minimum size=8pt, scale=0.45, inner sep=0.5pt]
       		\begin{tikzpicture}[scale=0.2]
       			\node at (1,2) {  };
       			\node (p1) at (1,1) [vertexL]{$p_1$};
       			\node (p2) at (4,1) [vertexL]{$p_2$};
       			\node (p3) at (7,1) [vertexL]{$p_3$};
       			\node (p4) at (10,1) [vertexL]{$p_4$}; 
       			\node (p5) at (13,1) [vertexL]{$p_5$}; 
       			\node (p6) at (16,1) [vertexL]{$p_6$};
       			\node (p7) at (19,1) [vertexL]{$p_7$};
       			\node (p8) at (22,1) [vertexL]{$p_8$};
       			\node (p9) at (25,1) [vertexL]{$p_9$};
       			\node (p10) at (28,1) [vertexS]{$p_{10}$};
       			\draw[line width=0.04cm] (p2) -- (p3);
       			\draw[line width=0.04cm] (p4) -- (p5);
       			\draw[line width=0.04cm] (p6) -- (p7);
       			\draw[line width=0.04cm] (p8) -- (p9);
       		\end{tikzpicture} \\ \hline
       	\end{tabular}
       \end{center}

       {\bf Case 5:} $n=4k+3$, where $k \geq 1$
       
    First consider the case when $p_1 p_{n-1} \not\in E(H)$.  In this case we pick the bipartite subgraphs induced by the following sets (with the given probabilities).
   \begin{itemize}
	\item Use $\{p_1,p_2,p_3,p_4\}, \{p_5,p_6,p_7,p_8\}, \ldots , \{p_{n-6}, p_{n-5},p_{n-4},p_{n-3}\}, \{p_{n-1},p_{n}\}$ with probability $2/8$.
	\item Use $ \{p_1,p_2\}, \{p_3,p_4,p_5,p_6\}, \{p_7,p_8,p_{9},p_{10}\}, \ldots , \{p_{n-4},p_{n-3},p_{n-2},p_{n-1}\}$ with probability $2/8$.
	\item Use $ \{p_2,p_3\}, \{p_4,p_5,p_6,p_7\}, \{p_8,p_9,p_{10},p_{11}\}, \ldots , \{p_{n-3},p_{n-2},p_{n-1},p_{n}\}$ with probability $1/8$.

	\item Use $ \{p_1,p_2\}, \{p_4,p_5,p_6,p_7\}, \{p_8,p_9,p_{10},p_{11}\}, \ldots , \{p_{n-3},p_{n-2},p_{n-1},p_{n}\}$ with probability $1/8$.

		\item Use $\{p_2,p_3,p_4,p_5\}, \{p_6,p_7,p_8,p_{9}\}, \ldots , \{p_{n-5},p_{n-4},p_{n-3},p_{n-2}\}$ with probability $1/8$.
	\item Use $ \{p_2,p_3\}$,$\{p_{n-2},p_{n-1},p_{n},p_1\}$ with probability $1/8$.

\end{itemize}
      Now, every edge of $P$ is in $B_P$ with probability at least $5/8$ and every vertex except $p_2$ is not in $V(B_P)$ with probability at least $1/8$ (and $p_2$ is incident with the edge $p_2 p_n$). We illustrate the above probability space for $n=7$ and $n=11$ below (in the last picture the A's indicate one partite set and the B's the other partite set of the subgraph induced by $\{p_1,p_{n-2},p_{n-1},p_n\}$).

         \begin{center}
       	\begin{tabular}{|c|c|c|} \hline
       		$Prob$ & Subgraph $n=7$ & Subgraph $n=11$ \\ \hline \hline
       		$2/8$ & \tikzstyle{vertexL}=[circle,draw, minimum size=8pt, scale=0.6, inner sep=0.5pt]
       		\begin{tikzpicture}[scale=0.2]
       			\node at (1,2) {  };
       			\node (p1) at (1,1) [vertexL]{$p_1$}; 
       			\node (p2) at (4,1) [vertexL]{$p_2$};
       			\node (p3) at (7,1) [vertexL]{$p_3$};
       			\node (p4) at (10,1) [vertexL]{$p_4$}; 
       			\node (p5) at (13,1) [vertexL]{$p_5$}; 
       			\node (p6) at (16,1) [vertexL]{$p_6$}; 
       			\node (p7) at (19,1) [vertexL]{$p_7$};
       			\draw[line width=0.04cm] (p1) -- (p2);
       			\draw[line width=0.04cm] (p2) -- (p3);
       			\draw[line width=0.04cm] (p3) -- (p4);
       			\draw[line width=0.04cm] (p6) -- (p7);
       		\end{tikzpicture} &  \tikzstyle{vertexL}=[circle,draw, minimum size=8pt, scale=0.6, inner sep=0.5pt]
       		\tikzstyle{vertexS}=[circle,draw, minimum size=8pt, scale=0.45, inner sep=0.5pt]
       		\begin{tikzpicture}[scale=0.2]
       			\node at (1,2) {  };
       			\node (p1) at (1,1) [vertexL]{$p_1$};
       			\node (p2) at (4,1) [vertexL]{$p_2$};
       			\node (p3) at (7,1) [vertexL]{$p_3$};
       			\node (p4) at (10,1) [vertexL]{$p_4$}; 
       			\node (p5) at (13,1) [vertexL]{$p_5$};
       			\node (p6) at (16,1) [vertexL]{$p_6$};
       			\node (p7) at (19,1) [vertexL]{$p_7$};
       			\node (p8) at (22,1) [vertexL]{$p_8$};
       			\node (p9) at (25,1) [vertexL]{$p_9$}; 
       			\node (p10) at (28,1) [vertexS]{$p_{10}$}; 
       			\node (p11) at (31,1) [vertexS]{$p_{11}$};
       			\draw[line width=0.04cm] (p1) -- (p2);
       			\draw[line width=0.04cm] (p2) -- (p3);
       			\draw[line width=0.04cm] (p3) -- (p4);
       			\draw[line width=0.04cm] (p5) -- (p6);
       			\draw[line width=0.04cm] (p6) -- (p7);
       			\draw[line width=0.04cm] (p7) -- (p8);
       			\draw[line width=0.04cm] (p10) -- (p11);
       		\end{tikzpicture} \\ \hline
       		$2/8$ & \tikzstyle{vertexL}=[circle,draw, minimum size=8pt, scale=0.6, inner sep=0.5pt]
       		\begin{tikzpicture}[scale=0.2]
       			\node at (1,2) {  };
       			\node (p1) at (1,1) [vertexL]{$p_1$}; 
       			\node (p2) at (4,1) [vertexL]{$p_2$};
       			\node (p3) at (7,1) [vertexL]{$p_3$}; 
       			\node (p4) at (10,1) [vertexL]{$p_4$};
       			\node (p5) at (13,1) [vertexL]{$p_5$};
       			\node (p6) at (16,1) [vertexL]{$p_6$};
       			\node (p7) at (19,1) [vertexL]{$p_7$};
       			\draw[line width=0.04cm] (p1) -- (p2);
       			  \draw[line width=0.04cm] (p3) -- (p4);
       			\draw[line width=0.04cm] (p4) -- (p5);
       			\draw[line width=0.04cm] (p5) -- (p6);
       		\end{tikzpicture} &  \tikzstyle{vertexL}=[circle,draw, minimum size=8pt, scale=0.6, inner sep=0.5pt]
       		\tikzstyle{vertexS}=[circle,draw, minimum size=8pt, scale=0.45, inner sep=0.5pt]
       		\begin{tikzpicture}[scale=0.2]
       			\node at (1,2) {  };
       			\node (p1) at (1,1) [vertexL]{$p_1$};
       			\node (p2) at (4,1) [vertexL]{$p_2$}; 
       			\node (p3) at (7,1) [vertexL]{$p_3$}; 
       			\node (p4) at (10,1) [vertexL]{$p_4$}; 
       			\node (p5) at (13,1) [vertexL]{$p_5$};
       			\node (p6) at (16,1) [vertexL]{$p_6$};
       			\node (p7) at (19,1) [vertexL]{$p_7$};
       			\node (p8) at (22,1) [vertexL]{$p_8$};
       			\node (p9) at (25,1) [vertexL]{$p_9$};
       			\node (p10) at (28,1) [vertexS]{$p_{10}$}; 
       			\node (p11) at (31,1) [vertexS]{$p_{11}$};
       			\draw[line width=0.04cm] (p1) -- (p2);
       		    \draw[line width=0.04cm] (p3) -- (p4);
       			\draw[line width=0.04cm] (p4) -- (p5);
       			\draw[line width=0.04cm] (p5) -- (p6);
       		   \draw[line width=0.04cm] (p7) -- (p8);
       			\draw[line width=0.04cm] (p8) -- (p9);
       			\draw[line width=0.04cm] (p9) -- (p10);
       		\end{tikzpicture} \\ \hline
       		$1/8$ & \tikzstyle{vertexL}=[circle,draw, minimum size=8pt, scale=0.6, inner sep=0.5pt]
       		\begin{tikzpicture}[scale=0.2]
       			\node at (1,2) {  };
       			\node (p1) at (1,1) [vertexL]{$p_1$};
       			\node (p2) at (4,1) [vertexL]{$p_2$};
       			\node (p3) at (7,1) [vertexL]{$p_3$};
       			\node (p4) at (10,1) [vertexL]{$p_4$};
       			\node (p5) at (13,1) [vertexL]{$p_5$};
       			\node (p6) at (16,1) [vertexL]{$p_6$};
       			\node (p7) at (19,1) [vertexL]{$p_7$};
       			\draw[line width=0.04cm] (p2) -- (p3);
       			\draw[line width=0.04cm] (p4) -- (p5);
       			\draw[line width=0.04cm] (p5) -- (p6);
       			\draw[line width=0.04cm] (p6) -- (p7);
       		\end{tikzpicture} &  \tikzstyle{vertexL}=[circle,draw, minimum size=8pt, scale=0.6, inner sep=0.5pt]
       		\tikzstyle{vertexS}=[circle,draw, minimum size=8pt, scale=0.45, inner sep=0.5pt]
       		\begin{tikzpicture}[scale=0.2]
       			\node at (1,2) {  };
       			\node (p1) at (1,1) [vertexL]{$p_1$};
       			\node (p2) at (4,1) [vertexL]{$p_2$};
       			\node (p3) at (7,1) [vertexL]{$p_3$};
       			\node (p4) at (10,1) [vertexL]{$p_4$}; 
       			\node (p5) at (13,1) [vertexL]{$p_5$};
       			\node (p6) at (16,1) [vertexL]{$p_6$};
       			\node (p7) at (19,1) [vertexL]{$p_7$};
       			\node (p8) at (22,1) [vertexL]{$p_8$};
       			\node (p9) at (25,1) [vertexL]{$p_9$};
       			\node (p10) at (28,1) [vertexS]{$p_{10}$}; 
       			\node (p11) at (31,1) [vertexS]{$p_{11}$};
       			\draw[line width=0.04cm] (p2) -- (p3);
       			\draw[line width=0.04cm] (p4) -- (p5);
       			\draw[line width=0.04cm] (p5) -- (p6);
       			\draw[line width=0.04cm] (p6) -- (p7);
       			\draw[line width=0.04cm] (p8) -- (p9);
       			\draw[line width=0.04cm] (p9) -- (p10);
       			\draw[line width=0.04cm] (p10) -- (p11);
       		\end{tikzpicture} \\ \hline
       		$1/8$ & \tikzstyle{vertexL}=[circle,draw, minimum size=8pt, scale=0.6, inner sep=0.5pt]
       		\begin{tikzpicture}[scale=0.2]
       			\node at (1,2) {  };
       			\node (p1) at (1,1) [vertexL]{$p_1$};
       			\node (p2) at (4,1) [vertexL]{$p_2$};
       			\node (p3) at (7,1) [vertexL]{$p_3$};
       			\node (p4) at (10,1) [vertexL]{$p_4$};
       			\node (p5) at (13,1) [vertexL]{$p_5$};
       			\node (p6) at (16,1) [vertexL]{$p_6$};
       			\node (p7) at (19,1) [vertexL]{$p_7$};
       			\draw[line width=0.04cm] (p1) -- (p2);
       			\draw[line width=0.04cm] (p4) -- (p5);
       			\draw[line width=0.04cm] (p5) -- (p6);
       			 \draw[line width=0.04cm] (p6) -- (p7);
       		\end{tikzpicture} &  \tikzstyle{vertexL}=[circle,draw, minimum size=8pt, scale=0.6, inner sep=0.5pt]
       		\tikzstyle{vertexS}=[circle,draw, minimum size=8pt, scale=0.45, inner sep=0.5pt]
       		\begin{tikzpicture}[scale=0.2]
       			\node at (1,2) {  };
       			\node (p1) at (1,1) [vertexL]{$p_1$};
       			\node (p2) at (4,1) [vertexL]{$p_2$};
       			\node (p3) at (7,1) [vertexL]{$p_3$};
       			\node (p4) at (10,1) [vertexL]{$p_4$}; 
       			\node (p5) at (13,1) [vertexL]{$p_5$};
       			\node (p6) at (16,1) [vertexL]{$p_6$};
       			\node (p7) at (19,1) [vertexL]{$p_7$};
       			\node (p8) at (22,1) [vertexL]{$p_8$};
       			\node (p9) at (25,1) [vertexL]{$p_9$};
       			\node (p10) at (28,1) [vertexS]{$p_{10}$}; 
       			\node (p11) at (31,1) [vertexS]{$p_{11}$};
       			\draw[line width=0.04cm] (p1) -- (p2);
       			\draw[line width=0.04cm] (p4) -- (p5);
       			\draw[line width=0.04cm] (p5) -- (p6);
       			 \draw[line width=0.04cm] (p6) -- (p7);
       			\draw[line width=0.04cm] (p8) -- (p9);
       			\draw[line width=0.04cm] (p9) -- (p10);
       			 \draw[line width=0.04cm] (p10) -- (p11);
       		\end{tikzpicture} \\ \hline
       		$1/8$ & \tikzstyle{vertexL}=[circle,draw, minimum size=8pt, scale=0.6, inner sep=0.5pt]
       		\begin{tikzpicture}[scale=0.2]
       			\node at (1,2) {  };
       			\node (p1) at (1,1) [vertexL]{$p_1$};
       			\node (p2) at (4,1) [vertexL]{$p_2$};
       			\node (p3) at (7,1) [vertexL]{$p_3$};
       			\node (p4) at (10,1) [vertexL]{$p_4$};
       			\node (p5) at (13,1) [vertexL]{$p_5$};
       			\node (p6) at (16,1) [vertexL]{$p_6$};
       			\node (p7) at (19,1) [vertexL]{$p_7$};
       			\draw[line width=0.04cm] (p2) -- (p3);
       			\draw[line width=0.04cm] (p3) -- (p4);
       			\draw[line width=0.04cm] (p4) -- (p5);
       		\end{tikzpicture} &  \tikzstyle{vertexL}=[circle,draw, minimum size=8pt, scale=0.6, inner sep=0.5pt]
       		\tikzstyle{vertexS}=[circle,draw, minimum size=8pt, scale=0.45, inner sep=0.5pt]
       		\begin{tikzpicture}[scale=0.2]
       			\node at (1,2) {  };
       			\node (p1) at (1,1) [vertexL]{$p_1$};
       			\node (p2) at (4,1) [vertexL]{$p_2$};
       			\node (p3) at (7,1) [vertexL]{$p_3$};
       			\node (p4) at (10,1) [vertexL]{$p_4$}; 
       			\node (p5) at (13,1) [vertexL]{$p_5$};
       			\node (p6) at (16,1) [vertexL]{$p_6$};
       			\node (p7) at (19,1) [vertexL]{$p_7$};
       			\node (p8) at (22,1) [vertexL]{$p_8$};
       			\node (p9) at (25,1) [vertexL]{$p_9$};
       			\node (p10) at (28,1) [vertexS]{$p_{10}$}; 
       			\node (p11) at (31,1) [vertexS]{$p_{11}$};
       			\draw[line width=0.04cm] (p2) -- (p3);
       			\draw[line width=0.04cm] (p3) -- (p4);
       			\draw[line width=0.04cm] (p4) -- (p5);
       			\draw[line width=0.04cm] (p6) -- (p7);
       			\draw[line width=0.04cm] (p7) -- (p8);
       			\draw[line width=0.04cm] (p8) -- (p9);
       		\end{tikzpicture} \\ \hline
       		$1/8$ & \tikzstyle{vertexL}=[circle,draw, minimum size=8pt, scale=0.6, inner sep=0.5pt]
       		\begin{tikzpicture}[scale=0.2]
       			\node at (1,2) {  };
       			\node (p1) at (1,1) [vertexL]{$p_1$};\node at (1,-0.4) {{\tiny A}};
       			\node (p2) at (4,1) [vertexL]{$p_2$};
       			\node (p3) at (7,1) [vertexL]{$p_3$};
       			\node (p4) at (10,1) [vertexL]{$p_4$};
       			\node (p5) at (13,1) [vertexL]{$p_5$};\node at (13,-0.4) {{\tiny B}};
       			\node (p6) at (16,1) [vertexL]{$p_6$};\node at (16,-0.4) {{\tiny A}};
       			\node (p7) at (19,1) [vertexL]{$p_7$};\node at (19,-0.4) {{\tiny B}};
       			  \draw[line width=0.04cm] (p2) -- (p3);
       			  \draw[line width=0.04cm] (p5) -- (p6);
       			  \draw[line width=0.04cm] (p6) -- (p7);
       		\end{tikzpicture} &  \tikzstyle{vertexL}=[circle,draw, minimum size=8pt, scale=0.6, inner sep=0.5pt]
       		\tikzstyle{vertexS}=[circle,draw, minimum size=8pt, scale=0.45, inner sep=0.5pt]
       		\begin{tikzpicture}[scale=0.2]
       			\node at (1,2) {  };
       			\node (p1) at (1,1) [vertexL]{$p_1$}; \node at (1,-0.4) {{\tiny A}};
       			\node (p2) at (4,1) [vertexL]{$p_2$};
       			\node (p3) at (7,1) [vertexL]{$p_3$};
       			\node (p4) at (10,1) [vertexL]{$p_4$}; 
       			\node (p5) at (13,1) [vertexL]{$p_5$};
       			\node (p6) at (16,1) [vertexL]{$p_6$};
       			\node (p7) at (19,1) [vertexL]{$p_7$};
       			\node (p8) at (22,1) [vertexL]{$p_8$};
       			\node (p9) at (25,1) [vertexL]{$p_9$}; \node at (25,-0.4) {{\tiny B}};
       			\node (p10) at (28,1) [vertexS]{$p_{10}$}; \node at (28,-0.4) {{\tiny A}};
       			\node (p11) at (31,1) [vertexS]{$p_{11}$}; \node at (31,-0.4) {{\tiny B}};
       			  \draw[line width=0.04cm] (p2) -- (p3);
       			  \draw[line width=0.04cm] (p9) -- (p10);
       			  \draw[line width=0.04cm] (p10) -- (p11);
       		\end{tikzpicture} \\ \hline
       		
       	\end{tabular}
       \end{center}

       We now consider the case when $p_1 p_{n-1} \in E(H)$.  In this case we pick the bipartite subgraphs induced by the following sets (with the given probabilities).
       \begin{itemize}
       	\item Use $\{p_1,p_2,p_3,p_4\}, \{p_5,p_6,p_7,p_8\}, \ldots , \{p_{n-6}, p_{n-5},p_{n-4},p_{n-3}\}, \{p_{n-1},p_{n}\}$ with probability $2/8$.
       	\item Use $ \{p_1,p_2\}, \{p_3,p_4,p_5,p_6\}, \{p_7,p_8,p_{9},p_{10}\}, \ldots , \{p_{n-4},p_{n-3},p_{n-2},p_{n-1}\}$ with probability $2/8$.
       	\item Use $ \{p_2,p_3\}, \{p_4,p_5,p_6,p_7\}, \{p_8,p_9,p_{10},p_{11}\}, \ldots , \{p_{n-3},p_{n-2},p_{n-1},p_{n}\}$ with probability $1/8$.
       	
       	\item Use $ \{p_1,p_2\}, \{p_4,p_5,p_6,p_7\}, \{p_8,p_9,p_{10},p_{11}\}, \ldots , \{p_{n-3},p_{n-2},p_{n-1},p_{n}\}$ with probability $1/8$.
       	
       	\item Use $\{p_2,p_3,p_4,p_5\}, \{p_6,p_7,p_8,p_{9}\}, \ldots , \{p_{n-5},p_{n-4},p_{n-3},p_{n-2}\}, \{p_{n-1},p_n\}$  with probability $1/8$.
       	\item Use $ \{p_2,p_3\}$,$\{p_{n-2},p_{n-1}\}$ with probability $1/8$.
       	
       \end{itemize}
           Now, every edge of $P$ is in $B_P$ with probability at least $5/8$ and every vertex except $p_2$ and $p_{n-1}$ is not in $V(B_P)$ with probability at least $1/8$ (and $p_2$ \AY{and $p_{n-1}$ is incident with the edges $p_2 p_n$ and $p_1 p_{n-1}$, respectively).}
           We illustrate the above probability space for $n=7$  and $n=11$ below.
       
       \begin{center}
       	\begin{tabular}{|c|c|c|} \hline
       		$Prob$ & Subgraph $n=7$ & Subgraph $n=11$ \\ \hline \hline
       		$2/8$ & \tikzstyle{vertexL}=[circle,draw, minimum size=8pt, scale=0.6, inner sep=0.5pt]
       		\begin{tikzpicture}[scale=0.2]
       			\node at (1,2) {  };
       			\node (p1) at (1,1) [vertexL]{$p_1$}; 
       			\node (p2) at (4,1) [vertexL]{$p_2$};
       			\node (p3) at (7,1) [vertexL]{$p_3$};
       			\node (p4) at (10,1) [vertexL]{$p_4$}; 
       			\node (p5) at (13,1) [vertexL]{$p_5$}; 
       			\node (p6) at (16,1) [vertexL]{$p_6$}; 
       			\node (p7) at (19,1) [vertexL]{$p_7$};
       			\draw[line width=0.04cm] (p1) -- (p2);
       			\draw[line width=0.04cm] (p2) -- (p3);
       			\draw[line width=0.04cm] (p3) -- (p4);
       			\draw[line width=0.04cm] (p6) -- (p7);
       		\end{tikzpicture} &  \tikzstyle{vertexL}=[circle,draw, minimum size=8pt, scale=0.6, inner sep=0.5pt]
       		\tikzstyle{vertexS}=[circle,draw, minimum size=8pt, scale=0.45, inner sep=0.5pt]
       		\begin{tikzpicture}[scale=0.2]
       			\node at (1,2) {  };
       			\node (p1) at (1,1) [vertexL]{$p_1$};
       			\node (p2) at (4,1) [vertexL]{$p_2$};
       			\node (p3) at (7,1) [vertexL]{$p_3$};
       			\node (p4) at (10,1) [vertexL]{$p_4$}; 
       			\node (p5) at (13,1) [vertexL]{$p_5$};
       			\node (p6) at (16,1) [vertexL]{$p_6$};
       			\node (p7) at (19,1) [vertexL]{$p_7$};
       			\node (p8) at (22,1) [vertexL]{$p_8$};
       			\node (p9) at (25,1) [vertexL]{$p_9$}; 
       			\node (p10) at (28,1) [vertexS]{$p_{10}$}; 
       			\node (p11) at (31,1) [vertexS]{$p_{11}$};
       			\draw[line width=0.04cm] (p1) -- (p2);
       			\draw[line width=0.04cm] (p2) -- (p3);
       			\draw[line width=0.04cm] (p3) -- (p4);
       			\draw[line width=0.04cm] (p5) -- (p6);
       			\draw[line width=0.04cm] (p6) -- (p7);
       			\draw[line width=0.04cm] (p7) -- (p8);
       			\draw[line width=0.04cm] (p10) -- (p11);
       		\end{tikzpicture} \\ \hline
       		$2/8$ & \tikzstyle{vertexL}=[circle,draw, minimum size=8pt, scale=0.6, inner sep=0.5pt]
       		\begin{tikzpicture}[scale=0.2]
       			\node at (1,2) {  };
       			\node (p1) at (1,1) [vertexL]{$p_1$}; 
       			\node (p2) at (4,1) [vertexL]{$p_2$};
       			\node (p3) at (7,1) [vertexL]{$p_3$}; 
       			\node (p4) at (10,1) [vertexL]{$p_4$};
       			\node (p5) at (13,1) [vertexL]{$p_5$};
       			\node (p6) at (16,1) [vertexL]{$p_6$};
       			\node (p7) at (19,1) [vertexL]{$p_7$};
       			\draw[line width=0.04cm] (p1) -- (p2);
       			\draw[line width=0.04cm] (p3) -- (p4);
       			\draw[line width=0.04cm] (p4) -- (p5);
       			\draw[line width=0.04cm] (p5) -- (p6);
       		\end{tikzpicture} &  \tikzstyle{vertexL}=[circle,draw, minimum size=8pt, scale=0.6, inner sep=0.5pt]
       		\tikzstyle{vertexS}=[circle,draw, minimum size=8pt, scale=0.45, inner sep=0.5pt]
       		\begin{tikzpicture}[scale=0.2]
       			\node at (1,2) {  };
       			\node (p1) at (1,1) [vertexL]{$p_1$};
       			\node (p2) at (4,1) [vertexL]{$p_2$}; 
       			\node (p3) at (7,1) [vertexL]{$p_3$}; 
       			\node (p4) at (10,1) [vertexL]{$p_4$}; 
       			\node (p5) at (13,1) [vertexL]{$p_5$};
       			\node (p6) at (16,1) [vertexL]{$p_6$};
       			\node (p7) at (19,1) [vertexL]{$p_7$};
       			\node (p8) at (22,1) [vertexL]{$p_8$};
       			\node (p9) at (25,1) [vertexL]{$p_9$};
       			\node (p10) at (28,1) [vertexS]{$p_{10}$}; 
       			\node (p11) at (31,1) [vertexS]{$p_{11}$};
       			\draw[line width=0.04cm] (p1) -- (p2);
       			\draw[line width=0.04cm] (p3) -- (p4);
       			\draw[line width=0.04cm] (p4) -- (p5);
       			\draw[line width=0.04cm] (p5) -- (p6);
       			\draw[line width=0.04cm] (p7) -- (p8);
       			\draw[line width=0.04cm] (p8) -- (p9);
       			\draw[line width=0.04cm] (p9) -- (p10);
       		\end{tikzpicture} \\ \hline
       		$1/8$ & \tikzstyle{vertexL}=[circle,draw, minimum size=8pt, scale=0.6, inner sep=0.5pt]
       		\begin{tikzpicture}[scale=0.2]
       			\node at (1,2) {  };
       			\node (p1) at (1,1) [vertexL]{$p_1$};
       			\node (p2) at (4,1) [vertexL]{$p_2$};
       			\node (p3) at (7,1) [vertexL]{$p_3$};
       			\node (p4) at (10,1) [vertexL]{$p_4$};
       			\node (p5) at (13,1) [vertexL]{$p_5$};
       			\node (p6) at (16,1) [vertexL]{$p_6$};
       			\node (p7) at (19,1) [vertexL]{$p_7$};
       			\draw[line width=0.04cm] (p2) -- (p3);
       			\draw[line width=0.04cm] (p4) -- (p5);
       			\draw[line width=0.04cm] (p5) -- (p6);
       			\draw[line width=0.04cm] (p6) -- (p7);
       		\end{tikzpicture} &  \tikzstyle{vertexL}=[circle,draw, minimum size=8pt, scale=0.6, inner sep=0.5pt]
       		\tikzstyle{vertexS}=[circle,draw, minimum size=8pt, scale=0.45, inner sep=0.5pt]
       		\begin{tikzpicture}[scale=0.2]
       			\node at (1,2) {  };
       			\node (p1) at (1,1) [vertexL]{$p_1$};
       			\node (p2) at (4,1) [vertexL]{$p_2$};
       			\node (p3) at (7,1) [vertexL]{$p_3$};
       			\node (p4) at (10,1) [vertexL]{$p_4$}; 
       			\node (p5) at (13,1) [vertexL]{$p_5$};
       			\node (p6) at (16,1) [vertexL]{$p_6$};
       			\node (p7) at (19,1) [vertexL]{$p_7$};
       			\node (p8) at (22,1) [vertexL]{$p_8$};
       			\node (p9) at (25,1) [vertexL]{$p_9$};
       			\node (p10) at (28,1) [vertexS]{$p_{10}$}; 
       			\node (p11) at (31,1) [vertexS]{$p_{11}$};
       			\draw[line width=0.04cm] (p2) -- (p3);
       			\draw[line width=0.04cm] (p4) -- (p5);
       			\draw[line width=0.04cm] (p5) -- (p6);
       			\draw[line width=0.04cm] (p6) -- (p7);
       			\draw[line width=0.04cm] (p8) -- (p9);
       			\draw[line width=0.04cm] (p9) -- (p10);
       			\draw[line width=0.04cm] (p10) -- (p11);
       		\end{tikzpicture} \\ \hline
       		$1/8$ & \tikzstyle{vertexL}=[circle,draw, minimum size=8pt, scale=0.6, inner sep=0.5pt]
       		\begin{tikzpicture}[scale=0.2]
       			\node at (1,2) {  };
       			\node (p1) at (1,1) [vertexL]{$p_1$};
       			\node (p2) at (4,1) [vertexL]{$p_2$};
       			\node (p3) at (7,1) [vertexL]{$p_3$};
       			\node (p4) at (10,1) [vertexL]{$p_4$};
       			\node (p5) at (13,1) [vertexL]{$p_5$};
       			\node (p6) at (16,1) [vertexL]{$p_6$};
       			\node (p7) at (19,1) [vertexL]{$p_7$};
       			\draw[line width=0.04cm] (p1) -- (p2);
       			\draw[line width=0.04cm] (p4) -- (p5);
       			\draw[line width=0.04cm] (p5) -- (p6);
       			\draw[line width=0.04cm] (p6) -- (p7);
       		\end{tikzpicture} &  \tikzstyle{vertexL}=[circle,draw, minimum size=8pt, scale=0.6, inner sep=0.5pt]
       		\tikzstyle{vertexS}=[circle,draw, minimum size=8pt, scale=0.45, inner sep=0.5pt]
       		\begin{tikzpicture}[scale=0.2]
       			\node at (1,2) {  };
       			\node (p1) at (1,1) [vertexL]{$p_1$};
       			\node (p2) at (4,1) [vertexL]{$p_2$};
       			\node (p3) at (7,1) [vertexL]{$p_3$};
       			\node (p4) at (10,1) [vertexL]{$p_4$}; 
       			\node (p5) at (13,1) [vertexL]{$p_5$};
       			\node (p6) at (16,1) [vertexL]{$p_6$};
       			\node (p7) at (19,1) [vertexL]{$p_7$};
       			\node (p8) at (22,1) [vertexL]{$p_8$};
       			\node (p9) at (25,1) [vertexL]{$p_9$};
       			\node (p10) at (28,1) [vertexS]{$p_{10}$}; 
       			\node (p11) at (31,1) [vertexS]{$p_{11}$};
       			\draw[line width=0.04cm] (p1) -- (p2);
       			\draw[line width=0.04cm] (p4) -- (p5);
       			\draw[line width=0.04cm] (p5) -- (p6);
       			\draw[line width=0.04cm] (p6) -- (p7);
       			\draw[line width=0.04cm] (p8) -- (p9);
       			\draw[line width=0.04cm] (p9) -- (p10);
       			\draw[line width=0.04cm] (p10) -- (p11);
       		\end{tikzpicture} \\ \hline
       		$1/8$ & \tikzstyle{vertexL}=[circle,draw, minimum size=8pt, scale=0.6, inner sep=0.5pt]
       		\begin{tikzpicture}[scale=0.2]
       			\node at (1,2) {  };
       			\node (p1) at (1,1) [vertexL]{$p_1$};
       			\node (p2) at (4,1) [vertexL]{$p_2$};
       			\node (p3) at (7,1) [vertexL]{$p_3$};
       			\node (p4) at (10,1) [vertexL]{$p_4$};
       			\node (p5) at (13,1) [vertexL]{$p_5$};
       			\node (p6) at (16,1) [vertexL]{$p_6$};
       			\node (p7) at (19,1) [vertexL]{$p_7$};
       			\draw[line width=0.04cm] (p2) -- (p3);
       			\draw[line width=0.04cm] (p3) -- (p4);
       			\draw[line width=0.04cm] (p4) -- (p5);
       			\draw[line width=0.04cm] (p6) -- (p7);
       		\end{tikzpicture} &  \tikzstyle{vertexL}=[circle,draw, minimum size=8pt, scale=0.6, inner sep=0.5pt]
       		\tikzstyle{vertexS}=[circle,draw, minimum size=8pt, scale=0.45, inner sep=0.5pt]
       		\begin{tikzpicture}[scale=0.2]
       			\node at (1,2) {  };
       			\node (p1) at (1,1) [vertexL]{$p_1$};
       			\node (p2) at (4,1) [vertexL]{$p_2$};
       			\node (p3) at (7,1) [vertexL]{$p_3$};
       			\node (p4) at (10,1) [vertexL]{$p_4$}; 
       			\node (p5) at (13,1) [vertexL]{$p_5$};
       			\node (p6) at (16,1) [vertexL]{$p_6$};
       			\node (p7) at (19,1) [vertexL]{$p_7$};
       			\node (p8) at (22,1) [vertexL]{$p_8$};
       			\node (p9) at (25,1) [vertexL]{$p_9$};
       			\node (p10) at (28,1) [vertexS]{$p_{10}$}; 
       			\node (p11) at (31,1) [vertexS]{$p_{11}$};
       			\draw[line width=0.04cm] (p2) -- (p3);
       			\draw[line width=0.04cm] (p3) -- (p4);
       			\draw[line width=0.04cm] (p4) -- (p5);
       			\draw[line width=0.04cm] (p6) -- (p7);
       			\draw[line width=0.04cm] (p7) -- (p8);
       			\draw[line width=0.04cm] (p8) -- (p9);
       			\draw[line width=0.04cm] (p10) -- (p11);
       		\end{tikzpicture} \\ \hline
       		$1/8$ & \tikzstyle{vertexL}=[circle,draw, minimum size=8pt, scale=0.6, inner sep=0.5pt]
       		\begin{tikzpicture}[scale=0.2]
       			\node at (1,2) {  };
       			\node (p1) at (1,1) [vertexL]{$p_1$};
       			\node (p2) at (4,1) [vertexL]{$p_2$};
       			\node (p3) at (7,1) [vertexL]{$p_3$};
       			\node (p4) at (10,1) [vertexL]{$p_4$};
       			\node (p5) at (13,1) [vertexL]{$p_5$};
       			\node (p6) at (16,1) [vertexL]{$p_6$};
       			\node (p7) at (19,1) [vertexL]{$p_7$};
       			\draw[line width=0.04cm] (p2) -- (p3);
       			\draw[line width=0.04cm] (p5) -- (p6);
       		\end{tikzpicture} &  \tikzstyle{vertexL}=[circle,draw, minimum size=8pt, scale=0.6, inner sep=0.5pt]
       		\tikzstyle{vertexS}=[circle,draw, minimum size=8pt, scale=0.45, inner sep=0.5pt]
       		\begin{tikzpicture}[scale=0.2]
       			\node at (1,2) {  };
       			\node (p1) at (1,1) [vertexL]{$p_1$}; 
       			\node (p2) at (4,1) [vertexL]{$p_2$};
       			\node (p3) at (7,1) [vertexL]{$p_3$};
       			\node (p4) at (10,1) [vertexL]{$p_4$}; 
       			\node (p5) at (13,1) [vertexL]{$p_5$};
       			\node (p6) at (16,1) [vertexL]{$p_6$};
       			\node (p7) at (19,1) [vertexL]{$p_7$};
       			\node (p8) at (22,1) [vertexL]{$p_8$};
       			\node (p9) at (25,1) [vertexL]{$p_9$};
       			\node (p10) at (28,1) [vertexS]{$p_{10}$}; 
       			\node (p11) at (31,1) [vertexS]{$p_{11}$};
       			\draw[line width=0.04cm] (p2) -- (p3);
       			\draw[line width=0.04cm] (p9) -- (p10);
       		\end{tikzpicture} \\ \hline

       	\end{tabular}
       \end{center}\smallQED
       
\vspace{2mm}
       
Recall that we have assumed that $|V(G)|$ is even. Let $B$ be the union of the bipartite subgraphs taken randomly and independently for each cycle and the path from the probability space in Claims \ref{cl:neq5711}-\ref{cl:pathcase}. Note that every edge of the $5$-cycles in $H-M$ is included with probability at least $3/5$ and every other edge of $H-M$ is included with probability at least $5/8$. Let $M=M_{55}' \cup M_5' \cup M'$, where $M_{55}'$ are the edges in $M$ where both end-points lie in $5$-cycles in $H-M$ and $M_{5}'$ are the edges in $M$ where exactly one end-point lies in $5$-cycles in $H-M$ and $M'=M - (M_{55}' \cup M_5')$. Let $C_5'$ be the family of all $5$-cycles in $H-M$.
       
Now, let $B$ include edges $e\in M_{55}'\cup M_{5}'$ if both endpoints of $e$ are not in $B$ at the moment. By Claims \ref{cl:neq5711}-\ref{cl:pathcase}, the endpoints of every edge in $M_{55}'$ are not in $B$ with probability $1/25$ and the endpoints of every edge in $M_{5}'$ are not in $B$ with probability at least $1/40$. Observe that $B$ is still in $\mathcal{B}_b(H)$. Therefore, the following holds. 
       
       \begin{equation}\label{eq:eB}
       \mathbb{E}(w(B)) \geq \frac{3 w(C_5')}{5} + \frac{5 (w(H-M)-w(C_5'))}{8} + \frac{w(M_{55}')}{25} + \frac{w(M_5')}{40}.
       \end{equation} 
       
      Let $B'$ be the union of the bipartite subgraphs taken randomly and independently from the 5-cycles from the probability space in Claim \ref{cl:c5}. Let $B'$ include edges $e\in M$ if both endpoints of $e$ are not in $B'$ at the moment. Since $B'$ only includes vertices on 5-cycles, every edge in $M_{55}'$ can be added to $B'$ with probability at least $1/25$. In addition, every edge in $M_5'$ can be added with probability at least $1/5$, and all edges in $M'$ can be added with probability $1$. Thus, the following holds. 
       
       \begin{equation}\label{eq:eB'}
       \mathbb{E}(w(B')) \geq \frac{3 w(C_5')}{5} + \frac{w(M_{55}')}{25} + \frac{w(M_5')}{5} + w(M').
       \end{equation}
       Then, let $B^*$ be the \AY{bipartite} subgraph obtained by taking $B^*=B$ with probability $24/25$ and $B^*=B'$ with probability $1/25$. Thus, the following holds.
       
\AY{
       \begin{equation}\label{eq:eB*}
\begin{array}{rcl} 
      \mathbb{E}(w(B^*)) & \geq & \frac{24}{25} \left( \frac{3 w(C_5')}{5} + \frac{5 (w(H-M)-w(C_5'))}{8} + \frac{w(M_{55}')}{25} + \frac{w(M_5')}{40} \right) \\ \vspace{0.1cm}
& & + \frac{1}{25} \left( \frac{3 w(C_5')}{5} + \frac{w(M_{55}')}{25} + \frac{w(M_5')}{5} + w(M')   \right) \\ \vspace{0.1cm}
 & = & \frac{3 w(C_5')}{5} + \frac{3 (w(H-M)-w(C_5'))}{5} + \frac{w(M_{55}')}{25} + \frac{4w(M_5')}{125} + \frac{w(M')}{25} \\
 & \geq &  \frac{3(w(H)-w(M))}{5}+\frac{4w(M)}{125}. \\
\end{array}
       \end{equation}
}


              Let $mb(H)$ denote the weight of a maximum bisection in $H$. By Lemma \ref{lem:bbsg}, (\ref{eq:eB*}) implies the following:
       
       \[
       \begin{array}{rcl}
       	mb(H) & \geq & \frac{1}{2} \left( w(H) + \frac{3 (w(H)-w(M))}{5} + \frac{4w(M)}{125} \right) \\
       	& = & \frac{4}{5} w(H) - \frac{71 w(M)}{250}. \\
       \end{array}
       \]
       
       By (\ref{eq:1}), we have that $mb(H) \geq \frac{3}{5} w(H) + \frac{2}{5} w(M) = \frac{3}{5} w(H) + \frac{100}{250} w(M)$. Merging these results give us the following:
       
       \[
       \begin{array}{rcl}
       	mb(H) & \geq & \frac{100}{171} \left( \frac{4}{5} w(H) - \frac{71 w(M)}{250} \right) 
       	+ \frac{71}{171} \left( \frac{3}{5} w(H) + \frac{100}{250} w(M) \right) \\
       	& = & \frac{400 + 213}{5 \times 171} w(H) \\
       	& = & \frac{613}{855} w(H) \approx 0.716959 w(H), \\
       \end{array}     
       \]
        which completes the proof for the even case.
       
Now we consider the case when $|V(G)|$ is odd. We only need to consider the case when $|V_2(G)|=1$. Indeed, if $|V_2(G)|\geq 2$, then $|V_2(G)|\geq 3$ and therefore two of the vertices in $V_2(G)$ are not adjacent to each other (as $G$ is triangle-free). We can add a new vertex $x$ and two edges with weight 0 between $x$ and the two nonadjacent vertices in $V_2(G)$. The new graph is bridgeless with an even number of vertices and therefore has a bisection $(X, Y)$ with the desired weight. Then $(X\setminus\{x\},Y\setminus\{x\})$ is a desired bisection for $G$. Thus, we assume that $V_2(G)=\{y\}$. Since $H$ is triangle-free the two neighbours $y_1$ and $y_2$ of $y$ are not adjacent. Let us obtain a new graph $H'$ by adding $y_1y_2$ with weight 0 and deleting $y$. Note that $H'$ is bridgeless cubic and therefore by Lemma \ref{lem:mt-gvpt}, there is a matching $M_1$ in $H'$ not containing $y_1y_2$. Thus, $M_1$ is also a matching in $G$ that includes all vertices in $G$ but $y$. Now, we add a new vertex $x$ and a new edge $xy$ of weight 0 to $G$ and denote the new graph by $H''$. Observe that $M=M_1\cup \{xy\}$ is a perfect matching of $H''$ and $H''-M$ is a disjoint union of cycles and one isolated vertex $x$. Note that a bisection $(X,Y)$ of $H''$ and the bisection $(X\setminus\{x\},Y\setminus\{x\})$ of $G$ are of the same weight. Thus, we only need to show a desired bound for the bisection of $H''$. 

The argument is similar to the above one for 2-edge-connected graphs. We first use Lemma \ref{lem:triangle-free cubic} to show that $mb(H'')\geq \frac{3}{5}w(H'')+\frac{2}{5}w(M)$ (recall that Lemma \ref{lem:triangle-free cubic} does not require the graph to be 2-edge-connected). Now we only need to construct a random bipartite subgraph $B^*$ with $\mathbb{E}(w(B^*)) \geq \frac{3(w(H'')-w(M))}{5}+\frac{4w(M)}{125}$. Then, again by Lemma \ref{lem:bbsg}, $mb(H'')\geq  \frac{4}{5} w(H'') - \frac{71 w(M)}{250}$. Thus, a convex combination of this lower bound and $mb(H'')\geq \frac{3}{5}w(H'')+\frac{2}{5}w(M)$ implies our result. In the following, we illustrate how to use a similar argument to construct the desired random bipartite subgraph $B^*$ with $\mathbb{E}(w(B^*)) \geq \frac{3(w(H'')-w(M))}{5}+\frac{4w(M)}{125}$.  Let $M'_{55}$, $M'_5$, $M'$, and $C'_5$ have the same definition as above but for this new graph $H''$ and the new matching $M$. 

\begin{itemize}
	\item  Let $B$ be the union of the bipartite subgraphs taken randomly and independently for each cycle in $H''-M$ from the probability space in Claims \ref{cl:neq5711}-\ref{cl:c11}. (Note that $H''-M$ is a collection of cycles and one isolated vertex $x$ so we don't need to use Claim \ref{cl:pathcase}.) Let $B$ include edges $e\in M_{55}'\cup M_{5}'$ if both endpoints of $e$ are not in $B$ at the moment. By a discussion similar to that for (\ref{eq:eB}), we have that $\mathbb{E}(w(B)) \geq \frac{3 w(C_5')}{5} + \frac{5 (w(H''-M)-w(C_5'))}{8} + \frac{w(M_{55}')}{25} + \frac{w(M_5')}{40}$. 
	
	\item Let $B'$ be the union of the bipartite subgraphs taken randomly and independently for each 5-cycle in $H''-M$ from the probability space in Claim \ref{cl:c5}. Let $B'$ include edges $e\in M$ if both endpoints of $e$ are not in $B'$ at the moment. By a discussion similar  to that for (\ref{eq:eB'}), we have that $ \mathbb{E}(w(B')) \geq \frac{3 w(C_5')}{5} + \frac{w(M_{55}')}{25} + \frac{w(M_5')}{5} + w(M') $. 
	
	\item Let $B^*$ be the \AY{bipartite} subgraph obtained by taking $B^*=B$ with probability $24/25$ and $B^*=B'$ with probability $1/25$. And we have $\mathbb{E}(w(B^*)) \geq \frac{3(w(H'')-w(M))}{5}+\frac{4w(M)}{125}$ as desired.
 
\end{itemize}
 
This completes the proof.
\end{proof}

\begin{remark1}
The method for 2-edge-connected triangle-free subcubic graphs cannot be used directly for triangle-free subcubic graphs with bridges in general because the graph may not have a perfect matching and therefore we cannot apply Lemma \ref{lem:triangle-free cubic} to get the first lower bound. Thus, we need more efforts to extend this result to all cases but $K_{1,3}$.
\end{remark1}
The following remark observed from the proof of the above theorem is useful for extending the result to all triangle-free subcubic graphs but $K_{1,3}$. 
\begin{observation} \label{obsX1}
By $(\ref{eq:eB*})$, we note that if $M$ is a matching $($not necessarily perfect since the second bound of the above theorem did not use this fact$)$ of a bridgeless triangle-free subcubic graph $H$ and $H-M$ is a collection of cycles, then we can also use the probability space of bipartite subgraphs $B^*\in \mathcal{B}_b(G)$ defined in the proof of the above theorem $($obtained by taking $B$ with probability $\frac{24}{25}$ and $B'$ with probability $\frac{1}{25}$$)$ of $H-M$, where $\mathbb{E}(w(B^*))\geq\frac{3 (w(H)-w(M))}{5} + \frac{4w(M)}{125}$. In addition, for every vertex $v$ on the cycles in $H-M$, the probability of the \AY{following event 

\begin{itemize}
\item $v$ is not in $V(B^*)$ or the matching edge in $M$ incident to $v$ is in $B^*$,
\end{itemize}
}

is at most $\frac{24}{25}\times \frac{1}{4}+\frac{1}{25}\times 1=\frac{7}{25}$ since if \AY{the event is true} then $v$ must not \AY{be} in $B$ before we start including the edges in $M$ and this probability is at most $1/4$ as shown by Claims $\ref{cl:neq5711}$-$\ref{cl:c11}$. 
\end{observation}

If $G$ is an unweighted \AY{bridgeless cubic}      graph, then by Petersen's theorem \footnote{\YC{Petersen's theorem states that every cubic graph with at most 2 cut edges contains a perfect matching}}  it has a perfect matching with $\frac{|E(G)|}{3}$ edges. Thus, by Lemma \ref{lem:triangle-free cubic}, we can confirm this case for Conjecture \ref{conj:triangle-free cubic}. In the next subsection, we will extend our bound from bridgeless to all subcubic triangle-free graphs not isomorphic to $K_{1,3}$. Also note that to prove Conjecture \ref{conj:triangle-free cubic} for bridgeless cubic graphs, using Lemma \ref{lem:triangle-free cubic}, all we need to show is that the following weighted version of Peterson's theorem holds, which is a consequence of the famous Berge-Fulkerson conjecture (see e.g. \cite{Seymour:BFC}). We believe that this conjecture of ours is interesting in its own right.
\begin{conjecture}\label{cj:2}
	If $ G $ is a weighted bridgeless cubic  graph, then there exists a perfect matching $M$ such that $w(M)\geq \frac{w(G)}{3}$.
\end{conjecture}

\subsection{Extending bridgeless case to all cases except a claw}
Let $G$ be a triangle-free subcubic graph and $H$ a 2-edge-connected component of $G$. We say that $H$ is a {\em leaf} component if it \AY{is} only incident to one bridge in $G$. We say that an odd leaf component $H$ of $G$ is {\em tight} if, and only if, all but one \AY{vertex has} degree 3 in $H$ and exactly one vertex has degree two in $H$ (this vertex is incident to a bridge in $G$). We say $H$ is trivial if $|V(H)|=1$ and otherwise nontrivial.

We call a bridgeless triangle-free subcubic graph of odd order {\em tight} if all but one vertex is of degree 3 and only one vertex has degree 2. Recall that $\theta = \frac{613}{855} \approx 0.716959$.
\begin{theorem}        \label{TightProb}
Let $G$ be a tight bridgeless triangle-free subcubic graph. Let $x$ be the degree 2 vertex in $G$. Then we can find a random bisection of $G$ with average weight at least $\theta \cdot w(G)$, such that the $x$ belongs to the larger part with probability at most $0.79$ and at least $0.74$.
\end{theorem}
\begin{proof}
 We construct a new graph $G'$ from $G$ by adding a new vertex $y$ and the edge $xy$ with weight 0. Let $M'$ be the perfect matching containing $xy$ in $G'$ (one can find this perfect matching by the same argument in the paragraph before the last of the proof of Theorem \ref{thm:main}). Then $G'-M'$ is a collection of cycles with an isolated vertex $y$. Let $M=M'-xy$. Next we will prove two claims.

\2
\begin{claim}\label{cl:b1}
 There exists a bisection of $G$ with weight at least $\frac{3}{5}w(G)+\frac{2}{5}w(M)$, where $x$ is in the larger part.
\end{claim}
\2

{\bf The proof of Claim \ref{cl:b1}.} By Lemma \ref{lem:triangle-free cubic}, $G'$ has a bisection $(X,Y)$ with weight at least $\frac{3}{5}w(G')+\frac{2}{5}w(M')$ such that $xy\in (X,Y)$ and therefore $x$ and $y$ are in different parts. Observe that $(X\setminus \{y\},Y\setminus \{y\})$ is now a desired bisection of $G$. This completes the proof of Claim \ref{cl:b1}. \smallQED

\2
\begin{claim}\label{cl:b2}
There exists a random bisection of $G$ with average weight at least  $\frac{4}{5} w(G) - \frac{71 w(M)}{250}$  and such that $x$ belongs to the larger set with probability at 
most $16/25$ and at least $9/16$.
\end{claim}
\2

{\bf The proof of Claim \ref{cl:b2}.} Let $B^*$ be the random bipartite subgraph described in Observation \ref{obsX1}. Let $p_x$ be the probability that $x$ is not in $V(B^*)$. Since $x$ is not incident to any edge in $M$, we can always leave $x$ unpaired as long as $x\notin V(B^*)$ when applying Lemma \ref{lem:bbsg}. Thus, if $x\notin V(B^*)$, then $x$ will always be in the larger part, and the expectation of the weight of the bisection obtained from $B^*$ by applying Lemma \ref{lem:bbsg} is still $\frac{4}{5} w(G) - \frac{71 w(M)}{250}$. Recall that by Claims \ref{cl:neq5711}-\ref{cl:c11}, $p_x\geq \frac{1}{8}$ and by Observation~\ref{obsX1}, $p_x\leq 7/25$. If $x\in V(B^*)$ then there is a 50\% chance that $x$ will belong to the larger part and a 50\% chance it will be in the smaller part when applying Lemma \ref{lem:bbsg}. So the probability that $x$ will belong to the larger part is $p_x+0.5\times (1-p_x)$ which is at most $7/25 + 0.5 \times 18/25 = 16/25$ and at least $1/8+0.5\times 7/8=9/16$. This completes the proof of Claim \ref{cl:b2}. \smallQED

\2 

As in our proof for Theorem \ref{thm:main}, we now obtain the following by picking the bisection from Claim~\ref{cl:b1} with probability $\frac{71}{171}$ and the bisection from Claim~\ref{cl:b2} with probability  \AY{$\frac{100}{171}$,}

\[
\begin{array}{rcl} \vspace{0.1cm}
mb(G) & \geq & \frac{100}{171} \left( \frac{4}{5} w(G) - \frac{71 w(M)}{250} \right)
          + \frac{71}{171} \left( \frac{3}{5} w(G) + \frac{100}{250} w(M) \right) \\ \vspace{0.1cm}
& = & \frac{400 + 213}{5 \cdot 171} w(G) \\
& = & \frac{613}{855} w(G) = \theta \cdot w(G) \approx 0.716959 w(G). \\
\end{array}
\]

We furthermore note that the probability that $x$ belongs to the larger part in the bisection is at most the following:
\[
\frac{100}{171} \times \frac{16}{25} +  \frac{71}{171} \times 1 = \frac{135}{171} < 0.79.
\]
The probability that $x$ belongs to the larger part in the bisection is at least the following:
\[
\frac{100}{171} \times \frac{9}{16} +  \frac{71}{171} \times 1 = \frac{225+284}{4 \cdot 171} > 0.74,
\]
which completes the proof.
\end{proof}     

\YC{Recall that for any $i\in \{2,3\}$, $V_i(G)$ denotes the set of vertices with degree $i$ in $G$. }

\begin{lemma}\label{lem:2component}
Let $G$ be a triangle-free subcubic graph with two 2-edge-connected components $H_1$ and $H_2$, where $|V(H_1)|$ is odd, $V_2(H_1)=N_{H_1}[v]=\{v,v_1,v_2\}$, $V(H_2)=\{u\}$ and $vu$ is bridge. Then, $G$ has a bisection with weight at least $\theta\cdot w(G)$. 
\end{lemma}
\begin{proof}
Let $H'$ be the multigraph obtained from $H_1-v$ by adding two parallel edges $v_1v_2$ with weight 0. Since $H'$ is 2-edge-connected cubic, by Lemma \ref{lem:mt-gvpt}, $H'$ has a perfect matching $M'$ that includes the edge incident to $v_1$ in $H_1-v$ and therefore does not include $v_1v_2$. Then, $M'$ is also a perfect matching of $H_1-v$. Thus, $M=M'\cup\{vw\}$ is a perfect matching of $G$. By Lemma \ref{lem:triangle-free cubic}, $G$ has a bisection $B_a$ with weight at least $\frac{3}{5}w(G)+\frac{2}{5}w(M)$. 

	Observe that $G-M$ is a collection of cycles, one path $P=v_1vv_2$ and one isolated vertex \AY{$u$.}
	
	 \begin{claim}\label{cl:4e}
Let $C$ be a cycle in $G-M$. If $|C|=0 (mod~4)$, then there exists a $B_C\in \mathcal{B}_b(G[V(C)])$ such that $w(B_C)\geq \frac{3}{4}w(C)$. If $|C|=2 (mod~4)$, then there exists a $B_C\in \mathcal{B}_b(G[V(C)])$ such that $w(B_C)\geq \frac{7}{10}w(C)$.
	\end{claim}
	{\bf The proof the Claim \ref{cl:4e}. } Assume that $C=x_1x_2\dots x_nx_1$. If $n=4$ or 6, then since $G$ is triangle-free, $C$ is an induced cycle in $G$ and therefore $B_C=C$ is a bipartite subgraph in $\mathcal{B}_b(G[V(C)])$. 
		
		Now we assume $n\geq 8$, choose $i$ from $[n]$ randomly such that $i$ take any number in $[n]$ with probability $1/n$. Let $B_C\in \mathcal{B}_b(G[V(C)])$ formed by 
		\begin{itemize}
		 \item $\cup_{t=0}^{\frac{n-4}{4}} G[\{x_{i+4t},x_{i+4t+1}, x_{i+4t+2},x_{i+4t+3}\}]$ when $n=0 (mod~4)$,\\
		 \item $(\cup_{t=0}^{\frac{n-6}{4}} G[\{x_{i+4t},x_{i+4t+1}, x_{i+4t+2},x_{i+4t+3}\}])\cup \{x_{i+n-2}x_{i+n-1}\}$ when $n=2 (mod~4)$,
		\end{itemize}
		where indices are taken modulo $n$. Observe that each edge of $C$ is in $B_C$ with the same probability which is 
$\frac{3}{4}$ when \AY{$n=0(mod~4)$} and $\frac{3n-2}{4n}$ when \AY{$n=2(mod~4)$.} By the linearity of expectation there exists a bipartite subgraph $B_C\in \mathcal{B}_b(G[V(C)])$ with weight 
		at least $\frac{3}{4}w(C)$ when $n=0(mod~4)$ and with weight 
		at least $\frac{3n-2}{4n}w(C)\geq \frac{7}{10}w(C)$ when $n=2(mod~4)$. \smallQED
		
		\vspace{2mm}
	
Now we continue with the proof of the lemma.	If $G-M$ has no odd cycle, then let $C$ be an even cycle in $G-M$. By Claim \ref{cl:4e}, let $B_C\in G[V(C)]$ with $w(B_C)\geq \frac{7}{10}w(C)$ and let $G_1=G[\{v_4,v_5,v_6,v_7\}]$ be one of the components in $B_C$. Observe that either $w(B_C-v_5)$ or $w(B_C-v_7)$ has weight at least $\frac{w(B_C)}{2}$. \YC{Assume that $w(B_C-v_i)\geq \frac{w(B_C)}{2}\geq \frac{7}{20} w(C)$ where $i\in \{5,7\}$. } Furthermore, we assume without loss of generality that $w(vv_1)\geq \frac{1}{2}w(P)$. Let $B_1=(B_C-V(G_1))\cup \YC{G[\{v_4,v_6\},\{v_{12-i},u\}]\cup G[\{v,v_i\},\{v_1,v_2\}]}$ and $B_2=B_C\cup\{vv_1\}$. Observe that both $B_1$ and $B_2$ are in $\mathcal{B}_b(G)$. In addition,
\[
		w(B_1)\geq \frac{7}{20} w(C)+w(P)=\frac{7}{20}(w(C)+w(P))+\frac{13}{20}w(P),
\]
	and 
\[
		w(B_2)\geq \frac{7}{10}w(C)+\frac{1}{2}w(P)=\frac{7}{10}(w(C)+w(P))-\frac{1}{5}w(P).\]

Thus, 
	\[
	\max\{w(B_1),w(B_2)\} \geq \frac{4}{17}w(B_1)+\frac{13}{17}w(B_2)\geq\frac{21}{34}(w(C)+w(P)).\]
	
Without loss of generality, we assume that $w(B_1)\geq \frac{21}{34}(w(C)+w(P))$. By Claim \ref{cl:4e}, there exists a bipartite subgraph $B_3$ in $\mathcal{B}_b(G-V(C)-V(P))$ with weight at least $\frac{7}{10}(w(G)-w(M)-w(C)-w(P))\geq \frac{21}{34}(w(G)-w(M)-w(C)-w(P))$. By Lemma \ref{lem:bbsg}, $G$ has a bisection $B_b$ with weight at least $\frac{1}{2}(w(G)+w(B_1)+w(B_3))\geq \frac{55}{68}w(G)-\frac{21}{68}w(M)$. Thus, 

\begin{eqnarray*}
\max\{w(B_a),w(B_b)\}&\geq& \frac{105}{241}\left(\frac{3}{5}w(G)+\frac{2}{5}w(M)\right)+\frac{136}{241}\left(\frac{55}{68}w(G)-\frac{21}{68}w(M)\right) \\
&=& \frac{173}{241}w(G)>\theta \cdot w(G),
\end{eqnarray*}
which completes the proof of this case.

 Thus, we now assume that $G-M$ has an odd cycle $C$. If there are no 5-cycles in $G-M$, then by Claims \ref{cl:neq5711}-\ref{cl:c11}, there is a bipartite subgraph $B$ in $\mathcal{B}_b(G-V(P))$ with weight at least $\frac{5 (w(G)-w(M)-w(P))}{8}$ and every vertex on an odd cycle of $G-M$ in $V(G)\setminus V(B)$ can be used to balance $P$ so that we have a bipartite subgraph in $\mathcal{B}_b(G)$ with weight at least $w(B)+w(P)$. Thus, by Lemma \ref{lem:bbsg}, $G$ has a bisection with weight at least $\frac{13w(G)}{16}-\frac{5w(M)}{16}$. Therefore,
 
 \begin{eqnarray*}
\max\{w(B_a),w(B_c)\}&\geq& \frac{25}{57}\left(\frac{3}{5}w(G)+\frac{2}{5}w(M)\right)+\frac{32}{57}\left(\frac{13}{16}w(G)-\frac{5}{16}w(M)\right) \\
&=& \frac{41}{57}w(G)>\theta \cdot w(G),
 \end{eqnarray*}
 
Thus, we assume that $C$ is a 5-cycle in $G-M$. Let $B^*$ be the random bipartite graph for \AY{$G-M-(V(P)\cup\{u\})$} discribed in Observation \ref{obsX1}. The probability $p_1$ of $V(C)\subseteq V(B^*)$ is the probability that the matching edge in $M$, incident to the isolated vertex in $C$, is contained in $B$, which is at most $7/25$ by Observation \ref{obsX1}. Therefore, there is $1-p_1\geq 18/25$ chance that there is a vertex $v'\in V(C)\setminus V(B^*)$. Let $0<p_2<1$, such that $(1-p_1)\times p_2=\frac{5}{8}$. Thus, when there is such a vertex $v'$, with probability $p_2$, we add $G[v_1,v_2,v,v']$ to $B^*$. Thus, with probability $3/8$, we leave $V(P)$ as isolated vertices, and therefore we can add edges between $V(P)$ and $V(G)\setminus V(P) $ to $V(B^*) $ when the endpoints of the edges in $ V(G)\setminus V(P)$ is not in $V(B^*) $ at the moment. One can show that $\mathbb{E}(w(B^*))\geq \frac{3(w(G)-w(M))}{5}+\frac{4w(M)}{125}$ by observing that the edges in $V(P)$ is included in $B^*$ with probability $5/8$ and every edge in $M$ between $V(G)\setminus V(P)$ and $V(P)$ is included in $B^*$ with probability at least $\frac{3}{8}\times \frac{1}{8}>\frac{1}{25}$. 
So, we are done by a discussion similar  to that in Theorem \ref{thm:main}.
\end{proof}

\begin{theorem}         
Every triangle-free subcubic graph $G$, different than the claw, has a bisection with weight at least $\theta \cdot w(G)$.
\end{theorem}           

\begin{proof}     

Let $c_2(G)$ be the number of 2-edge-connected components in $G$. 
We prove this by induction on \AY{the lexicographical order of $(|V(G)|,c_2(G))$.} 
It is trivial when $|V(G)|\leq 3$, and by Theorem \ref{thm:main}, it is true when $c_2(G)=1$. 
	
We assume that $|V(G)|\geq 4$ and $c_2(G)\geq 2$, and this theorem is true for graphs with order less than $|V(G)|$ or \AY{with the order equal to $|V(G)|$ but with fewer} 2-edge-connected components than $G$. If there is a leaf component $H$ with an even order, then by the induction hypothesis, $G-V(H)$ and $H$ have bisections $(X_1, Y_1)$ and $(X_2, Y_2)$ with the desired proportion of weight. Assume without loss of generality that the bridge is between $X_1$ and $Y_2$ then $w(X_1\cup X_2, Y_1\cup Y_2)$ is a bisection of $G$ with weight at least $\theta\cdot w(G)$. Thus, we assume that all leaf components of $G$ are of odd order.

\2

{\bf Case 1:} {\em No 2-edge-connected component in $G$ is tight.}

\2

{\bf Proof for Case 1.} We first assume that $G$ has a nontrivial leaf component. Let $H_1$ be \AY{a nontrivial leaf component and let $H_2$ be} another leaf component. Since $G$ has no tight component, there exist vertices $v\in V(H_1)$ and $w\in V(H_2)$ such that $d_G(v)=2$ and $d_G(w)\leq 2$. \YC{If $c_2(G)>2$, then there is a leaf component such that none of the cut edges is between $H_1$ and this component. Thus, we may further assume that $H_2$ is this component. One can observe that $v$ and $w$ have no common neighbour (as $v$ is not incident to any cut edges). If $c_2(G)=2$ and $H_2$ is nontrivial, one can also observe that $v$ and $w$ have no common neighbour (as both $v$ and $w$ are not incident to any cut edges). Thus, in the above cases, we can add the edge $vw$ with weight 0. $G+vw$ is still triangle-free subcubic and $c_2(G+vw)<c_2(G)$.} By the induction hypothesis, $G+vw$ has a desired bisection, \AY{which corresponds} to a desired bisection of $G$. Now assume that $c_2(G)=2$ and $H_2$ is trivial. Let $xy$ be the bridge between $H_1$ and $H_2$, and $V(H_2)=\{y\}$. Since $H_1$ is not tight we have that $G$ has at least two vertices of degree 2, i.e. $|V_2(H_1)|\geq 2$. Since $|V(H_1)|$ is odd, $|V_2(H_1)|$ is odd and therefore at least 3. If $V_2(H_1)\setminus N_{H_1}[x]\neq \emptyset$, then let $x'$ be a vertex in this set. We add edge $x'y$ with weight 0. The resulting graph is 2-edge-connected and triangle-free, and every bisection of the resulting graph is a bisection of $G$ with the same weight. Therefore, we are done by Theorem \ref{thm:main}. Thus, we may assume that $V_2(H_1)=N_{H_1}[x]$. Then, we are done by Lemma \ref{lem:2component}.
	
Now, we assume that all leaf components of $G$ are trivial. Since $|V(G)|\geq 4$, all leaves are not adjacent to each other. If two leaves $v$ and $w$ in $G$ do not have a common neighbour, we can add an edge $vw$ with weight 0. Then, $G+vw$ is still a triangle-free subcubic graph, and we are done by applying the induction hypothesis. Thus, every pair of leaves have a common neighbour, and therefore $G\cong K_{1,3}$ as $|V(D)|\ge 4$ and $\Delta(G)\le 3$.

\2

{\bf Case 2:}  {\em There exists a 2-edge-connected component in $G$ that is tight.}

\2

{\bf Proof for Case 2.} Let $G_S$ be a tight 2-edge-connected component in $G$ and let $G_T=G[V(G) \setminus V(G_S)]$. Let $S=V(G_S)$ and let $T=V(G_T)$ and let
$uv$ be the unique $(S,T)$-edge in $G$, which exists as $G_S$ only has one vertex of degree two and all other vertices are of degree 3 and $G$ is connected, but not 2-edge-connected.
Note that $(S,T)$ partitions $V(G)$. Let $(X_S,Y_S)$ be a maximum weight bisection of $G_S$ and let $(X_T,Y_T)$ be a maximum weight  bisection of $G_T$. Note that $|S|$ is odd.
We now consider the following subcases, where subcases 2.C-2.E exhaust all possibilities (Subcases 2.A and 2.B are used in the other subcases).

\2

{\bf Subcase 2.A:} {\em $u$ belongs to the larger set of $\{X_S,Y_S\}$ and $G[T \cup \{u\}]$ or $G_T$ is a claw.}

\2

{\bf Proof for Subcase 2.A.} We first consider the case when $G_T$ is a claw.
In this case assume $T=\{v,x_1,x_2,y\}$ and $y$ is the center of the claw $G_T$ (recall that $uv$ is a bridge in $G$).
Without loss of generality assume that $u \in Y_S$ and note that in this case we have $|Y_S|>|X_S|$.
Consider the bisection $B=(X_S \cup \{v,x_1,x_2\}, Y_S \cup \{y\} )$ of $G$.
The weight of $B$ is at least $\theta \cdot w(G_S) + (w(G)-w(G_S)) \geq \theta \cdot w(G)$ as desired.
This completes the case when $G_T$ is a claw.

\2

Now assume that $G[T \cup \{u\}]$ is a claw.
Let $w_S^{large}$ denote the maximum weight of a bisection of $G_S$ where $u$ belongs to the larger set. 
By the statement of Subcase 2.A we note that $w_S^{large}$ is equal to the weight for the bisection $(X_S,Y_S)$. 
Let $w_S^{small}$ denote the maximum weight of a bisection of $G_S$ where $u$ belongs to the smaller set. 
Note that $w_S^{small} \leq w_S^{large}$. 
Let $p$ denote the probability of $u$ belonging to the larger set in the random bisection created in Theorem~\ref{TightProb}.
By Theorem~\ref{TightProb} we note that $p < 0.79 < (3-3\theta)$, as $\theta \approx 0.716959$. The following now holds.
\begin{equation}\label{(*)}
\hspace{1cm} (3-3\theta) \cdot w_S^{large} + (3\theta -2) \cdot w_S^{small} \geq p \cdot w_S^{large} + (1-p) \cdot w_S^{small}  \geq \theta \cdot w(G_S)
\end{equation}
Let $(X_S^s,Y_S^s)$ denote the maximum weight bisection of $G_S$ where $u$ belongs to the smaller set and recall that the weight of this bisection is $w_S^{small}$.
Let $V(G[T \cup \{u\}])= \{u,v,x_1,x_2\}$, where $v$ is the center of the claw.
Without loss of generality assume that $u \in Y_S$ and $u \in Y_S^s$.
Now consider the following bisections.
\[
\begin{array}{ccc}
B_1=(X_S^s \cup \{v\} ,Y_S^s \cup \{x_1,x_2\} ), & &   %
B_2=(X_S \cup \{x_1,x_2\}    ,Y_S \cup \{v\}), \\        
B_3=(X_S \cup \{v,x_2\}      ,Y_S \cup \{x_1\}),     & &   
B_4=(X_S \cup \{v,x_1\}      ,Y_S \cup \{x_2\}).  \\        
\end{array}
\]
 If we pick bisection $B_1$ with probability $3\theta -2$ (which is between $0$ and $1$ as $\theta \approx 0.716959$) 
and each of the other bisections with a probability of $1-\theta$, then 
we note that the average weight of our bisection will be the following, by (\ref{(*)}).
\[
\begin{array}{rclc} 
\multicolumn{3}{l}{\vspace{0.1cm} 3(1-\theta) \cdot w_S^{large} + (3\theta -2) \cdot w_S^{small} + (3\theta -2 + 2(1-\theta)) \cdot w(E(G[T \cup \{u\}]))} & \\ \vspace{0.1cm}
& = & (3-3\theta) \cdot w_S^{large} + (3\theta -2) \cdot w_S^{small} + \theta \cdot w(G[T \cup \{u\}]) & \\ \vspace{0.1cm}
& \geq & \theta \cdot w(G_S) + \theta \cdot w(G[T \cup \{u\}]) & \\
& = & \theta \cdot w(G) .& \\ 
\end{array}
\]
Therefore there must exist a bisection of $G$ with weight at least $\theta \cdot w(G)$.

\2

{\bf Subcase 2.B:} {\em $u$ belongs to the smaller set of $\{X_S,Y_S\}$ and $G_T$ is a claw or an isolated vertex.}

\2

{\bf Proof for Subcase 2.B.} Thus, $(X_S,Y_S)$ has weight $w_S^{small}$ \AY{(where $w_S^{small}$ is defined as in Subcase 2.A).} Without loss of generality assume that $u \in Y_S$, and therefore $|Y_S|<|X_S|$. First consider the case when $G_T$ is an isolated vertex, i.e., $V(G_T)=\{v\}$. Let $(X_S^l,Y_S^l)$ be the bisection with weight $w_S^{large}$ and $u\in X^l_S$. With probability $\theta$, we choose $(X^l_S,Y^l_S\cup\{v\})$. With probability $1-\theta$, we choose $(X_S,Y_S\cup\{v\})$. Thus, the expectation of the weight of our bisection will be the following:
	\[
	\theta\cdot w_S^{large}+(1-\theta)\cdot w_S^{small}+\theta\cdot w(uv)\geq 	p\cdot w_S^{large}+(1-p)\cdot w_S^{small}+\theta\cdot w(uv)\geq \theta \cdot w(G),
	\]
where the second inequality follows from $w_S^{small}\geq w_S^{large}$ and $p>0.74>\theta$.

Now consider the case when $G_T$ is a claw and 
assume that $T=\{v,x_1,x_2,y\}$ and $y$ is the center of the claw $G_T$ (recall that $uv$ is a bridge in $G$).
Consider the following four bisections of $G$.
\[
\begin{array}{ccc}
B_1=(X_S \cup \{y\} ,Y_S \cup \{v,x_1,x_2\}), & &   
B_2=(X_S \cup \{v,y\} ,Y_S \cup \{x_1,x_2\}), \\    
B_3=(X_S \cup \{v,x_2\} ,Y_S \cup \{y,x_1\}), & &   
B_4=(X_S \cup \{v,x_1\} ,Y_S \cup \{y,x_2\}).  \\    
\end{array}
\]
If we pick a bisections from the four above randomly and uniformly, then we note that every edge in $E(G) \setminus E(G_S) = \{uv,yv,yx_1,yx_2\}$ will be included
with probability $3/4$. Therefore, by induction, the following holds (as $\theta \leq 3/4$).
\[              
mb(G) \geq \theta \cdot w(G_S) + \frac{3}{4} \left( w(G)-w(G_S) \right) \geq \theta \cdot w(G).
\]

In the four remaining cases, without loss of generality, we assume that $|Y_S|\ge |X_S|$ and $|Y_T|\ge |X_T|.$

\2



%


{\bf Subcase 2.C:} {\em $u\in X_S$ and $v\in X_T$.}

\2

{\bf Proof for Subcase 2.C:} If $G_T$ is a claw we are done by Subcase 2.B, so assume this is not the case.
Let $X = X_S \cup Y_T$ and $Y=Y_S \cup X_T$ and consider the bisection $B=(X,Y)$. 
As neither $G_S$ nor $G_T$ are claws (as $G_S$ is tight and by Subcase 2.B) we can use induction
on $G_S$ and $G_T$ which implies the following.

\[
mb(G) \geq \theta \cdot w(G_S) + \theta \cdot w(G_T) + w(uv) \geq \theta \cdot w(G).
\]

\2

{\bf Subcase 2.D:} {\em $u\in Y_S$.}

\2

{\bf Proof for Subcase 2.D:} Let $(X_T^+,Y_T^+)$ be a maximum weight bisection of $G[T \cup \{u\}]$. By Subcase 2.A we may assume that $G[T \cup \{u\}]$
is not a claw. Without loss of generality assume that $u \in Y_S \cap Y_T^+$ (otherwise we can swap the names of $X_T^+$ and $Y_T^+$).
Let $X = X_S \cup X_T^+$ and $Y=Y_S \cup Y_T^+$.
As $|Y_S|=|X_S|+1$ (as $u\in Y_S$, $|S|$ is odd, and $u$ belongs to the larger set of $\{X_S,Y_S\}$) and
$|Y|=|Y_S|+|Y_T^+|-1$ (as $u$ is in both $Y_S$ and $Y_T^+$) we note that $(X,Y)$ is a bisection of $G$. By induction the following holds (as
$|S|<|V(G)|$ and $|T|+1<|V(G)|$).

\[
mb(G) \geq \theta \cdot w(G_S) + \theta \cdot w(G[T \cup \{u\}]) \geq \theta \cdot w(G).
\]

\2

{\bf Subcase 2.E:} {\em $v\in Y_T$.}

\2

{\bf Proof for Subcase 2.E.}  Let $(X_S^+,Y_S^+)$ be a maximum weight bisection of $G[S \cup \{v\}]$. 
\AY{Note that as $|S \cup \{v\}|$ is even we have $|X_S^+|=|Y_S^+|$.} As $G_S$ is tight we note that $G[S \cup \{v\}]$
is not a claw. By Subcase 2.B we note that $G_T$ is not a claw or a single vertex.
Without loss of generality assume that $v \in Y_S^+ \cap Y_T$ (otherwise we can swap the names of $X_S^+$ and $Y_S^+$ ).
Let $X = X_S^+ \cup X_T$ and $Y=Y_S^+ \cup Y_T$. As \AY{$|Y_T| \geq |X_T|$} (as $v\in Y_T$ and $v$ belongs to the larger set of $\{X_T,Y_T\}$) and
$|Y|=|Y_S^+|+|Y_T|-1$ (as $v$ is in both $Y_S^+$ and $Y_T$) we note that $(X,Y)$ is a bisection of $G$. By induction the following holds
(as $|S|+1<|V(G)|$ and $|T|<|V(G)|$).

\[
mb(G) \geq \theta \cdot w(G[S \cup \{v\}]) + \theta \cdot w(G_T) \geq \theta \cdot w(G).
\]
This completes the proof.
\end{proof}


\begin{thebibliography}{111111}
		\bibitem {AGN1993}
		L. Anderson, D. Grant and N. Linial, Extremal $k$-colourable subgraphs, Ars Combin. 16 (1983) 259–270. 
		
		\bibitem{BL1986}
		J. Bondy and S. Locke, Largest bipartite subgraphs in triangle-free graphs with maximum degree three, J. Graph Theory, 10(4):477--505, 1986.
		
		\bibitem{BM2008}
		J. Bondy and U. Murty, Graph theory, New York: Springer, 2008.
		
		\bibitem{BS2002}
	
		B. Bollob\'as and A. Scott, Better bounds for Max Cut, Contemporary Combinatorics, Bolyai Society Mathematical Studies 10 (2002), 185--246
		
		\bibitem{BS2004}
		B. Bollob\'as and A. Scott, Judicious Partitions of Bounded-Degree Graphs, Journal of Graph Theory, 46(2004), 131--143.
		
		\bibitem{EW}
		P. Erd\H{o}s and R. J. Wilson, Note on the chromatic index of almost all graphs, Journal of Combinatorial Theory, Series B, 23(1977), 255--257,
		
		\bibitem{GY}
		G. Gutin and A. Yeo, Lower Bounds for Maximum Weighted Cut, SIAM Journal on Discrete Mathematics, 37 (2023), no. 2, 1142--1161. 
		
		\bibitem{HS}
		
		A. Hajnal and E. Szemeredi, Proof of a conjecture of P. Erd\H{o}s, in: Combinatorial Theory and Its Applications, II, Proc. Colloq., Balatonf\"ured, 1969, North-Holland, Amsterdam, 1970, 601--623.
		
		
		\bibitem{LLS}
		C. Lee, P. Loh and B. Sudakov, Bisections of graphs, Journal of Combinatorial Theory, Series B, 103(2013), 599--629.
		
		\bibitem{LZ1982}
		J. Lehel and Zs. Tuza, Triangle-free partial graphs and edge-covering theorems, Discrete Math. 30 (1982) 59–63.
		
		\bibitem{Locke1982}
		S. Locke, Maximum k-colorable subgraphs, J. Graph Theory 6 (1982) 123–132. 
		
		\bibitem{MM}
		D. Mattiolo and G. Mazzuoccolo, On 3-Bisections in Cubic and Subcubic Graphs, Graphs and Combinatorics 37 (2021), 743--746.
		
		\bibitem{Peterson}
	J. Petersen, Die Theorie der regul\"aaren graphs, Acta Math, 15(1891), 193--220.
	
	\bibitem{Seymour:BFC}
	P. D. Seymour, On multi-colourings of cubic graphs, and conjectures of Fulkerson and Tutte, Proc. London Math Soc., 38 (1979), 423--460.
	
\bibitem{SSTF} M. Stiebitz, D. Scheide, B. Toft and L. M. Favrholdt, Graph Edge Coloring: Vizing's Theorem and Goldberg's Conjecture, Wiley, 2012.

	
	\bibitem{Tutte}
	W. T. Tutte, The factorization of linear graphs, J. London Math. Soc., 22 (1947), 107--111.

	
	\bibitem{Vizing}
	V. G. Vizing, On an estimate of the chromatic class of a $p$-graph, Diskret. Analiz., 3(1964), 25--30.
	
	
	\end{thebibliography}
\end{document}